\documentclass[11pt,reqno]{amsart}
\usepackage{amsmath}
\usepackage{amsthm}
\usepackage{graphicx}
\usepackage{amsfonts}
\usepackage{amssymb}
\usepackage{hyperref}
\usepackage{bm}
\usepackage{color}
\usepackage[margin=1.1in]{geometry}
\usepackage{enumerate}
\newcommand{\indep}{\perp \!\!\! \perp}
\usepackage{mathrsfs}

\numberwithin{equation}{section}
\newtheorem{theorem}{Theorem}[section]
\newtheorem{lemma}[theorem]{Lemma}
\newtheorem{proposition}[theorem]{Proposition}

\theoremstyle{definition}

\newtheorem{definition}[theorem]{Definition}
\newtheorem{remark}[theorem]{Remark}

\newtheorem{innercustomthm}{Assumption}
\newenvironment{assumption}[1]
  {\renewcommand\theinnercustomthm{#1}\innercustomthm}
  {\endinnercustomthm}

\def\E{{\mathbb E}}
\def\R{{\mathbb R}}

\def\N{{\mathbb N}}
\def\FF{{\mathbb F}}
\def\GG{{\mathbb G}}
\def\HH{{\mathbb H}}
\def\PP{{\mathbb P}}

\def\QQ{{\mathbb Q}}
\def\P{{\mathcal P}}

\def\V{{\mathcal V}}
\def\M{{\mathcal M}}
\def\H{{\mathcal H}}
\def\X{{\mathcal X}}

\def\L{{\mathcal L}}
\def\G{{\mathcal G}}

\def\S{{\mathcal S}}

\def\W{{\mathcal W}}

\def\A{{\mathcal A}}

\def\F{{\mathcal F}}

\def\C{{\mathcal C}}

\def\CP{{C([0,T];\P(\R^d))}}
\def\CE{{C([0,T];E)}}

\title[Closed-loop convergence  for mean field games with common noise]{Closed-loop convergence  for mean field games with common noise}

\author{Daniel Lacker and Luc Le Flem}
\thanks{This work was partially supported by the Air Force Office of Scientific Research Grant FA9550-19-1-0291.}
\address{Department of Industrial Engineering \& Operations Research, Columbia University}
\email{daniel.lacker@columbia.edu, ll3240@columbia.edu}

\begin{document}

\begin{abstract}
This paper studies the convergence problem for mean field games with common noise. We define a suitable notion of weak mean field equilibria, which we prove captures all subsequential limit points, as $n\to\infty$, of closed-loop approximate equilibria from the corresponding $n$-player games. 
This extends to the common noise setting a recent result of the first author, while also simplifying a key step in the proof and allowing unbounded coefficients and non-i.i.d.\ initial conditions.
Conversely, we show that every weak mean field equilibrium arises as the limit of some sequence of approximate equilibria for the $n$-player games, as long as the latter are formulated over a broader class of closed-loop strategies which may depend on an additional common signal.
\end{abstract}

\maketitle

\section{Introduction}

We consider $n$-player stochastic differential games of mean field type. The state processes $\bm{X}^n = (X^{n,1}, \ldots, X^{n,n})$ are governed by the stochastic differential equation (SDE) system 
\begin{align}
dX^{n,i}_t &= b(t,X^{n,i}_t,\mu^n_t,\alpha^i(t,\bm{X}^n_t))dt + \sigma dW^i_t + \gamma dB_t, \quad\quad \mu^n_t = \frac{1}{n}\sum_{k=1}^n\delta_{X^{n,k}_t}. \label{def:intro_SDE}
\end{align}
Here $B, W^1,\ldots,W^n$ are independent Brownian motions and $X^{n,1}_0,\ldots,X^{n,n}_0$ are given.
We mostly omit specific assumptions in this introduction but highlight that $\sigma$ must be non-degenerate. As usual, we refer to $W^1,\ldots,W^n$ as \emph{idiosyncratic noises}, and $B$ is the \emph{common noise}.
The empirical measure process $\mu^n$ encapsulates the aggregate behavior of the $n$ players, and it is the large-$n$ behavior of $\mu^n$, in equilibrium, which will be the primary focus of this paper.

The control $\alpha^i = \alpha^i(t,\bm{X}^n_t)$ chosen by player $i$ can be any measurable function of the time and current state processes of all players, taking values in a compact space $A$. We focus for now on \emph{Markovian} controls and will discuss important generalizations in Section \ref{se:summary}. Each player $i=1, \ldots, n$ aims to maximize their personal reward
\begin{align*}
J_i^n(\alpha^1, \ldots, \alpha^n) = \E\left[\int_0^T f(t, X^{n,i}_t, \mu^n_t, \alpha^i(t,\bm{X}^n_t))dt + g(X^{n,i}_T, \mu^n_T)\right].
\end{align*}
For $\epsilon \ge 0$, we work with the usual notion of $\epsilon$-Nash equilibrium, which is defined as a vector of controls $(\alpha^1, \ldots, \alpha^n)$ such that, for any other control $\beta$ and any player $i\in\{1, \ldots n \}$,
\begin{align*}
J_i^n(\alpha^1, \ldots, \alpha^n) \geq J_i^n(\alpha^1, \ldots, \alpha^{i-1}, \beta, \alpha^{i+1}, \ldots, \alpha^{n}) - \epsilon.
\end{align*}
In other words, player $i$ can do no better than $\alpha^i$, with $\epsilon$-precision, if all other players maintain their choices $(\alpha^j)_{j \neq i}$.
In this paper, we mostly take for granted the existence of equilibria, focusing instead on their large-$n$ behavior when they do exist, but see \cite[Section I.2.1]{carmona-delarue-book} for a general discussion and some existence theorems for $n$-player stochastic differential games.

Mean field game (MFG) theory arose in the series of papers \cite{lasrylions, lasry2006jeux1,lasry2006jeux2, huang2006large, huang2007large} as a systematic framework for identifying \emph{continuum limits} of stochastic games of the above form, which are often simpler to analyze.
There is now a rich analytic and probabilistic literature on the theory, which has found widespread applications; see \cite{carmona-delarue-book} for a comprehensive recent overview. 
When there is no common noise ($\gamma=0$), the continuum MFG model can be formulated as a fixed point problem on the space $C([0,T];\P(\R^d))$ of measure flows.
When common noise is present, the natural fixed point problem is instead over \emph{stochastic} $B$-adapted measure flows.
A mean field equilibrium (MFE), loosely speaking, is defined as a $\P(\R^d)$-valued process $\mu=(\mu_t)_{t \in [0,T]}$, adapted to the filtration $\FF^B=(\F^B_t)_{t \in [0,T]}$ generated by the common noise $B$, such that if $(X^*_t)_{t \in [0,T]}$ denotes the optimal state process of the stochastic optimal control problem
\begin{align}
\begin{split}
\sup_{\alpha} \ & \E\left[\int_0^Tf(t,X_t,\mu_t,\alpha_t)dt + g(X_T,\mu_T)\right], \\
dX_t & = b(t,X_t,\mu_t,\alpha_t)dt + \sigma dW_t + \gamma dB_t,
\end{split} \label{intro:MFE}
\end{align}
then the consistency condition $\mu_t = \mathrm{Law}(X^*_t\,|\,\F^B_t)$ holds for all $t$. 

There are by now many applications of MFG models with common noise, particularly in economics under the name of \emph{heterogeneous agent models} (with \emph{aggregate shocks}); see \cite{achdou2014partial,achdou2014heterogeneous,ahn2018inequality} and references therein.
But the theoretical literature on MFGs with common noise remained relatively limited until recent years.
Typically, the above fixed point problem is recast in terms of a forward-backward stochastic partial differential equation (SPDE) or a forward-backward conditional McKean-Vlasov equation. 
These systems have been successfully analyzed under the Lasry-Lions monotonicity condition \cite{chassagneux2014probabilistic,cardaliaguet-delarue-lasry-lions,cardaliaguet2020first}, and a so-called \emph{weak monotonicity condition} \cite{ahuja2016wellposedness,ahuja2019forward}.
In both cases, monotonicity plays an important structural role in ensuring existence and uniqueness, without which one is guaranteed uniqueness only on a small time horizon.
The other main approach to construct equilibria for MFGs with common noise is based on probabilistic compactness arguments, initiated in \cite{carmona-delarue-lacker}, and see also \cite[Volume II, Chapter 3]{carmona-delarue-book}.
This approach to existence theorems was extended in several directions, including controlled diffusion coefficients \cite{barrasso2020controlled}, interactions involving the controls (also known as \emph{extended} MFGs) \cite{djete2020mean}, and models with absorption \cite{burzoni-campi}.
The idea is to analyze the fixed point problem \eqref{intro:MFE} by discretizing the common noise (for compactness purposes) and then taking weak limits. In general, one can only expect to find \emph{weak mean field equilibria} (defined in analogy with weak solutions of SDEs), in which $\mu$ is not necessarily $B$-measurable, and the consistency condition is instead $\mu_t=\mathrm{Law}(X^*_t\,|\,\F^{B,\mu}_t)$. In contrast, a \emph{strong} MFE is one in which $\mu$ is $B$-measurable.

The advantage of this compactness approach is that it works in very general settings, notably avoiding monotonicity assumptions. 
The advantage of the SPDE and McKean-Vlasov approaches (as well as the master equation, discussed later) is that they tend to provide \emph{strong} MFE, and much more information about them (regularity, etc.).
Some other techniques exist: Common noise models with a special structure of \emph{translation-invariance} permit a reduction to the case without common noise \cite{lacker-webster}.
See also \cite{dianetti2019submodular} for a construction of strong MFE for so-called \emph{submodular} models, via order-theoretic rather than topological methods.

\subsection{The convergence problem}
This paper is devoted to what has come to be known as \emph{the convergence problem}, which is to rigorously establish that the equilibria of the $n$-player game  converge to the MFG, thereby justifying the MFG as the correct limiting model. More precisely, suppose we are given for each $n$ an $\epsilon_n$-Nash equilibrium $(\alpha^{n,1},\ldots,\alpha^{n,n})$ for the $n$-player game, where $\epsilon_n\to 0$. Let $\mu^n$ denote the empirical measure process, as in \eqref{def:intro_SDE}.
The key question is if $\mu^n$ converges to the MFE in some sense.

The convergence problem has seen significant developments in recent years. It can be divided into two categories, corresponding to the choice of open-loop and closed-loop controls for the $n$-player game. In open-loop equilibria, each player  chooses a control as an adapted process on some given filtered probability space. In the closed-loop setting, which is the focus of this paper, controls are specified as functions of the states. Unlike in (one-player) stochastic control problems, the open-loop and closed-loop equilibria are typically distinct; see \cite[Volume I, Section 2.1.2]{carmona-delarue-book} for a careful explanation and \cite{carmona-fouque-sun} for an instructive, explicitly solvable example.

The convergence problem for open-loop MFGs has been well understood since \cite{lacker2016general} (see also \cite{fischer2017connection}), which proved a general result in the common noise setting using weak convergence and compactness techniques. Notably, no monotonicity or uniqueness of mean field equilibria are needed here. More recent progress appeared in \cite{djete2020mean}, treating mean field games of controls with common noise.
Similar techniques have resolved the open-loop convergence problem for other kinds of MFG models, such as discrete-time models \cite{acciaio2020cournot}, first-order (noiseless) models \cite{fischer2019asymptotic}, and  correlated equilibria \cite{campi2020correlated}, as well as for cooperative (mean field control) models in great generality \cite{lacker2017limit,djete2020mckean,djete2020extended}. A new approach was developed recently in \cite{lauriere2019backward,LauriereTangpi} based on propagation of chaos for mean field (F)BSDEs.

Closed-loop equilibria are often considered more realistic, with players able to react to each other, but their analysis is significantly more complicated than the open-loop case. It was shown in \cite{cardaliaguet-delarue-lasry-lions} how to solve the convergence problem if one has access to a sufficiently smooth solution of the master equation, an infinite-dimensional PDE system, by using it to constructing approximate solutions of the $n$-dimensional  PDE system associated with the Nash equilibrium of the $n$-player game. This powerful idea was extended in \cite{DelarueLackerRamananCLT,DelarueLackerRamananLDP} to derive a central limit theorem, large deviations, and non-asymptotic concentration bounds for $\mu^n$. This versatile approach has been adapted to models with finite state space \cite{cecchin2019convergence,bayraktar2017analysis}, a major player \cite{cardaliaguet2018remarks}, local couplings \cite{cardaliaguet2017convergence}, and graph-based interactions \cite{delarue-ERgraph}.
Proving well-posedness of the master equation, however, is a notoriously difficult task, especially when common noise is involved. 
There has been great progress in recent years, beginning with the groundbreaking work \cite{cardaliaguet-delarue-lasry-lions} which heavily exploited the Lasry-Lions monotonicity condition. Well-posedness of the master equation has since been shown in several settings, 
including major player models \cite{cardaliaguet2018remarks}, finite state space \cite{bayraktar2021finite,bertucci2019some}, lower regularity \cite{mou2020wellposedness}, degenerate idiosyncratic noise \cite{cardaliaguet2020first}, and the recent paper \cite{gangbo2021mean} using the alternative \emph{weak} (or \emph{displacement}) \emph{monotonicity}  concept originating in \cite{ahuja2016wellposedness}.

There is an important shortcoming of the master equation approach to the convergence problem, apart from the difficulty of constructing a solution: The  existence of a smooth enough classical solution of the master equation already implies its uniqueness; see \cite[Remark 4.5]{DelarueLackerRamananCLT}.
 This essentially limits the scope of the approach to situations in which the MFE is unique, which are well known to be atypical.
 These shortcomings were recently addressed by the first author \cite{Lacker_closedloop}, which adapts to the closed-loop setting the probabilistic compactness approach which was previously limited to the open-loop setting. This avoids any use of the master equation and, in the case of no common noise, solves the convergence problem for a much broader class of models than in \cite{cardaliaguet-delarue-lasry-lions}.
Notably, it covers not only Markovian equilibria but also path-dependent equilibria, though we refer to  \cite{Cardaliaguet2020counterexample} for a strong caveat regarding how general a result one can hope for: Essentially, if the $n$ players are able to observe each other's \emph{controls}, then they can punish deviations in the spirit of the Folk theorem, leading to a large set of $n$-player equilibria which are not approximable by the mean field limit; this mechanism is unavailable in our setting because of the non-degenerate idiosyncratic noise.

In another direction,  the recent work \cite{possamai2021non} introduces a quantitative stochastic approach to the closed-loop convergence problem for MFGs without common noise, based on propagation of chaos for BSDEs and the weak formulation of MFGs due to \cite{carmona-lacker}. Their assumptions are quite different from those required for the other approaches discussed above, though the framework is much closer in spirit to the master equation approach than the compactness approach. Notably, the results of \cite{possamai2021non} do not appear to cover many cases of non-unique MFE because of their key assumption 2.6(iii), which requires existence and uniqueness for a McKean-Vlasov BSDE that appears to be, in most cases, equivalent to the MFE itself.

Another interesting recent idea is to study the convergence problem via the \emph{set of approximate equilibrium values} \cite{iseri-zhang}. The paper \cite{iseri-zhang} embraces non-uniqueness and clarifies some connections between $n$-player and mean field information structures. In the continuous-time case \cite[Section 7]{iseri-zhang}, however, the analysis is limited so far to models without common noise and with symmetric equilibrium controls which are Lipschitz functions of the state process and empirical measure, with Lipschitz constant crucially uniform in the number of players $n$.  

\subsection{Summary of main results} \label{se:summary}

The first main result of the paper, Theorem \ref{th:mainlimit}, resolves the closed-loop convergence problem in a setting including common noise, non-i.i.d.\ initial conditions, and unbounded coefficients.
If the initial distribution $\mu^n_0$ converges, then we show that the sequence of empirical measure processes $(\mu^n_\cdot)_{n \in \N}$ associated to any Markovian approximate equilibria is tight, and every limit in distribution is a MFE in a suitable weak sense discussed below.
This gives the first convergence result for closed-loop equilibria with common noise which does not resort to the master equation, and it is thus not limited in scope to the unique MFE regime.
The arguments extend the probabilistic compactness approach of \cite{Lacker_closedloop}, but with non-trivial adaptations to cover the case of common noise. We also make a notable simplification in the key step of checking the optimality property, taking advantage of the recent superposition principle for SPDEs obtained in \cite{superposition_theorem}.
Along the way, we identify a natural MFE concept, which we call \emph{weak semi-Markov} MFE as it adapts the notion used in \cite{Lacker_closedloop} to the common noise setting. This notion involves controls of the form $\alpha(t,X_t,\mu,B)$, which can depend on the current value of the state process as well as the histories of the measure flow and common noise. It is simpler than the weak equilibrium concept proposed in \cite{carmona-delarue-lacker} (see also \cite[Volume II, Chapter 3]{carmona-delarue-book}), which requires a delicate \emph{compatibility} condition and a somewhat unnatural augmentation of the measure flow to include the conditional law of not only $X$ but rather $(X,\alpha,W)$. However, the latter notion of equilibrium turns out to be equivalent to our own, in a certain distributional sense; see Section \ref{se:connectionsWeakMFE}.

Our convergence theorem in fact holds for a broader class of $n$-player equilibria, beyond the Markovian case. Players' controls may depend only on the \emph{histories} of all state processes (as in \cite{Lacker_closedloop}), or more generally, also on an additional \emph{common signal} process 
 $S=(S_t)_{t \in [0,T]}$, seen by all players, which takes values in some Polish space. We call the latter \emph{$S$-closed-loop} controls. The common signal must be independent of the idiosyncratic noises and the initial states, but it may depend on (or even equal) the common noise.
We show in Proposition \ref{pr:np-eq-inclusion0} that any Markovian or history-dependent equilibrium is also an $S$-closed-loop equilibrium.

Our second main result, Theorem \ref{th:converselimit-strongMFE}, is a converse to the first: \emph{Every weak MFE} arises as the limit of a sequence of $S$-closed-loop $\epsilon_n$-Nash equilibria for the $n$-player games, for some $\epsilon_n\to 0$. This is known from \cite{lacker2016general} in the case where the $n$-player games are posed in the open-loop sense. In the closed-loop case, it was shown in \cite[Theorem 2.12]{Lacker_closedloop} that every \emph{strong MFE} arises as a limit of \emph{Markovian} approximate equilibria, but nothing was settled about weak MFE in general. It remains an open question, even in the case without common noise, whether or not all \emph{weak} MFE can be achieved as limits of \emph{Markovian} approximate equilibria of the $n$-player game. However, our Theorem \ref{th:converselimit-strongMFE} shows how to recover all weak MFE via a natural extension of the $n$-player closed-loop equilibrium concept.

We finally mention the important concurrent work \cite{Djete2021}, which appeared on arXiv shortly after the first version of the present paper. It uses similar compactness arguments to solve the closed-loop convergence problem for models with common noise and even covers extended MFGs for the first time. It introduces a notion of \emph{measure–valued MFG equilibrium}, which in the non-extended case is equivalent to our own notion of weak equilibrium \cite[Remark 2.8]{Djete2021}. Another major achievement of \cite{Djete2021} is to prove a converse result similar to our Theorem \ref{th:converselimit-strongMFE} but without resorting to the additional randomness built into our notion of $S$-closed-loop controls, instead constructing $n$-player equilibria over closed-loop path-dependent controls.
Several  assumptions made in \cite{Djete2021} are not needed in our paper, such as bounded data, Lipschitz drift, and most notably a nondegenerate common noise $\gamma > 0$.

\subsection{Outline of the paper}

The paper is organized as follows. Section \ref{se:mainresults} introduces the setup, equilibrium concepts, and the main results of the paper, including a summary of some key ideas of the proofs in Section \ref{se:ideas_of_main_proof}. Section \ref{se:estimates} summarizes some simple moment bounds that will recur throughout the paper. Section \ref{se:mainlimitproof} is devoted to the main line of the proof of the convergence theorem.  Section \ref{se:constructing_nash_eq} proves the converse result, showing that all weak MFE arise as limits of a class $n$-player equilibria. Finally, Section \ref{se:connectionsWeakMFE} connects our notion of MFE with those of previous work on common noise models.

\section{Main results}\label{se:mainresults}

We start this section by summarizing some notations and the main assumptions that will be in force throughout the paper. Sections \ref{se:nplayergame} and \ref{se:MFE} then introduce the $n$-player game and mean field game equilibrium concepts, respectively.   The main convergence result appears in Section \ref{se:mainlimit_sec}, and the converse in Section \ref{se:converse}. Lastly, Section \ref{se:ideas_of_main_proof} discusses key new ideas of the proof.

\subsection{Notation}
If $X$ is a random variable on a probability space $(\Omega, \F,\PP)$, we write $\L(X)$ or $\PP \circ X^{-1}$ for the law of $X$. We write $X \sim \lambda$ to mean that $\L(X) = \lambda$. For another random variable $Y$ defined on the same space, we denote by $\L(X|Y)$ a version of the conditional law of $X$ given $Y$, which is well defined up to a.s.\ equality if $X$ takes values in a Polish space \cite[Theorem 6.3]{kallenberg-foundations}.  We will also use the usual integral shorthand $\langle m,\varphi \rangle := \int \varphi \, dm$.
 
If $Z=(Z_t)_{t\in [0,T]}$ is an $E$-valued stochastic process defined on some space $\Omega$, we write $\FF^Z = (\F_t^Z)$ for the filtration it generates on $\Omega$. For multiple processes $Z^1,\ldots,Z^n$ defined on the same space, taking values in some Polish spaces, we similarly write $\FF^{Z^1,\ldots,Z^n}=(\F^{Z^1,\ldots,Z^n}_t)_{t \in [0,T]}$ for the filtration generated by the process $(Z^1_t,\ldots,Z^n_t)_{t \in [0,T]}$. For an $\R^d$-valued random variable $\xi$, we write $\FF^{\xi,Z^1,\ldots,Z^n}$ for the filtration $(\sigma(\xi) \vee \F^{Z^1,\ldots,Z^n}_t)_{t \in [0,T]}$.

For a given complete separable metric space $(E, \rho)$ and $T>0$, we will always equip the space $\CE$ of continuous functions from $[0, T]$ to $E$ with the sup-metric $(x, y) \mapsto \sup_{t\in[0, T]} \rho(x_t, y_t)$ for $x,y \in \CE$. If $E'$ is another Polish space, we identify the spaces $C([0,T];E \times E') \cong C([0,T];E) \times C([0,T];E')$. A function $h:[0, T] \times \CE \mapsto E'$ is called \emph{progressively measurable} if it is Borel measurable and satisfies $h(t, x) = h(t, y)$ for all $t\in[0, T]$ and $x,y \in \CE$ such that $x_s = y_s$ for all $s\in[0,t]$.
For the following generalization of progressively measurable functions, we adopt a more specialized terminology, as in \cite{Lacker_closedloop}:

\begin{definition} \label{def:semiMarkov}
A function $h: [0,T] \times \R^d \times \CE \to E'$ is said to be \emph{semi-Markov} if it is Borel measurable and satisfies $h(t,z,x)= h(t,z,y)$ for all $t\in[0, T]$, $z \in \R^d$, and $x,y \in \CE$ such that $x_s = y_s$ for all $s\in[0,t]$.
\end{definition}

Let $\P(E)$ be the space of Borel probability measure on $E$. For $r \geq 0$, let $\P^r(E)$ denote the set of  $m\in\P(E)$ that satisfy $\int_E \rho(x, x_0)^r m(dx) < \infty$ for some $x_0 \in E$, with the convention $\P^0(E):=\P(E)$. We endow $\P(E)$  with the weak convergence topology and its corresponding Borel $\sigma$-field. Define the $r$-Wasserstein distance between $\mu,\nu \in \P^r(E)$ as usual by
\begin{align*}
\W_r(\mu, \nu) := \left(\inf_\pi \int_{E \times E} \rho(x,y)^r \,\pi(dx, dy) \right)^{1/(1 \vee r)}, \quad r > 0,
\end{align*}
where the infimum is over all couplings $\pi$ of $(\mu,\nu)$. For the case $r=0$, define
\begin{align}
\W_0(\mu, \nu) := \inf_\pi \int_{E \times E} 1 \wedge \rho(x,y) \,\pi(dx, dy). \label{def:W0}
\end{align}
As shown in  \cite[Theorem 7.3 \& 7.12]{Villani2003}, $(\P^r(E),\W_r)$ defines a complete separable metric space.
For $r=0$, the metric $\W_0$ on $\P(E)$ is compatible with weak convergence.

For the special case $E=\R^d$, we will denote $\C^d := C([0, T]; \R^d)$ equipped with the supremum norm.
As usual, for $k \in \N$, $C^\infty_c(\R^k)$ denotes the set of smooth functions $\R^k\to\R$ of compact support.

\subsection{Main assumptions}
We fix a time horizon $T>0$ and a dimension $d\in \N$. We are given a control space $A$ and two exponents $p,p' \ge 0$. We are also given an initial distribution $\lambda \in \P(\R^d)$ and the functions
\begin{align*}
(b,f) &: [0,T] \times \R^d \times \P^p(\R^d) \times A \rightarrow \R^d \times \R, \\ 
g &: \R^d \times \P^p(\R^d) \rightarrow \R,
\end{align*}
along with the non-singular matrix $\sigma \in \R^{d\times d}$ and the general matrix $\gamma \in \R^{d\times d}$. We make the following assumptions throughout the paper:
\begin{assumption}{\textbf{A}} \label{assumption:A}
{\ }
\begin{enumerate}
\item[(A.1)] $A$ is a compact metric space.
\item[(A.2)] The functions $b, f, g$ of $(t, x, m, a)$ are measurable in $t$ and continuous in $(x, m, a)$.
\item[(A.3)] There exists $c_1 > 0$ such that, for all $(t, x, m, a) \in [0, T ] \times \R^d \times \mathcal{P}^p(\mathbb{R}^d) \times A$,
\begin{align}
|b(t, x, m, a)|  \leq c_1 \left(1 + |x| 1_{\{p > 0\}}  + 1_{\{p > 0\}}\left(\int_{\R^d} |z|^p\,m(dz)\right)^{1/(1 \vee p)} \right).
\label{def:Ass3}
\end{align}
\item[(A.4)] There exist $c_2 > 0$ such that, for all $(t, x, m, a) \in [0, T ] \times \R^d \times \P^p(\R^d) \times A$, 
\begin{align}
     |f(t, x, m, a)| + |g(x, m)| \leq c_2 \Big(1 + |x|^{p} + \int_{\R^d} |z|^{p} \, m(dz) \Big).
    \label{def:Ass4}
\end{align}
\item[(A.4)] The distribution $\lambda$ belongs to $\P^{p'}(\R^d)$, and either $0< p\leq p  \vee 2 < p'$ or $p'=p=0$.
\item[(A.5)] For each $(t,x,m) \in [0,T] \times \R^d \times \P^p(\R^d)$, the following set is convex:
\[
K(t,x,m) = \left\{(b(t,x,m,a),z) : a \in A, \ z \le f(t,x,m,a)\right\} \subset \R^d \times \R.
\]
\end{enumerate}
\end{assumption}
The most unfortunate of these assumptions is the compactness of $A$, which we were unable to generalize; see Remark \ref{re:possible_extensions} for discussion.
The convexity assumption (\ref{assumption:A}.5) holds, for instance, if $b=b(t,x,m,a)$ is affine in $a$ and $f=f(t,x,m,a)$ concave in $a$, for each $(t,x,m)$. As a particular case, this assumption also covers the \emph{relaxed control setup}, where $A = \P(\widetilde A)$ for some compact metric space $\widetilde A$, and $(b,f)$ are of the form $(b,f)(t,x,m,a) = \int_{\widetilde A}(b,f)(t,x,m,\widetilde a)\,a(d\widetilde a)$. In particular, Assumption (\ref{assumption:A}.5) can be dropped at the price of stating our main results in terms of relaxed controls; see Remark \ref{re:working-with-relaxed-controls} for details.
Assumption (\ref{assumption:A}.4) is designed to cover either the case where all coefficients are bounded ($p'=p=0$), for which no integrability is needed for $\lambda$, or the case where the coefficients can be unbounded, for which $\lambda$ must admit a finite moment of order $p' > p \vee 2$.
Our methods should adapt with little change to cover the case where $\sigma=\sigma(t,x)$ is non-constant, as long as it is bounded, Lipschitz, and uniformly non-degenerate. However, it does not seem that our methods could accommodate control or interactions in $\sigma$, because of our repeated use of Girsanov's theorem. It is even more difficult to handle a non-constant $\gamma$, as our proofs exploit a common change of variables $X\to X-\gamma B$.

\subsection{The n-player games} \label{se:nplayergame}

Let $n\in\N$.  Assume the $n$-player game is defined on a filtered probability space $(\Omega^n,\F^n,\FF^n=(\F^n_t)_{t \in [0,T]},\PP^n)$, supporting independent $d$-dimensional $\FF^n$-Brownian motions $B,W^1,\ldots,W^n$ and $\F^n_0$-measurable $\R^d$-valued random variables $X^{n,1}_0,\ldots,X^{n,n}_0$.

We introduce three classes of controls for the $n$-player game.
\begin{itemize}
\item The set of \emph{Markovian controls} $\M_n$ is the set of measurable functions $\alpha : [0, T] \times (\R^d)^n \to A$. These controls are functions of time and the current values of the $n$ state processes.
\item The set of \emph{closed-loop controls} $\A_n$ is the set of progressively measurable functions $\alpha : [0, T] \times (\C^d)^n \to A$. These controls are functions of time and the trajectories of the $n$ state processes.
\item 
A \emph{signal process} is a continuous $\FF^n$-adapted process  $S=(S_t)_{t \in [0,T]}$ taking values in a Polish space $\mathcal{S}$, such that $S$ is independent of $(X^{n,i}_0,W^i)_{i=1}^n$, and $B$ is adapted to $\FF^S$. The set of \emph{$S$-closed-loop controls} $\A_n(S)$ is the set of progressively measurable functions $\alpha : [0, T] \times (\C^d)^n \times C([0,T];\mathcal{S}) \to A$. These controls are functions of time, the trajectories of the $n$ state processes, and the trajectory of the signal.
\end{itemize}

It is clear, accepting a minor abuse of notation, that we have the inclusions $\M_n \subset \A_n \subset \A_n(S)$, and we thus focus on $\A_n(S)$ in the subsequent definitions.
For  $\alpha^1,\ldots,\alpha^n \in \A_n(S)$, we define the associated state process $\bm{X}^n = (X^{n,1}, \ldots, X^{n,n})$ as the unique in law weak solution of the system of SDEs
\begin{align}
dX^{n,i}_t &= b(t,X^{n,i}_t,\mu^n_t,\alpha^i(t,\bm{X}^n,S))dt + \sigma dW^i_t + \gamma dB_t , \quad\quad \mu^n_t = \frac{1}{n}\sum_{k=1}^n\delta_{X^{n,k}_t}, 
\label{def:n_player_games}
\end{align}
starting from the given initial states $X^{n,1}_0,\ldots,X^{n,n}_0$, and we call the process $\mu^n=(\mu^n_t)_{t \in [0,T]}$ the \emph{associated measure flow}, viewed as a random element of $C([0,T];\P^p(\R^d))$. Were it not for  the presence of $S$ in \eqref{def:n_player_games}, the existence and uniqueness in law would be a well known consequence of Girsanov's theorem. There is some subtlety in our setting of random coefficients. 

\begin{definition}\label{def:nplayerSDE}
A \emph{weak solution} $\bm{X}^n = (X^{n,1}, \ldots, X^{n,n})$ of the SDE \eqref{def:n_player_games} corresponding to $(\alpha^1,\ldots,\alpha^n)$ is an extension of the filtered probability space $(\Omega^n,\F^n,\FF^n,\PP^n)$ such that the following hold:
\begin{itemize}
\item $B,W^1,\ldots,W^n$ remain independent Brownian motions in the enlarged filtration, and $\bm{X}^n=(X^{n,1},\ldots,X^{n,n})$ is a continuous adapted $(\R^d)^n$-valued process.
\item The SDE \eqref{def:n_player_games} holds.
\item The process $(W^1,\ldots,W^n)$ is a $\FF^{\bm{X}^n,S}$-Brownian motion under the conditional measure $\PP^n(\cdot\,|\,S=s)$, for $\L(S)$-a.e.\ $s \in C([0,T];\mathcal{S})$.
\end{itemize}
\end{definition} 
\begin{lemma} \label{le:nplayerSDE-lemma}
Suppose Assumption \ref{assumption:A} holds, and let $\alpha^1,\ldots,\alpha^n \in \A_n(S)$. There exists a weak solution to \eqref{def:n_player_games} in the sense of Definition \ref{def:nplayerSDE}, and it is unique in law in the sense that any two weak solutions induce the same joint law of $(\bm{X}^n,S)$.
\end{lemma}

Lemma \ref{le:nplayerSDE-lemma} follows from Lemmas \ref{ap:le:SDE-exist} and \ref{ap:le:SDE-uniq} in Appendix \ref{ap:SDE}, and see also Remark \ref{re:ap:SDEs-nplayer}.
The last bullet point in Definition \ref{def:nplayerSDE} is the important one, constraining the relationship between the signal and state processes, which allows us to assert uniqueness. 

\begin{remark}
Regarding the third property in Definition \ref{def:nplayerSDE}, notice that the SDE \eqref{def:n_player_games} and non-degeneracy of $\sigma$ imply that $\bm{W}^n = (W^1,\ldots,W^n)$  is adapted to $\FF^{\bm{X}^n,S,B}$, and the assumed adaptedness of $B$ with respect to $\FF^S$ then implies that $\bm{W}^n$  is adapted to $\FF^{\bm{X}^n,S}$.
Note also that the third property does not follow automatically from the independence of $S$ and $\bm{W}^n$; it is equivalent to $\bm{W}^n_t-\bm{W}^n_s$ being conditionally independent of $\F^{\bm{X}^n,S}_s$ given $\F^S_T$, for $0 \le s < t \le T$.
\end{remark}

To simplify notation, we use the same symbols $(\Omega^n,\F^n,\FF^n,\PP^n)$ to denote any extension of the probability space arising from Lemma \ref{le:nplayerSDE-lemma}.
We also avoid indexing the Brownian motions by $n$, but this should cause no confusion.
We similarly suppress the dependence  of  the signal process $S$ on $n$, although $S$ and its domain $\mathcal{S}$ could be different for each $n$.  Lastly, we write simply $\E$ rather than $\E^{\PP^n}$ for expectation.

When using $S$-closed-loop controls, players have access to the \emph{common signal process} $S$. Of course, the notion of equilibrium depends on the choice of signal process. This may simply be $B$ itself, or it may contain additional randomness, but this additional randomness must crucially be independent of $(W^1,\ldots,W^n)$. Otherwise, this would include open-loop controls,  and Proposition \ref{pr:np-eq-inclusion0} would fail.
Note that the \emph{filtration} of $S$ is important in the above definitions, but there is no real need for a topological structure or continuity of the process $S$; however, it makes the formulations simpler and covers our main case of interest in Theorem \ref{th:converselimit-strongMFE}.

We need not assume i.i.d.\ initial positions, but rather just the following:

\begin{assumption}{\textbf{B}} \label{assumption:B}
The initial conditions satisfy the moment bound
\begin{align}
    \sup_{n\in \N}\frac{1}{n} \sum_{k=1}^n\E|X_0^{n,k}|^{p'} < \infty.
    \label{ass:assumption_C} 
\end{align}
Moreover, the sequence of empirical measures $(\mu_0^n)_{n \in \N}$ converges to $\lambda$ in probability in $\P^p(\R^d)$.
\end{assumption}

Player $i$ seeks to maximize an objective functional, given for $\alpha^1,\ldots,\alpha^n \in \A_n(S)$ by
\begin{align*}
J^n_i(\alpha^1,\ldots,\alpha^n) := \E\left[\int_0^Tf(t,X^{n,i}_t,\mu^n_t,\alpha^i(t,\bm{X}^n, S))dt + g(X^{n,i}_T,\mu^n_T)\right].
\end{align*}
Given our three notions of control, the relevant Nash equilibrium concepts are:

\begin{definition} \label{def:nplayer-eq}
Let $\epsilon \ge 0$. 
A \emph{Markovian $\epsilon$-Nash equilibrium} is a tuple $(\alpha^1,\ldots,\alpha^n) \in \M_n^n$ such that
\begin{align*}
J^n_i(\alpha^1,\ldots,\alpha^n) \ge \sup_{\beta \in \M_n}J^n_i(\alpha^1,\ldots,\alpha^{i-1},\beta,\alpha^{i+1},\ldots,\alpha^n) - \epsilon, \quad \text{for } i=1,\ldots,n.
\end{align*}
A \emph{closed-loop $\epsilon$-Nash equilibrium} is a tuple $(\alpha^1,\ldots,\alpha^n) \in \A_n^n$ such that
\begin{align*}
J^n_i(\alpha^1,\ldots,\alpha^n) \ge \sup_{\beta \in \A_n}J^n_i(\alpha^1,\ldots,\alpha^{i-1},\beta,\alpha^{i+1},\ldots,\alpha^n) - \epsilon, \quad \text{for }  i=1,\ldots,n.
\end{align*}
An \emph{$S$-closed-loop $\epsilon$-Nash equilibrium} is a tuple $(\alpha^1,\ldots,\alpha^n) \in \A_n^n(S) := (\A_n(S))^n$ such that
\begin{align*}
J^n_i(\alpha^1,\ldots,\alpha^n) \ge \sup_{\beta \in \A_n(S)}J^n_i(\alpha^1,\ldots,\alpha^{i-1},\beta,\alpha^{i+1},\ldots,\alpha^n) - \epsilon, \quad \text{for }  i=1,\ldots,n.
\end{align*}
\end{definition}

The following Proposition relates these three equilibrium concepts, extending \cite[Proposition 2.2]{Lacker_closedloop}. The proof is given in Appendix \ref{ap:relations_Nash_eq}.

\begin{proposition} \label{pr:np-eq-inclusion0}
Suppose Assumptions \ref{assumption:A} and \ref{assumption:B} hold. Let $\epsilon \ge 0$.  
\begin{enumerate}[(a)]
\item Any Markovian $\epsilon$-Nash equilibrium is also a closed-loop $\epsilon$-Nash equilibrium.
\item Any closed-loop $\epsilon$-Nash equilibrium is also an $S$-closed-loop $\epsilon$-Nash equilibrium, for any signal process $S$. 
\end{enumerate} 
\end{proposition}

The rough intuition here is that if every other player is using a certain information set (e.g., Markovian), then no single player can do any better by incorporating a broader information set, at least as long as this broader information does not directly affect the players' objective functions.
The opposite implications do not hold (e.g., there can be closed-loop equilibria which are not Markovian), simply because the inclusion of control sets fails, i.e., $\A_n \not\subset \M_n$; the simplest counterexamples arise by taking $f\equiv 0$ and $g \equiv 0$, so that \emph{any} controls form an equilibrium.

Proposition \ref{pr:np-eq-inclusion0} remains valid under more general assumptions; for instance, the action space $A$ does not need to be compact, as long as suitable growth or integrability assumptions are imposed.
We could also define S-closed-loop controls without requiring that $B$ is $\FF^S$-adapted and Proposition  \ref{pr:np-eq-inclusion0} would stay valid. However, this restriction is needed for our main theorem, so we prefer to include it in the definition.

\subsection{Mean field equilibrium, existence and uniqueness} \label{se:MFE}

We are now ready to define the mean field equilibrium (MFE).  The definition that follows should more precisely be referred to as \emph{MFE corresponding to the model inputs $(b,\sigma,\gamma,f,g,\gamma,A,\lambda,p,d)$}, but we will leave this dependence as implicit.
The claimed existence and uniqueness for the SDEs \eqref{def:weakMFE-SDE} and \eqref{def:SDE-semimarkov-strong} below are justified in Appendix \ref{ap:SDE}, by Lemmas \ref{ap:le:SDE-exist}, \ref{ap:le:SDE-uniq}, and \ref{ap:le:SDE-uniq-strong}, as well as Remark \ref{re:ap:SDEs-MFG}.

\begin{definition} \label{def:weakMFE}
We say that the tuple $(\Omega,\F,\FF,\PP ,W , B, \alpha^*,\mu, X^*)$ is a \emph{weak semi-Markov mean field equilibrium} (weak MFE, for short) if the following hold:
\begin{enumerate}[(1)]
\item $(\Omega,\F,\FF,\PP)$ is a filtered probability space.
\item $\mu = (\mu_t)_{t \in [0,T]}$ is a continuous $\P^p(\R^d)$-valued $\FF$-adapted process, and $W$ and $B$ are independent $\FF$-Brownian motions.
\item $X^*_0$, $W$, and $(\mu, B)$ are independent.
\item $\alpha^* : [0,T] \times \R^d \times C([0, T]; \P^p(\R^d) \times \R^d)  \rightarrow A$ is semi-Markov, in the sense of Definition \ref{def:semiMarkov}.
\item $X^*$ is a continuous $\R^d$-valued $\FF$-adapted process, with $\L(X^*_0)=\lambda$, which is the unique strong solution of the SDE
\begin{align}
dX^*_t = b(t,X^*_t,\mu_t,\alpha^*(t,X^*_t,\mu, B))dt + \sigma dW_t + \gamma dB_t, \label{def:weakMFE-SDE}
\end{align}
started from $X^*_0$. Moreover, the consistency condition  $\mu_t = \L(X^*_t \, | \, \F_t^{\mu,B})$ holds a.s.\ for each $t \in [0,T]$.
\item For every alternative semi-Markov $\alpha : [0,T] \times \R^d \times C([0, T]; \P^p(\R^d) \times \R^d)  \rightarrow A$, we have
\begin{align}
\begin{split}
\E&\left[\int_0^T f(t,X^*_t,\mu_t,\alpha^*(t,X^*_t,\mu, B)) dt + g(X^*_T,\mu_T)\right] \\
	&\ge \E\left[\int_0^T f(t, X_t,\mu_t, \alpha(t, X_t,\mu, B)) dt + g( X_T, \mu_T)\right],
\end{split}
	\label{def:optimality_part_generalized_MFE}
\end{align}
where $X$ is the unique strong solution of
\begin{align}
d X_t = b(t, X_t,\mu_t,\alpha(t, X_t,\mu, B))dt + \sigma dW_t + \gamma dB_t, \qquad X_0 = X_0^*. 
\label{def:SDE-semimarkov-strong}
\end{align}
\end{enumerate}
If $\mu$ is adapted to the completion of $\FF^B$, then we call $(\Omega,\F,\FF,\PP ,W , B, \alpha^*,\mu, X^*)$ a \emph{strong} MFE. 
\end{definition}

The \emph{semi-Markov} control in the Definition \ref{def:weakMFE} depends on $(t,X_t)$ plus the \emph{history of the common noise}.
This common noise can be decomposed into two parts: the exogeneous $B$, and the endogenous $\mu$. The latter need not be $B$-measurable in a weak MFE, and it then represents a common signal to which all players correlate their behavior.

We can add the restriction that $\FF$ equals the completion of $\FF^{X_0^*,W,B,\mu}$ in Definition \ref{def:weakMFE}(1), essentially without loss of generality, because we have strong solutions in \eqref{def:weakMFE-SDE} and \eqref{def:SDE-semimarkov-strong}. Indeed, it is straightforward to check that $(\Omega,\F,\FF,\PP ,W , B, \alpha^*,\mu, X^*)$ is a weak MFE if and only if $(\Omega,\F^{X_0^*,W,B,\mu}_T,\FF^{X_0^*,W,B,\mu},\PP ,W , B, \alpha^*,\mu, X^*)$ is a weak MFE.

\begin{remark}
Definition \ref{def:weakMFE} admits a natural reformulation purely in terms of stochastic Fokker-Planck equations (which is equivalent by the superposition principle quoted in Theorem \ref{th:superposition} below). Indeed, we could omit $(W,X^*)$ from the definition, replace the SDEs \eqref{def:weakMFE-SDE} and \eqref{def:SDE-semimarkov-strong} respectively by the SPDEs (in weak form) satisfied by $\mu$ and by $\nu_t=\L(X_t\,|\,\F^{\mu,B}_t)$, and rewrite \eqref{def:optimality_part_generalized_MFE} in terms of these random measures. The proofs of the following results rely mainly on these SPDEs and thus would not change much. We favor the form given in Definition \ref{def:weakMFE}, which we find to be more transparent.
\end{remark}

\begin{remark}
For a weak MFE, it holds that $\F^{X^*}_t$ is conditionally independent of $\F^{\mu,B}_T$ given $\F^{\mu,B}_t$, and thus $\mu_t=\L(X^*_t\,|\,\F^{\mu,B}_T)$ a.s. Indeed, this is because $(X^*_0,W)$ is independent of $(\mu,B)$, and because $X^*$ is a strong solution and is thus adapted to the completion of $\FF^{X_0^*,W,B,\mu}$.
\end{remark}

We will explain in Section \ref{se:connectionsWeakMFE} that the notion of weak MFE is equivalent, in a distributional sense, to the notion of \emph{weak MFG solution} introduced in \cite{carmona-delarue-lacker,lacker2016general,carmona-delarue-book}.
This will let us quickly transfer existence and uniqueness theorems from \cite{carmona-delarue-lacker} to our new notion of equilibrium, such as:

\begin{theorem} \label{th:existence}
In addition to Assumption \ref{assumption:A}, suppose also that $b$ is uniformly Lipschitz in $x$. That is, there exists $C < \infty$ such that
\begin{align*}
|b(t,x,m,a)-b(t,x',m,a)| \le C|x-x'|, \qquad \forall x,x' \in \R^d, \ (t,m,a) \in [0,T] \times \P^p(\R^d) \times A.
\end{align*}
Then there exists a weak MFE. Suppose in  addition that the following conditions hold:
\begin{itemize}
\item The drift $b(t, x, m, a) = b(t,x,a)$ does not depend on the measure $m$.
\item The running cost is separable, i.e., $f(t, x, m, a) = f_1(t, x, m) + f_2(t, x, a)$ for some measurable functions $f_1$ and $f_2$.
\item $A$ is a compact convex subset of $\R^k$ for some $k$, 
\item For each $(t,m)$, $(x,a) \mapsto b(t,x,a)$ is affine, $x\mapsto g(x,m)$ is concave, and $(x,a) \mapsto f_2(t,x,a)$ is strictly concave.
\item The Lasry-Lions monotonicity condition holds: For all $t \in [0,T]$ and $m,\widetilde{m} \in \P^p(\R^d)$,
\begin{align*}
\int_{\R^d} \left[g(x, m) - g(x, \widetilde{m})\right](m-\widetilde{m})(dx) &\le 0, \\
\int_{\R^d} \left[f_1(t,x, m) - f_1(t,x, \widetilde{m})\right](m-\widetilde{m})(dx) &\le 0.
\end{align*}
\end{itemize}
Then there exists a unique in law weak MFE, in the sense that $\L(\mu^1,B^1)=\L(\mu^2,B^2)$ for any two weak MFE $(\Omega^i,\F^i,\FF^i,\PP^i ,W^i , B^i, \alpha^i,\mu^i, X^i)$, and it is in fact a strong MFE.
\end{theorem}

\subsection{The convergence theorem}\label{se:mainlimit_sec}

We may now state the first main result of the paper: 

\begin{theorem} \label{th:mainlimit}
Suppose Assumptions \ref{assumption:A} and \ref{assumption:B} hold.
Fix a sequence $\epsilon_n \ge 0$ with $\epsilon_n \rightarrow 0$. For each $n$, suppose $\bm{\alpha}^n=(\alpha^{n,1},\ldots,\alpha^{n,n}) \in \A_n^n(S)$ is an $S$-closed-loop $\epsilon_n$-Nash equilibrium for some signal process $S$. Let $\mu^n$ denote the associated measure flow. Then the sequence $\{\L(\mu^n, B) : n \in \N\} \subset \P(C([0, T]; \P^p(\R^d) \times \R^d))$ is precompact, and every limit is of the form $\L(\mu,B)$ for some weak MFE $(\Omega,\F,\FF,\PP ,W , B, \alpha^*,\mu, X^*)$.
\end{theorem}

The proof is given in Section \ref{se:mainlimitproof}. By Proposition \ref{pr:np-eq-inclusion0}, the conclusions of Theorem \ref{th:mainlimit} hold if $\bm{\alpha}^n$ is assumed instead to be a closed-loop or Markovian $\epsilon_n$-Nash equilibrium.

\begin{remark}
As in \cite[Remark 2.8]{Lacker_closedloop}, Theorem \ref{th:mainlimit} does not need the full strength of the Nash equilibrium concept.
It remains true, with the same proof, if we assume merely that
\[
\frac{1}{n}\sum_{k=1}^{n}J^{n}_k(\alpha^{n,1},\ldots,\alpha^{n,n}) + \epsilon_n \ge \sup_{\beta \in \A_n(S)}\frac{1}{n}\sum_{k=1}^n J^{n}_k(\alpha^{n,1},\ldots,\alpha^{n,k-1},\beta,\alpha^{n,k+1},\ldots,\alpha^{n,n}).
\]
\end{remark}

\begin{remark} \label{re:fullyMarkov1}
If one is interested in characterizing limits of \emph{Markovian} equilibria of the $n$-player games, it would be arguably more natural to do so in terms of a MFE concept in which controls depend on $(t, X_t, \mu_t)$ instead of $(t,X_t,\mu,B)$.
This at first seemed feasible, in light of the recent \emph{mimicking theorem} in \cite{superposition_theorem} for McKean-Vlasov processes with common noise. In particular, given a weak semi-Markov MFE $(\Omega,\F,\FF,\PP ,W , B, \alpha^*,\mu, X^*)$, we could apply  \cite[Corollary 1.6]{superposition_theorem} to obtain a tuple $(\widehat\Omega,\widehat\F,\widehat\FF,\widehat\PP ,\widehat W , \widehat B, \widehat\alpha^*,\widehat\mu, \widehat X^*)$, satisfying properties (1--5) of Definition \ref{def:weakMFE}, except with a ``fully Markov" control $\widehat\alpha^*(t,\widehat X^*_t,\widehat\mu_t)$, and also satisfying the ``mimicking" property $\L(\widehat\mu_t,\widehat X_t)=\L(\mu_t,X_t)$ for all $t \in [0,T]$. But the joint law of $(\mu,B)$ is not preserved, and it is thus unclear that this resulting tuple satisfies any meaningful optimality property. More concretely, we were unable to show that it is optimal against all ``fully Markov" deviations. It is similarly unclear how to obtain a meaningful notion of equilibrium in which controls are semi-Markov but depending only on (the history of) $\mu$, not $B$. The essential issue seems to be that the law of the state process $X$ is determined by the control $\alpha$ as well as the joint law $\L((\mu_t,B_t)_{t \in [0,T]})$; on the other hand, the control and either $\L((\mu_t)_{t \in [0,T]})$ or $(\L(\mu_t))_{t \in [0,T]}$ do not together determine $\L(X)$.
Interestingly, this  issue does not arise in cooperative (i.e., mean field control) problems, where the mimicking theorem does yield optimal ``fully Markov" controls \cite[Section 8]{superposition_theorem}.
\end{remark}

\begin{remark}\label{re:working-with-relaxed-controls}
If we drop the convexity assumption (\ref{assumption:A}.5), a version of Theorem \ref{th:mainlimit} still holds in which every limit  is of the form $\L(\mu,B)$ for some weak \emph{relaxed} MFE. A weak \emph{relaxed} MFE $(\Omega,\F,\FF,\PP ,W , B, \alpha^*,\mu, X^*)$ is defined as in Definition \ref{def:weakMFE}, except that $\alpha^*$ takes value in $\P(A)$ instead of $A$. That is $\alpha^*:[0, T] \times \R^d \times C([0, T]; \P^p(\R^d) \times \R^d) \rightarrow \P(A)$, and the SDE \eqref{def:weakMFE-SDE} becomes
\begin{align*}
dX^*_t = \int_A b(t,X^*_t,\mu_t,a)\alpha^*(t,X^*_t,\mu, B)(da) dt + \sigma dW_t + \gamma dB_t, 
\end{align*}
whereas the optimality condition \eqref{def:optimality_part_generalized_MFE} becomes 
\begin{align*}
\begin{split}
\E&\left[\int_0^T \int_A f(t,X^*_t,\mu_t,a)\alpha^*(t,X^*_t,\mu, B)(da) dt + g(X^*_T,\mu_T)\right] \\
	&\ge \E\left[\int_0^T\int_A f(t, X_t,\mu_t, a)\alpha(t, X_t,\mu, B)(da) dt + g( X_T, \mu_T)\right],
\end{split}
\end{align*}
for every alternative semi-Markov $\alpha : [0,T] \times \R^d \times C([0, T]; \P^p(\R^d) \times \R^d) \rightarrow \P(A)$, where $X$ is the unique strong solution of
\begin{align*}
d X_t = \int_A b(t, X_t,\mu_t,a)\alpha(t, X_t,\mu, B)(da) dt + \sigma dW_t + \gamma dB_t, \qquad X_0 = X_0^*. 
\end{align*}
Under Assumption (\ref{assumption:A}.5) it can be shown that these definitions coincide: given a weak relaxed MFE as above, there exists $\hat\alpha^*$ with values in $A$ such that $(\Omega,\F,\FF,\PP ,W , B, \hat\alpha^*,\mu, X^*)$ is a weak MFE. We refer to \cite[Section 3]{Lacker_closedloop} for more on relaxed equilibria and how relaxed and \emph{strict} equilibria are related without the presence of a common noise. We prefer to avoid including a similar section to our paper, which would only be a straightforward adaptation in the case of a common noise.  
\end{remark}

\subsection{A converse to the main limit theorem}  \label{se:converse}

Our second main result is a converse to Theorem \ref{th:mainlimit}, showing that a weak MFE can be used to construct $\epsilon_n$-equilibria for the $n$-player games with $\epsilon_n \to 0$. 
As in \cite{Lacker_closedloop}, we adopt in this section an additional assumption which will allow us to apply a recent propagation of chaos results from \cite{lacker2018strong}. 

\begin{assumption}{\textbf{C}} \label{assumption:C}
The drift $b$ is uniformly Lipschitz with respect to total variation, meaning that there exists $c > 0$ such that, for each $(t,x,a) \in [0,T] \times \R^d \times A$ and $m,m' \in \P(\R^d)$, we have
\begin{align*}
\frac{1}{c}|b(t,x,m,a) - b(t,x,m',a)| \le  \|m-m'\|_{\mathrm{TV}} := \sup_h \int_{\R^d} h\,d(m-m'),
\end{align*}
where the supremum is over all measurable functions $h : \R^d \rightarrow [-1,1]$.
\end{assumption}

\begin{theorem} \label{th:converselimit-strongMFE}
Suppose Assumptions \ref{assumption:A} and \ref{assumption:C} hold, with $p=0$ so that $(b,f,g)$ are bounded. Suppose the initial states $X^{n,i}_0$ are i.i.d.\ with law $\lambda$. Let $(\Omega,\F,\FF,\PP ,W , B, \alpha^*,\mu, X^*)$ be a weak MFE. Let $S$ denote the  process $(\mu, B)$ which takes values in $\mathcal{S}=\P(\R^d) \times \R^d$.
Define $\alpha^{n,i} \in \A_n(S)$ for $n \ge i \ge 1$ by
\begin{align*}
\alpha^{n,i}(t,\bm x,s) &= \alpha^*(t,x^i_t,s),
\end{align*}
for $t \in [0,T]$, $\bm x=(x^1,\ldots,x^n) \in (\C^d)^n$, and $s \in C([0, T]; \P(\R^d) \times \R^d)$. Then for each $n$, $(\alpha^{n,1},\ldots,\alpha^{n,n})$ is an $S$-closed-loop $\epsilon_n$-Nash equilibrium, where $\epsilon_n \geq 0$ with $\epsilon_n \to 0$, and the associated measure flow $\mu^n$ satisfies $\L(\mu^n, B)\to \L(\mu,B)$ in $\P(C([0, T]; \P(\R^d) \times \R^d))$.
\end{theorem}

The proof can be found in Section \ref{se:constructing_nash_eq}.
Note if the given MFE is \emph{strong} instead of \emph{weak}, then $\mu$ is $\FF^B$-adapted, and the signal process may be taken to be $S=B$.
It is similar in spirit but significantly more involved than that of \cite[Theorem 3.10]{Lacker_closedloop}, which treated the case without common noise, and only \emph{strong} MFE.
In the case without common noise, an analogue of Theorem \ref{th:converselimit-strongMFE} is valid, with the signal process $S=\mu$.

It is not clear if Theorem \ref{th:converselimit-strongMFE} remains true if the \emph{S-closed-loop} controls are instead required be \emph{Markovian}. See \cite[Sections 2.4 and 7]{Lacker_closedloop} for additional discussion of the difficulties.
The recent paper \cite{Djete2021}, which appeared on arXiv soon after the first draft of the present paper, shows under somewhat different assumptions that the $n$-player approximate equilibria in Theorem \ref{th:converselimit-strongMFE} can be constructed as \emph{closed-loop} equilibria (i.e., with path-dependent controls,  belonging to the set $\A_n$ defined in the second bullet point of Section \ref{se:nplayergame}).

If one restricts attention further and only considers the \emph{true} equilibria (rather than approximate equilibria) for the $n$-player games, then one cannot even expect to obtain all of the \emph{strong} MFE as limit points; see the recent case studies  \cite{cecchin2018convergence,delarue2020selection,nutz2018convergence}.

The analogue of Theorem \ref{th:converselimit-strongMFE} in the open-loop common noise setting was shown in \cite[Theorem 2.11]{lacker2016general}, \cite[Volume II, Theorem 6.14]{carmona-delarue-book}, and more recently \cite[Section 5]{burzoni-campi}. In the closed-loop setting without common noise, this was shown for \emph{strong} MFE in \cite{Lacker_closedloop}, and see also \cite{converse_limit} for an adaptation to models with absorptions. The closest result to ours is \cite[Volume II, Theorem 6.15]{carmona-delarue-book}, which gives a result in the closed-loop common noise case, involving what they call \emph{generalized closed loop} or \emph{semi-closed loop} controls for the $n$-player games, which are essentially equivalent to our $S$-closed-loop controls. Their result involves a different and more complicated notion of weak MFE, closely related to that of \cite{carmona-delarue-lacker}, which is equivalent in a certain sense to our notion of weak MFE as explained in Section \ref{se:connectionsWeakMFE}. Neither the assumptions of our result nor those of \cite[Volume II, Theorem 6.15]{carmona-delarue-book} includes the other, but we highlight that the latter is based on entirely different arguments, relying on Lipschitz and smoothness assumptions and the construction of a decoupling field for an FBSDE system governing the MFG.

As mentioned before, the additional assumptions  in Theorem \ref{th:converselimit-strongMFE} are needed mainly in order to apply the strong propagation of chaos result of \cite{lacker2018strong}, which requires a bounded drift satisfying the TV-Lipschitz condition. Other results on propagation of chaos could be used here, for instance if we knew $\alpha^*$ to be continuous in the spatial variable, but we prefer to avoid imposing assumptions on the MFE control. The result of \cite{lacker2018strong} also requires i.i.d.\ initial positions (as opposed to merely Assumption \ref{assumption:B}), which we also use for technical reasons in the proof of Theorem \ref{th:converselimit-strongMFE} (namely to justify applying \cite[Lemma 2.1]{beiglbock2018denseness}). The boundedness of $f$ and $g$ in Theorem \ref{th:converselimit-strongMFE} is made for simplicity and could easily be relaxed. 

Note also that the total variation distance dominates the Wasserstein distance $\W_0$ defined in \eqref{def:W0}, which makes the Lipschitz assumption less restrictive than a $\W_0$-Lipschitz assumption. The total variation distance is not directly comparable to $\W_p$ for $p>0$, though.

\subsection{Ideas of the proof of the main result}\label{se:ideas_of_main_proof}

This section discusses some key ideas of the proof of Theorem \ref{th:mainlimit}. Several steps are similar to the case without common noise, detailed in \cite[Section 2.6]{Lacker_closedloop}, and we may similarly split the proof conceptually into three steps: Tightness, limiting dynamics, and optimality. The first two steps proceed similarly to \cite{Lacker_closedloop}, but with a \emph{stochastic} PDE replacing the Fokker-Planck PDE in \cite[Section 2.6.1]{Lacker_closedloop}, and with some additional care required to handle unbounded coefficients. So let us suppose  for the rest of the section that these first two steps are resolved, i.e., that we have already shown $(\mu^n,B)$ to be tight, and that for every subsequential limit $(\mu,B)$ can be  realized as part of a tuple  satisfying properties (1--5) of Definition \ref{def:weakMFE}, in particular for some semi-Markov control $\alpha^*$. We fix one such subsequential limit, and understand that all limits appearing below are taken along this same subsequence.

The third step, showing the optimality property (6) of Definition \ref{def:weakMFE}, is the most difficult, and this is where new ideas are needed. 
We begin as in \cite[Section 2.6.2]{Lacker_closedloop}.
Fix an arbitrary alternative semi-Markov control $\beta : [0,T] \times \R^d \times C([0,T]; \P^p(\R^d) \times \R^d) \to \R^d$. The idea is to try to transfer the assumed $n$-player Nash property to the limit, to show that $\alpha^*$ is superior to $\beta$. To do this, we allow each player in the $n$-player game to try out the control $\beta$ as a deviation, one at a time. 
Let $\bm{Y}^{n,k}=(Y^{n,k,1},\ldots,Y^{n,k,n})$ be the state processes associated with the controls $(\alpha^{n,1}, \ldots, \alpha^{n,k-1}, \beta, \alpha^{n,k+1}, \ldots, \alpha^{n,n})$, for $k=1, \ldots, n$. That is,
\begin{align*}
dY^{n,k,k}_t &=  b (t,Y^{n,k,k}_t,\mu^{n,k}_t, \beta(t,Y^{n,k,k}_t,\mu^{n,k}, B))dt + \sigma dW^{k}_t + \gamma dB_t, \\
dY^{n,k,i}_t &= b(t,Y^{n,k,i}_t,\mu^{n,k}_t,\alpha^{n,i}(t,\bm{Y}^{n,k}, S))dt + \sigma dW^{i}_t + \gamma dB_t, \quad i \neq k, \\
\mu^{n,k}_t &= \frac{1}{n}\sum_{j=1}^n\delta_{Y_t^{n,k,j}} \qquad \bm{Y}^{n, k}_0 = \bm{X}^n_0.
\end{align*}
Apply the $\epsilon_n$-Nash property to each player, and average over the players to get
\begin{align}
\begin{split}
 \frac{1}{n} \sum_{k=1}^n & \E\left[\int_0^T f(t, X_t^{n,k}, \mu_t^{n}, \alpha^{n,k}(t, \bm{X}^{n}, S))\,dt + g(X_T^{n,k}, \mu_T^{n}) \right] \\
    & \ge \frac{1}{n} \sum_{k=1}^n  \E\left[\int_0^T f(t, Y_t^{n,k,k}, \mu_t^{n,k}, \beta(t, Y_t^{n,k,k}, \mu^{n,k}, B))\,dt + g(Y_T^{n,k,k}, \mu_T^{n,k}) \right]  - \epsilon_n.
\end{split} \label{def:ideas_of_proof_Nash_eq}
\end{align}
Note that this is allowed because $B$ is $S$-adapted. The left-hand side can be expressed as the expectation of an integral of the extended empirical measure $\frac{1}{n}\sum_{k=1}^n \delta_{(X^{n,k},\alpha^{n,k})}$, where $\alpha^{n,k}$ are viewed in a suitable space of relaxed controls. It is then not difficult, using the results for the omitted ``limiting dynamics" step, to show that the limsup of the left-hand side is no more than
\begin{align*}
\E\left[\int_0^T f(t, X_t^*, \mu_t, \alpha^*(t, X_t^*, \mu, B))\,dt + g(X_T^*, \mu_T)\right],
\end{align*}
which is exactly the left-hand side of \eqref{def:optimality_part_generalized_MFE}.
The most delicate step is to show that the right-hand of \eqref{def:ideas_of_proof_Nash_eq} converges to the right-hand side of \eqref{def:optimality_part_generalized_MFE}. Note first that the former can be rewritten as
\begin{align}
\frac{1}{n} \sum_{k=1}^n \E\left[ \Gamma(Y^{n,k,k},\mu^{n,k},B)\right], \label{idea:LHS1}
\end{align}
for a suitable functional $\Gamma$, which is continuous if we assume $\beta$ to be continuous, which can be justified by an approximation argument. 

To study the convergence of \eqref{idea:LHS1}, we begin with the same change of measure argument as in \cite[Section 2.6.2]{Lacker_closedloop}. Recall that we write $(\Omega^n,\F^n,\FF^n,\PP^n)$ for the filtered probability space supporting the $n$-player game. Define $\zeta^{n,k}$ as the unique solution of
\begin{align*}
d\zeta^{n,k}_t = \zeta_t^{n,k} \sigma^{-1} \Big[b(t, X_t^{n,k}, \mu_t^{n}, \beta(t, X_t^{n,k}, \mu^n, B)) - b(t, X_t^{n,k}, \mu_t^{n},\alpha^{n,k}(t, \boldsymbol{X}^n, S)) \Big]\cdot dW_t^k,
\end{align*}
with $\zeta^{n,k}_0=1$. By Girsanov's theorem, the law of $(\bm{Y}^{n,k},B)$ under $\PP^n$ is equal to the law of $(\bm{X}^n,B)$ under the measure $\zeta^{n,k}_T\,d\PP^n$. Hence, \eqref{idea:LHS1} becomes
\begin{align}
\frac{1}{n} \sum_{k=1}^n \E\left[ \Gamma(Y^{n,k,k},\mu^{n,k},B)\right] = \frac{1}{n} \sum_{k=1}^n \E\left[\zeta^{n,k}_T \Gamma(X^{n,k},\mu^{n},B)\right]. \label{idea:LHS2}
\end{align}
We express this in terms of the empirical measure $\widetilde\mu^n_t := \frac{1}{n}\sum_{k=1}^n \delta_{(X^{n,k}_t,\zeta^{n,k}_t)}$ of a $(d+1)$-dimensional particle system. Using Fubini and the martingale property of $\zeta^{n,k}$, \eqref{idea:LHS2} becomes
\begin{align*}
\E\left[ \int_0^T \int_{\R^{d+1}} y f(t,x,\mu^n_t,\beta(t,x,\mu^n,B))\,\widetilde\mu^n_t(dx,dy)\,dt + \int_{\R^{d+1}} yg(x,\mu^n_T)\,\widetilde\mu^n_T(dx,dy)\right].
\end{align*}

In the case without common noise, the very involved approach of \cite[Proposition 5.6]{Lacker_closedloop} is to describe the support of all subsequential limit points of a path-space analogue of $\widetilde\mu^n$, augmented to include the idiosyncratic Brownian motions, in terms of a certain martingale problem. We take a different approach here, better-suited to the case with common noise. The idea is as follows. By applying It\^o's formula, we identify the dynamics of $\langle \widetilde\mu^n_t,\varphi\rangle$, for smooth test functions $\varphi$ on $\R^{d+1}$. Passing to the limit, we find that any limit point $\widetilde\mu$ satisfies a certain (degenerate) Fokker-Planck SPDE on $\R^{d+1}$, in the weak sense, which we do not know to be well-posed. However, applying test functions of the form $(x,y) \mapsto y\varphi(x)$, we can show that the measure flow $\nu_t(dx) = \int_\R y\,\widetilde\mu_t(dx,dy)$ satisfies a Fokker-Planck SPDE which does turn out to be well-posed:
\begin{align*}
d\langle\nu_t,\varphi\rangle &= \left\langle \nu_t, \nabla\varphi^\top b(t,\cdot,\mu_t,\beta(t,\cdot,\mu,B)) + \frac12\mathrm{tr}[(\sigma\sigma^\top + \gamma\gamma^\top)\nabla^2\varphi]\right\rangle\,dt +\langle \nu_t,\nabla\varphi^\top\rangle \gamma  dB_t,
\end{align*}
for $\varphi \in C^\infty_c(\R^d)$.
This is exactly the SPDE associated with the conditional laws of the SDE
\begin{align*}
dX_t = b(t,X_t,\mu_t,\beta(t,X_t,\mu,B))dt + \sigma dW_t + \gamma dB_t,
\end{align*}
and the rigorous passage from the Fokker-Planck SPDE to this (well-posed) SDE takes advantage of the recent superposition principle of \cite{superposition_theorem} (recalled in Theorem \ref{th:superposition} below).

This argument of passing from $\widetilde\mu_t$ (whose first marginal is $\mu_t$) to $\nu_t$ can be seen as implementing Girsanov's theorem at the level of the Fokker-Planck equation. The essence of the idea is captured by the following formal argument. Suppose that $b_i:\R^d \mapsto \R^d$  are some given nice drifts, for $i=1,2$.
 Suppose $(\overline{m}_t)_{t \in [0,T]} \in C([0,T];\P(\R^d \times \R_+))$ satisfies $\int_{\R^d \times \R} y\,\overline{m}_0(dx,dy)=1$ and the degenerate Fokker-Planck equation
\begin{align*}
\frac{d}{dt}\langle \overline{m}_t,\varphi\rangle = \langle \overline{m}_t, \overline{L}\varphi \rangle, \ \ \varphi \in C^\infty_c(\R^{d+1}),
\end{align*}
where $\overline{L}$ acts on a function $\varphi(x,y)$ on $\R^d \times \R$ via
\begin{align*}
\overline{L}\varphi(x,y) = \ &\nabla_x\varphi(x,y) \cdot b_1(x) + \tfrac12\Delta_x\varphi(x,y) +  y \partial_y\nabla_x \varphi(x,y) \cdot (b_2(x)-b_1(x)) \\
	&+\tfrac12y^2 \partial_{yy}\varphi(x,y)|b_2(x)-b_1(x)|^2.
\end{align*}
Note that $\int_{\R^d \times \R} y\,\overline{m}_t(dx,dy)=1$ for all $t > 0$ since it is assumed true at $t=0$.
On the one hand, if $m^1_t \in \P(\R^d)$ denotes the first marginal of $\overline{m}_t$, then choosing test functions independent of $y$ shows that $m^1$ satisfies the Fokker-Planck equation
\begin{align*}
\frac{d}{dt}\langle m^1_t,\varphi \rangle = \left\langle m^1_t, \nabla\varphi \cdot b_1 + \tfrac12\Delta \varphi\right\rangle, \ \ \varphi \in C^\infty_c(\R^d).
\end{align*}
On the other hand, defining $m^2_t \in \P(\R^d)$ via $\langle m^2_t,\varphi\rangle = \int_{\R^d \times \R} y \varphi(x)\,\overline{m}_t(dx,dy)$, we may apply the equation for $\overline{m}$ with test functions of the form $(x,y)\mapsto y\varphi(x)$ to find that $m^2$ satisfies 
\begin{equation*}
\frac{d}{dt}\langle m^2_t,\varphi \rangle = \left\langle m^2_t, \nabla\varphi \cdot b_2 + \tfrac12\Delta \varphi\right\rangle, \ \ \varphi \in C^\infty_c(\R^d).
\end{equation*}
The use of this argument is that it connects $m^1$ and $m^2$, viewing both as descending from the measure flow $\overline{m}$. To see more clearly the connection with Girsanov's theorem, note formally $\overline{m}_t = \L(X_t, \xi_t)$ for $t \in [0, T]$, where $(X,\xi)$ solves the SDE system 
\begin{align*}
d X_t & = b^1(X_t)dt + d W_t, \qquad 
d \xi_t  = \xi_t\left(b^2(X_t) - b^1(X_t)\right) \cdot d W_t,
\end{align*}
and we recognize that $\xi$ is a Dol\'eans-Dade exponential.

\begin{remark} \label{re:possible_extensions}
While we allow unbounded coefficients, we still require compact action space $A$. This is unfortunately restrictive, but it appears to be quite difficult to overcome. Natural assumptions to try to work with include linear growth of $b(t,x,m,a)$ with respect to $a$, and polynomial growth and coercivity assumptions on the objective functions $f$ and $g$. But to obtain tightness, some uniform integrability is required of $\zeta^{n,k}$, as well as $\zeta^{n,k}_T g(X_T^{n,k},\mu^n_T)$ and the analogous $f$ term. It does not appear that $\zeta^{n,k}$ has a moment of order $p > 1$ which is bounded in $n$. Even a more forgiving $L\log L$ (or entropy) bound would lead to the quantity
\begin{align*}
\frac{1}{n}\sum_{k=1}^n \E\int_0^T |\alpha^{n,k}(t, \bm{Y}^{n,k})|^2dt.
\end{align*}
This a puzzling quantity. There seems to be no way to use a coercivity assumption to obtain an estimate for it, because $\alpha^{n,k}$ does not appear in the dynamics of $\bm{Y}^{n,k}$. We did not find any truncation arguments  capable of bypassing this difficulty. This difficulty would be resolved if we imposed the constraint that controls must grow at most linearly in $(x,m)$, with the constant of linear growth being bounded uniformly in $n$, but such an a priori constraint would be unnatural and difficult to verify.
\end{remark}

\section{Estimates} \label{se:estimates}

This section states two simple moment estimates that we will use throughout our paper. The first deals with the $n$-player system \eqref{def:n_player_games}, whereas the second deals with the solution $X^*$ of the McKean-Vlasov equation from Definition \ref{def:weakMFE}.
In the following, define the truncated supremum norm $\|x\|_t = \sup_{s \in [0,t]}|x_s|$ for $x \in \C^d = C([0,T];\R^d)$ and $t \in [0,T]$. Assumptions \ref{assumption:A} and \ref{assumption:B} hold throughout this section.

\begin{lemma}\label{le:inequalities}
There exists a constant $C < \infty$, depending only on $p$, $p'$, $T$ and the constant $c_1$ of assumption (\ref{assumption:A}.3), such that the following holds:
Let $n \in \N$. For some signal process $S$, let $\alpha^1,\ldots,\alpha^n \in \A_n(S)$, and let $\bm X^n=(X^{n,1},\ldots,X^{n,n})$ be the corresponding solution of the SDE system \eqref{def:n_player_games}.
For $k=1,\ldots,n$ and $t\in [0, T]$, we have
\begin{align}
\|X^{n,k}\|_t^{p'} \leq C\Bigg(1 +&   |X_0^{n,k}|^{p'} + \frac{1}{n} \sum_{j=1}^n |X_0^{n,j}|^{p'} + \|W^{k}\|_t^{p'} + \frac{1}{n} \sum_{j=1}^n \|W^{j}\|_t^{p'} +\|B\|_t^{p'}\Bigg), \  \ a.s.  \nonumber \\
    \frac{1}{n} \sum_{k=1}^n \E\|X^{n,k}\|_T^{p'} & \leq C\Bigg(1 + \frac{1}{n} \sum_{k=1}^n\E |X_0^{n,k}|^{p'}\Bigg),     \label{pf:nice_estimates_1} \\
    \E\|X^{n,k}\|_T^{p'} & \leq C\Bigg(1 + \E|X_0^{n,k}|^{p'} + \frac{1}{n} \sum_{j=1}^n\E |X_0^{n,j}|^{p'}\Bigg).
    \label{pf:nice_estimates}
\end{align}
In particular, $\frac{1}{n} \sum_{k=1}^n \E\|X^{n,k}\|_T^{p'}$ is bounded uniformly with respect to $n$, and the choice of controls $\alpha^1,\ldots,\alpha^n \in \A_n(S)$.
\end{lemma}

The proof of Lemma \ref{le:inequalities}  is left to the reader as a straightforward application of the linear growth assumption (\ref{assumption:A}.3) and Gronwall's inequality. The final claim of uniform boundedness of $\frac{1}{n} \sum_{k=1}^n \E\|X^{n,k}\|_T^{p'}$ follows from \eqref{pf:nice_estimates_1} and the moment bound of Assumption \ref{assumption:B}. 

\begin{lemma}\label{le:moment_property_definition_generalilzed}
There exists a constant $C < \infty$, depending only on $p$, $p'$, $T$ and the constant $c_1$ of assumption (\ref{assumption:A}.3), such that the following holds:
For any tuple $(\Omega,\F,\FF,\PP,W,B,\alpha^*,\mu,X^*)$ satisfying the conditions (1--5) of Definition \ref{def:weakMFE}, we have
\begin{align}
    \E \left[\sup_{t\in[0, T]} \int_{\R^d} |x|^{p'} \,\mu_t(dx)\right] \leq C\Bigg(1 + \int_{\R^d} |x|^{p'} \,\lambda (dx)\Bigg) < \infty.
    \label{pf:moment_property_mu_from_def}
\end{align}
Similarly, the process $X$ from \eqref{def:SDE-semimarkov-strong} satisfies
\begin{align}
\E\|X\|_T^{p'} \le  C\Bigg(1 + \int_{\R^d} |x|^{p'} \,\lambda (dx)\Bigg) < \infty. \label{moment:Xmfe}
\end{align}
\end{lemma}
\begin{proof}
For $p=p'=0$ there is nothing to prove, so we focus on the other case, $0 < p \le p \vee 2 < p'$. Let $Y^{(q)}_t = \sup_{s \in [0,t]}\int_{\R^d} |x|^{q}\,\mu_s(dx)$ for $q > 0$. Note that $Y^{(p)}_t < \infty$ a.s.\ since $\mu$ is a continuous $\P^p(\R^d)$-valued process.
In the following, $C < \infty$ denotes a constant which may vary from line to line but depends only on $p$, $p'$, $T$ and the constant $c_1$ of assumption (\ref{assumption:A}.3).
Using the SDE solved by $X^*$ and assumption (\ref{assumption:A}.3), we have for all $t \in [0,T]$ and $q \in [p \vee 1,p']$,
\begin{align*}
   |X^*_t|^{q} & \leq C\Big(1 + |X^*_0|^{q} +\int_0^t |X^*_s|^{q} ds +   \int_0^t (Y^{(p)}_s)^{q/(p\vee 1)} \, ds + \|W\|_t^{q} + \|B\|_t^{q} \Big)
\end{align*}
Using Gronwall's inequality, we have 
\begin{align}
|X^*_t|^{q} \leq C\Big(1 + |X^*_0|^{q} + \int_0^t (Y^{(p)}_s)^{q/(p\vee 1)}  \, ds + \|W\|_t^{q} + \|B\|_t^{q} \Big) \label{pf:Gronwall1-0}
\end{align}
Take the conditional expectation given $\F^{\mu,B}_t$ on both sides of this inequality. Since $X^*_0$ and $W$ are independent of $\F^{\mu,B}_T$, we get
\begin{align*}
\int_{\R^d}|x|^{q}\,\mu_t(dx) = \E[|X_t^*|^{q} \,|\, \F^{\mu,B}_t ]\leq C\Big(1 + \E|X^*_0|^q + \int_0^t (Y^{(p)}_s)^{q/(p\vee 1)}  \, ds + \E\|W\|_t^{q} + \|B\|_t^{q} \Big)
\end{align*}
Apply this with $q=p \vee 1$, using $|x|^{p} \leq C(1 + |x|^{p\vee 1})$ and Gronwall's inequality, to get
\begin{align}
Y_t^{(p)} \leq C(1 + Y^{(p\vee 1)}_t) \leq C\Big(1 + \E|X^*_0|^{p\vee 1} + \E\|W\|_t^{p\vee 1}+ \|B\|_t^{p\vee 1} \Big).
 \label{pf:Gronwall1-1}
\end{align}
Apply this in \eqref{pf:Gronwall1-0} and take expectations to complete the proof. 
\end{proof}

\section{Proof of the main limit theorem} \label{se:mainlimitproof}

In this section, we prove Theorem \ref{th:mainlimit}. Assumptions \ref{assumption:A} and \ref{assumption:B} hold throughout this section. We consider a sequence of $\epsilon_n$-Nash equilibria $(\alpha^{n,1},\ldots,\alpha^{n,n}) \in \A_n^n(S)$ and denote by $\bm{X}^n=(X^{n,1},\ldots,X^{n,n})$ the corresponding state process in the $n$-player game, which is the weak solution of the SDE 
\begin{align}
dX^{n,i}_t &= b(t,X^{n,i}_t,\mu^n_t,\alpha^{n,i}(t,\bm{X}^n, S))dt + \sigma dW^{ i}_t + \gamma dB_t , \quad\quad \mu^n_t = \frac{1}{n}\sum_{k=1}^n\delta_{X^{n,k}_t}. \label{def:stateprocesses-limitsection}
\end{align}
Here $W^{1}, \cdots, W^{n}, B$ are independent Brownian motions and $S$ is a signal process (possibly different for each $n\in \N$ as explained in Section \ref{se:nplayergame}). The initial states $X^{n,1}_0,\ldots,X^{n,n}_0$ have a general distribution but, by Assumption \ref{assumption:B},  have a uniformly bounded $p'$-th moment and satisfy $\mu_0^n \to \lambda$ in probability in $\P^{p}(\R^d)$.

The proof is structured in three parts. The first and simplest part studies the tightness of the sequence $(\mu^n, B)$. The second step identifies the dynamics of its limit points; that is we show that the limit points satisfies the points (1--5) of Definition \ref{def:weakMFE} for some control $\alpha^*$. The third and most difficult step is to prove the optimality of this control, i.e., (6) of Definition \ref{def:weakMFE}.

\subsection{Relaxed controls and the extended empirical measure} \label{se:relaxedcontrols}
In this short section, we define the set $\V$ of \emph{relaxed controls} as the set of Borel measures $q$ on $[0, T] \times A$ with first marginal equal to Lebesgue measure. We equip $\V$ with the topology of weak convergence, which makes it a compact metric space, as $A$ is compact itself. We equip $\V$ with an arbitrary compatible metric, and the particular choice will be immaterial.

Each $q \in \V$ may be identified via disintegration with a (uniquely defined up to a.e.\ equality) measurable function $[0,T] \ni t \mapsto q_t \in \P(A)$, where $q(dt,da)=dtq_t(da)$. The natural filtration on $\V$ at time $t$ is the $\sigma$-field generated by the functions $\V \ni q \mapsto q(C)$, where $C \subset [0,t] \times A$ is Borel; in this filtration, there is a predictable version of the map $[0,T] \times \V \ni (t,q) \mapsto q_t \in \P(A)$, and this allows us to freely identify $\V$-valued random variables and measurable $\P(A)$-valued processes (see \cite[Lemma 3.2]{lacker2015mean}). We will abuse notation somewhat by writing $q_t$ for the value of this map at $(t,q)$, so that $q_t$ is well defined for \emph{every} $t$, not just \emph{almost every} $t$.

The control $\alpha^{n,k}$ of each player $k$ in the $n$-player game induces a $\V$-valued random variable,
\begin{align*}
\Lambda^{n,k}(dt,da) := dt\delta_{\alpha^{n,k}(t,\bm X^n,S)}(da).
\end{align*}
Rather than work directly with $\mu^n=(\mu^n_t)_{t \in [0,T]}$,
we will frequently work with the extended empirical measure,
\begin{align*}
\bm{\overline\mu}^n := \frac{1}{n}\sum_{k=1}^n\delta_{(X^{n,k},\Lambda^{n,k})},
\end{align*}
which is a random element of $\P(\C^d \times \V)$.

Generically, for a (random) element $\bm{\overline{m}}$ of $\P(\C^d \times \V)$, we will write $\bm{m}$ for the (random) element of $\P(\C^d)$ obtained by marginalizing, and we write $m=(m_t)_{t \in [0,T]}$ for the corresponding flow of time-$t$ marginals, a (random) element of $\CP$.
This way, we may write $\bm{\mu}^n$ for random element of $\P(\C^d)$ given by
\[
\bm{\mu}^n := \frac{1}{n}\sum_{k=1}^n\delta_{X^{n,k}},
\]
and $\mu^n = (\mu_t^n)_{t\in [0, T]}$ for the corresponding measure flow. Notice that the map 
\begin{align}
\P^r(\C^d) \ni \bm{m} \mapsto m:=(m_t)_{t \in [0,T]} \in C([0, T]; \P^r(\R^d)) \label{def:marginalmap}
\end{align}
is continuous for any $r \ge 0$.

\subsection{Tightness and continuity}

Recall that the exponents $p,p'$ have been introduced in Assumption \ref{assumption:A}. 
This section gives the fairly straightforward tightness argument, as well as a useful continuity result. Both are naturally stated in terms of the space $\P^p(\C^d \times \V)$. Recall that the compact space $\V$ is equipped with an arbitrary compatible metric, whereas $\C^d$ is equipped with the supremum norm $\|\cdot\|_T$. We then equip $\C^d \times \V$ with the $\ell^1$ (sum) metric and define the Wasserstein metric on $\P^p(\C^d \times \V)$ accordingly, but we need only topological properties: Convergence of a sequence $(\bm{\overline{m}}^n)$ in $\P^p(\C^d \times \V)$ is characterized by the convergence of $\langle \bm{\overline{m}}^n,\psi\rangle$ for every continuous function $\psi : \C^d \times \V \to \R$ such that $\sup_{x,q}|\psi(x,q)|/(1+\|x\|_T^p) < \infty$.

Note that by continuity of the natural (marginal) mappings from $\P^p(\C^d \times \V)$ to $\P^p(\C^d)$ and from  $\P^p(\C^d)$ to $C([0,T];\P^p(\R^d))$, as mentioned in \eqref{def:marginalmap}, it is enough to show tightness of $(\bm{\overline\mu}^n, B)$ to conclude the tightness of $(\mu^n, B)$.

\begin{lemma} \label{le:tightness-main}
The sequence $(\bm{\overline\mu}^n,B)$ is a tight family of $\P^{qp'}(\C^d \times \V) \times \C^d$-valued random variables, for any $q \in [0,1)$. For any limit point $(\bm{\overline\mu},B)$, we have $\mu_0 = \lambda$ a.s., and 
\begin{align}
\E\int_{\C^d} \|x\|_T^{p'} \,\bm{\mu}(dx)< \infty.   \label{def:bound_on_mu}
\end{align}
\end{lemma}
\begin{proof}
Lemma \ref{le:inequalities} implies that 
\begin{align}
\sup_n\E\int_{\C^d} \|x\|_T^{p'} \,\bm{\mu}^n(dx) &= \sup_n\frac{1}{n} \sum_{k=1}^n \E\|X^{n,k}\|_T^{p'} \leq C\Bigg(1 + \sup_n\frac{1}{n} \sum_{k=1}^n\E |X_0^{n,k}|^{p'}\Bigg) < \infty. \label{pf:tight-moment1}
\end{align}
Now, suppose that $(\bm{\overline\mu},B)$ is some subsequential limit of $(\bm{\overline\mu}^n,B)$. 
By Fatou's inequality, the $\C^d$-marginal $\bm{\mu}$ of $\bm{\overline\mu}$ must then satisfy \eqref{def:bound_on_mu}. Moreover, $\L(\mu_0) = \delta_{\lambda}$, since $\mu^n_0 = \frac{1}{n}\sum_{k=1}^n\delta_{X^{n,k}_0}$ converges in probability to $\lambda$ by assumption.

It remains to prove the claimed tightness. 
As the marginal law of $B$ does not depend on $n$, it suffices to prove tightness of $(\L(\bm{\overline\mu}^n))_{n \in \N}\subset \P(\P^{qp'}(\C^d \times \V))$. In light of the moment bound \eqref{pf:tight-moment1}, it in fact suffices (e.g., by \cite[Corollary A.2]{lacker2015mean}) to show that the sequence of mean measures $(\E\bm{\overline\mu}^n)_{n\in\N} \subset \P(\C^d \times \V)$ is tight, where $\E\bm{\overline\mu}^n$ is defined by
\begin{align*}
\E\bm{\overline\mu}^n :=  \frac{1}{n} \sum_{k=1}^n \L(X^{n,k},\Lambda^{n,k}).
\end{align*}
To prove tightness of $(\E\bm{\overline\mu}^n)$, it suffices to prove tightness of the two sequences of marginals,
\begin{align*}
\frac{1}{n} \sum_{k=1}^n \L(X^{n,k}), \qquad \frac{1}{n} \sum_{k=1}^n \L(\Lambda^{n,k}).
\end{align*}
The latter sequence is tight because $\V$ is compact. The former we can prove to be tight using Aldous's criterion \cite[Theorem 16.11, Lemma 16.12]{kallenberg-foundations}; it suffices to show that
\begin{align}
\lim_{\delta \to 0} \sup_{n \in \N} \sup_{\tau,\tau'} \frac{1}{n} \sum_{k=1}^n \E \left[ |X_{\tau}^{n, k} - X_{\tau'}^{n,k}|^{r}\right] = 0,
    \label{pf:modified_Aldous}
\end{align}
for some $r > 0$, where the inner supremum is over all stopping times $(\tau,\tau')$ in $[0,T]$ such that $\tau \le \tau' \le \tau+\delta$.
This follows from a standard argument using the linear growth assumption (\ref{assumption:A}.3) along with the uniform moment bounds of Lemma \ref{le:inequalities}, and we omit the details.
\end{proof}

We will need to check the continuity of several functionals, and the following will cover our needs. We omit its proof, as it is a straightforward consequence of results in \cite[Appendix A]{lacker2015mean}.

\begin{lemma} \label{le:continuity}
Let $\psi : [0,T] \times \R^d \times \P^p(\R^d) \times A \to \R$ be jointly measurable. Assume $(x,m,a) \mapsto \psi(t,x,m,a)$ is continuous for each $t$, and assume there exists $c < \infty$ such that
\begin{align*}
|\psi(t,x,m,a)| &\le c\left( 1 + |x|^p + \int_{\R^d} |y|^p m(dy)\right),
\end{align*}
for all $(t,x,m,a)$. Then the following functional is continuous:
\begin{align*}
\P^p(\C^d \times \V) \ni \bm{\overline{m}} \mapsto \int_{\C^d \times A} \int_0^T \int_A \psi(t,x_t,m_t,a) \, q_t(da) \, dt \, \bm{\overline{m}}(dx,dq).
\end{align*}
\end{lemma}

\subsection{Identification of limiting dynamics} \label{se:mainproof-limitingdynamics} 

One of the key ideas in the closed-loop regime is to ``project away" randomness so that the SDEs are equivalent in laws to SDEs with closed loop coefficients. The next lemma provides one such projection argument, in a simple context that will come in handy in the proof of the subsequent Theorem \ref{th:limit-dynamics}. The results in this section are essentially extensions of the results of \cite[Section 5.3]{Lacker_closedloop}, though we streamline some of the arguments.
We omit the proof of the first lemma, as it is a straightforward adaptation of \cite[Lemma 5.2]{Lacker_closedloop}.

\begin{lemma} \label{le:conditional-exp-derivatives}
Suppose $(Y_t)_{t \in [0,T]}$ is a continuous stochastic process taking values in a Polish space $E$ and defined on some filtered probability space $(\Omega,\F,\FF, \PP)$ supporting a $d$-dimensional $\FF$-Brownian motion $B$. Suppose $h : E \rightarrow \R$ and  $g : E \rightarrow \R^d$ are continuous, with $g$ bounded, and suppose it holds that
\[
h(Y_t) = h(Y_0) + \int_0^ta_s \, ds + \int_0^t g(Y_s)^\top \, dB_s, \ \ a.s., \ \text{ for a.e. } t \in [0,T],
\]
where $(a_t)_{t \in [0,T]}$ is some $dt \otimes d\PP$-integrable real-valued process. Suppose $\widehat{a} : [0,T] \times C([0,T]; E \times \R^d) \rightarrow \R$ is a progressively measurable function satisfying
\[
\widehat{a}(t,Y, B) = \E[a_t \, | \, \F^{Y, B}_t], \ \ a.s., \ \text{ for a.e. } t \in [0,T].
\]
Then 
\[
h(Y_t) = h(Y_0) + \int_0^t\widehat{a}(s,Y,B)ds  + \int_0^t g(Y_s)^\top dB_s, \ \ \text{ for all } t \in [0,T], \ \ a.s.
\]
\end{lemma}

Next, we state here a form of the \emph{superposition principle} for Fokker-Planck SPDEs obtained recently in \cite{superposition_theorem}, which simplifies somewhat  due to the presence of strong solutions in our setting. The proof is deferred to Appendix \ref{ap:se:superposition}. We work with the following infinitesimal generator of the controlled process \eqref{def:stateprocesses-limitsection}: For each $\varphi \in C^\infty_c(\R^d)$ and $(t,x,m,a) \in [0,T] \times \R^d \times \P(\R^d) \times A$, we let
\begin{align}
L_{t, m}\varphi(x,a) := b(t,x,m,a) \cdot \nabla\varphi(x) + \tfrac12 \mathrm{tr}[(\sigma\sigma^\top+\gamma\gamma^\top)\nabla^2\varphi(x)]. \label{def:generator}
\end{align}

\begin{theorem} \cite[Theorem 1.3]{superposition_theorem} \label{th:superposition}
Suppose $\alpha : [0,T] \times \R^d \times C([0,T];\P^p(\R^d) \times \R^d) \to A$ is semi-Markov.
Let $(\Omega,\F,\PP)$ be a probability space supporting continuous processes $\nu$ and $\mu$ with values in $\P^p(\R^d)$ and $B$ with values in $\R^d$. Assume $B$ is a Brownian motion under $\FF^{\nu,\mu,B}$.
Assume that
\begin{align*}
\langle \nu_t, \varphi \rangle &= \langle \lambda, \varphi \rangle  + \int_0^t \langle \nu_s,\nabla\varphi\rangle^\top \gamma \, dB_s + \int_0^t\int_{\R^d}L_{s,\mu_s}\varphi(x,\alpha(s, x,\mu, B))\,\nu_s(dx)\, ds,
\end{align*}
a.s., for each $t \in [0,T]$ and $\varphi \in C^\infty_c(\R^d)$, and also that
\begin{align*}
\E\int_0^T \int_{\R^d}| b(t,x,\mu_t,\alpha(t,x,\mu,B))|^q\,\nu_t(dx)\,dt < \infty, \  \text{ for some } q > 1.
\end{align*}
Enlarge the probability space by adjoining an independent $\R^d$-valued random variable $\xi$ with law $\lambda$ and an independent Brownian motion $W$.
Then the following hold:
\begin{itemize}
\item $B$ is a $\FF^{\xi,\nu,\mu,B,W}$-Brownian motion.
\item There exists a  unique strong (i.e., adapted with respect to the completion of $\FF^{\xi,\nu,\mu,B,W}$) solution $X$ of the SDE
\begin{align}
dX_t = b(t,X_t,\mu_t,\alpha(t,X_t,\mu,B))dt + \sigma dW_t + \gamma dB_t, \quad X_0=\xi, \label{def:superpos-SDE}
\end{align}
and it
satisfies $\nu_t=\L(X_t\,|\,\F^{\nu,\mu,B}_t)=\L(X_t\,|\,\F^{\nu,\mu,B}_T)$ a.s., for each $t \in [0,T]$.
\end{itemize}
\end{theorem}

The next theorem identifies the limiting dynamics of $(\mu^n,B)$, showing that every limit point satisfies properties (1--5) of Definition \ref{def:weakMFE}. 

\begin{theorem} \label{th:limit-dynamics}
Suppose a subsequence $(\mu^{n_k}, B)_k$ converges  in law to some random element $(\mu, B)$ of $C([0,T];\P^p(\R^d) \times \R^d)$, defined on some probability space $(\Omega,\F,\FF,\PP)$. Then there exists a semi-Markov function $\alpha^* : [0,T] \times \R^d \times C([0, T]; \P(\R^d) \times \R^d) \rightarrow A$ such that, by enlarging the probability space, we may construct continuous $\FF$-adapted $d$-dimensional processes $X^*$ and $W$ such that the tuple $(\Omega,\F,\FF,\PP ,W , B, \alpha^*,\mu, X^*)$ satisfies the  properties (1--5) of Definition \ref{def:weakMFE}. 
Moreover, it holds that
\begin{align}
    \lim_{k\to\infty} \frac{1}{n_k} & \sum_{i = 1}^{n_k} J_i^{n_k}(\alpha^{n_k, 1}, \cdots, \alpha^{n_k, n_k}) \leq  \E\left[\int_0^T f(t, X^*_t, \mu_t, \alpha^*(t, X^*_t, \mu, B))dt + g(X^*_T, \mu_T) \right].
    \label{pf:cost_function_limit}
\end{align}
\end{theorem}
\begin{proof}
In light of the tightness established in Lemma \ref{le:tightness-main}, we may pass to a further subsequence and assume that $(\bm{\overline\mu}^n,B)$ converges in law to some random element $(\bm{\overline\mu},B)$ of $\P^p(\C^d \times \V) \times \C^d$, with $\mu=(\mu_t)_{t \in [0,T]}$ as the corresponding marginal flow.

Let $\varphi \in C_c^\infty(\R^d)$, and recall the notation for the infinitesimal generator introduced in \eqref{def:generator}. Apply It\^o's formula to $\varphi(X_t^{n,k})$ and average over $k = 1,\ldots,n$ to get
\begin{align*}
\langle \mu_t^n, \varphi \rangle  & =  \frac{1}{n} \sum_{k=1}^n \varphi(X_t^{n,k})  \\
	& = \langle \mu_0^n, \varphi \rangle  + \frac{1}{n}\sum_{k=1}^n \int_0^t L_{s, \mu_s^n}\varphi(X_s^{n,k}, \alpha^{n,k}(s,\bm{X}^n,S))\,ds \\
		&\qquad + \frac{1}{n}  \sum_{k=1}^n \int_0^t \nabla \varphi(X_s^{n,k}) \cdot \Big( \sigma dW_s^{k} + \gamma dB_s \Big) \\
    &= \langle \mu_0^n, \varphi \rangle  + \int_{\C^d \times \V} \left[\int_0^t\int_A L_{s, \mu_s^n}\varphi(x_s, a)q_s(da)ds \right] \bm{\overline{\mu}}^n(dx, dq) + M_t^{n,\varphi} + \widehat{M}_t^{n,\varphi} ,
\end{align*}
for $t\in [0, T]$, where we define the martingales $M^{n,\varphi}$ and $\widehat{M}^{n,\varphi}$ by
\begin{align*}
M_t^{n,\varphi}  &: = \frac{1}{n}  \sum_{k=1}^n \int_0^t \nabla \varphi(X_s^{n,k})^\top \sigma dW_s^{k},  \\
\widehat{M}_t^{n,\varphi} &: = \int_0^t \langle \mu_s^n, \nabla \varphi \rangle^\top \gamma dB_s.
\end{align*}
Since $\nabla \varphi$ is bounded, $M^{n,\varphi}$ is an average of orthogonal martingales, and we have
\begin{align}
  \E|M_t^{n,\varphi}|^2 = \frac{1}{n^2} \sum_{k=1}^n \int_0^t \E|\nabla \varphi(X_s^{n,k})^\top\sigma|^2 ds \leq \frac{T \, \|\nabla \varphi\|_\infty^2 |\sigma|^2}{n}.
  \label{pf:convergence_probability_orthogonal_martingales}
\end{align}
Define a functional $F^\varphi_t : \P^p(\C^d \times \V) \to \R$ by 
\begin{align*}
F^\varphi_t(\bm{\overline{m}}) & : = \int_{\C^d\times\V}\left[\varphi(x_t)-\varphi(x_0) - \int_0^t\int_AL_{s, m_s}\varphi(x_s,a)q_s(da)ds \right] \bm{\overline{m}}(dx,dq),
\end{align*}
noting that the integral is well defined for $\bm{\overline m} \in \P^p(\C^d \times \V)$ thanks to the growth assumption (\ref{assumption:A}.3).
We obtain  the following equation for each $t\in [0, T]$ and $\varphi \in C_c^\infty(\R^d)$:
\begin{align*}
F_t^{\varphi}(\overline{\boldsymbol{\mu}}^n) - \widehat{M}_t^{n,\varphi} = M_t^{n,\varphi}.
\end{align*}
Since $(\bm{\overline{\mu}}^n, B)$ converges in distribution to $(\bm{\overline{\mu}}, B)$ in $\P^p(\C^d \times \V) \times \C^d$, we may pass to the limit in this equation.
Indeed, this comes from three observations:
\begin{itemize}
\item $\E|M_t^{n,\varphi}|^2 \to 0$  by  \eqref{pf:convergence_probability_orthogonal_martingales}.
\item $F_t^\varphi$ is continuous on $\P^p(\C^d \times \V)$  by Lemma \ref{le:continuity}.
\item $(\bm{\overline{\mu}}^n, B,\widehat{M}_{\cdot}^{n,\varphi})$ converges in distribution to $(\bm{\overline{\mu}}, B,\widehat{M}_{\cdot}^{\varphi})$, where $\widehat{M}_{\cdot}^{\varphi} := \int_0^{\cdot} \langle \mu_s, \nabla \varphi \rangle^\top \gamma dB_s$. To see this, use Skorohod's representation to assume $(\bm{\overline{\mu}}^n, B)$ converges a.s.\ to $(\bm{\overline{\mu}}, B)$, followed by the convergence in probability of stochastic integrals \cite[Theorem 2.2]{Kurtz-Protter}. 
\end{itemize}
We conclude that for every $\varphi \in C^\infty_c(\R^d)$ and $t \in [0,T]$ it holds a.s.\ that $F_t^{\varphi}(\overline{\bm{\mu}})=\widehat{M}_{t}^{\varphi}$. Noting that Assumption \ref{assumption:B} implies $\mu_0=\lambda$ a.s., we may write this as
\begin{align}
\begin{split}
\langle \mu_t, \varphi \rangle &= \langle \lambda, \varphi \rangle  + \int_0^t \langle \mu_s,\nabla\varphi\rangle^\top \gamma \, dB_s  \\
	&\quad + \int_{\C^d \times \V} \left[\int_0^t\int_A L_{s, \mu_s}\varphi(x_s, a)q_s(da)ds \right] \bm{\overline{\mu}}(dx, dq).
	\end{split} \label{def:equation_for_mu_barre}
\end{align}
What remains is to ``project away" the extra randomness in the last line of \eqref{def:equation_for_mu_barre}, and then apply the superposition principle Theorem \ref{th:superposition} to construct the desired process $X^*$.

For each $t \in [0,T]$ the expression $\int_{\C^d \times \V}\delta_{x_t} \times q_t \, \bm{\overline{\mu}}(dx,dq)$ defines a random probability measure on $\R^d \times A$, and we denote by $\widehat{\Lambda}^*(t,\mu,B)$ its conditional expectation given $\F^{\mu,B}_t$. (Recall that we are working with a predictable version of the map $[0,T] \times \V \ni (t,q) \mapsto q_t \in \P(A)$, as explained in Section \ref{se:relaxedcontrols}.) That is, $\widehat{\Lambda}^* : [0,T] \times C([0,T];\P^p(\R^d) \times \R^d) \to \P^p(\R^d \times A)$ is a function, which can be taken to be progressively measurable \cite[Lemma C.3]{Lacker_closedloop}, satisfying
\begin{align}
\int_{\R^d \times A} \psi(t, x, \mu_t, a)\widehat{\Lambda}^*(t, \mu, B)(dx,da) = \E\left[ \int_{\C^d \times \V} \int_A \psi(t, x_t, \mu_t, a)q_t(da) \bm{\overline{\mu}}(dx,dq)\Big| \F_t^{\mu, B}\right] \label{pf:limitdyn-Lamba1}
\end{align}
a.s., for each suitably integrable $\psi : [0,T] \times \R^d \times \P^p(\R^d) \times A\to \R$. Note that the $\R^d$-marginal of $\widehat{\Lambda}^*(t,\mu,B)$ equals $\mu_t$, so we can disintegrate it by
\begin{align}
\widehat{\Lambda}^*(t,\mu,B)(dx,da) = \mu_t(dx)\Lambda^*(t,x,\mu,B)(da), \label{pf:limitdyn-Lamba2}
\end{align}
for some semi-Markov function $\Lambda^* : [0,T] \times \R^d \times C([0,T];\P^p(\R^d) \times \R^d) \to \P(A)$.
In particular, we find
\begin{align*}
\int_{\R^d}\int_A L_{t,\mu_t}\varphi(x,a)\Lambda^*(t, x,\mu, B)(da)\mu_t(dx) = \E\left[ \int_{\C^d \times \V} \int_A L_{t,\mu_t}\varphi(x_t,a)q_t(da) \bm{\overline{\mu}}(dx,dq)\Big| \F_t^{\mu, B}\right].
\end{align*}
Combining this with \eqref{def:equation_for_mu_barre}, we may now apply Lemma \ref{le:conditional-exp-derivatives} (with the identifications $E=\P^p(\R^d)$, $Y=\mu$, $h(m)=\langle m,\varphi\rangle$, and $g(m) = \gamma^\top \langle m,\nabla\varphi\rangle$) to rewrite \eqref{def:equation_for_mu_barre} as
\begin{align}
\begin{split}
\langle \mu_t, \varphi \rangle &= \langle \lambda, \varphi \rangle  + \int_0^t \langle \mu_s,\nabla\varphi\rangle^\top \gamma \, dB_s  \\
	&\quad + \int_0^t\int_{\R^d}\int_A L_{s,\mu_s}\varphi(x,a)\Lambda^*(s, x,\mu, B)(da)\mu_s(dx)ds.
	\end{split} \label{def:equation_for_mu_barre2}
\end{align}

Next, we pass from the relaxed control to a strict control. Define semi-Markov functions $(c_1,c_2) : [0,T] \times \R^d \times C([0,T]; \P^p(\R^d) \times \R^d) \to \R^d \times \R$ by
\begin{align*}
\big(c_1(t,x,\mu,B),c_2(t,x,\mu,B)\big) &:= \int_A \big(b(t,x,\mu_t,a),f(t,x,\mu_t,a)\big) \Lambda^*(t, x,\mu, B)(da),
\end{align*}
a.s., for $(t,x) \in [0,T] \times \R^d$, and note that this belongs to the set $K(t,x,\mu_t)$ from Assumption (\ref{assumption:A}.5).
Hence, using a measurable selection argument (see \cite[Lemma 3.1]{dufourstockbridge-existence}), we may find a semi-Markov function $\alpha^* : [0,T] \times \R^d \times C([0,T]; \P(\R^d) \times \R^d) \to A$ such that
\begin{align}
c_1(t,x,\mu,B) &= b(t,x,\mu_t,\alpha^*(t,x,\mu,B)), \label{pf:limitdyn-projb1} \\
c_2(t,x,\mu,B) &\le f(t,x,\mu_t,\alpha^*(t,x,\mu,B)).  \label{pf:limitdyn-projf1}
\end{align}
Applying \eqref{pf:limitdyn-projb1} in \eqref{def:equation_for_mu_barre2}, we rewrite \eqref{def:equation_for_mu_barre2} as
\begin{align}
\begin{split}
\langle \mu_t, \varphi \rangle &= \langle \mu_0, \varphi \rangle  + \int_0^t \langle \mu_s,\nabla\varphi\rangle^\top \gamma \, dB_s + \int_0^t\int_{\R^d}L_{s,\mu_s}\varphi(x,\alpha^*(s, x,\mu, B))\,\mu_s(dx)\, ds.
	\end{split} \label{def:equation_for_mu_barre3}
\end{align}
Note that the linear growth assumption (\ref{assumption:A}.3) and the moment bound of Lemma \ref{le:tightness-main} imply
\begin{align*}
\E\int_0^T \int_{\R^d} |b(t, x, \mu_t, \alpha^*(t, x, \mu, B))|^{p'}\, \mu_t(dx) \, dt \le C\E\left(1 +  \int_{\C^d} \|x\|_T^{p'} \boldsymbol{\mu}(dx)\right) <\infty. 
\end{align*}
Note in the case $p=p'=0$ that the drift $b$ is bounded.
We are thus in a position to apply the superposition principle, Theorem \ref{th:superposition}, with $\nu\equiv\mu$: Enlarging if necessary the probability space $(\Omega,\F,\FF,\PP)$, we deduce the existence of $X^*$ and $W$ such that the tuple $(\Omega,\F,\FF,\PP ,W , B, \alpha^*,\mu, X^*)$ satisfies the  properties (1--5) of Definition \ref{def:weakMFE}. 

It remains to prove the final claim \eqref{pf:cost_function_limit}. Notice first that the left-hand side can be rewritten in terms of the empirical measure:
\begin{align*}
J_n &:= \frac{1}{n}\sum_{i=1}^nJ^n_i(\alpha^{n,1},\ldots,\alpha^{n,n}) \\
	&= \frac{1}{n}\sum_{i=1}^n\E\left[\int_0^T f(t,X^{n,i}_t,\mu^n_t,\alpha^{n,i}(t,\bm{X}^n, S))dt + g(X^{n,i}_T,\mu^n_T)\right] \\
	&= \E\left[\int_{\C^d \times \V}\left(\int_0^T\int_A f(t,x_t,\mu^n_t,a)q_t(da)dt + g(x_T,\mu^n_T)\right)\bm{\overline\mu}^n(dx,dq)\right].
\end{align*}
Recall that we passed to a subsequence $(n_k)$ along which $(\bm{\overline\mu}^n,B)$ converges in law to $(\bm{\overline\mu},B)$ in $\P^p(\C^d \times \V)$. 
In the case $p'=p=0$, we may apply Lemma \ref{le:continuity} thanks to the continuity and boundedness of $f$ and $g$ to deduce that $J_{n_k}$ converges to
\begin{align}
\E\left[\int_{\C^d \times \V}\left(\int_0^T\int_A f(t,x_t,\mu_t,a)q_t(da)dt + g(x_T,\mu_T)\right)\bm{\overline\mu}(dx,dq)\right]. \label{pf:LD-Jnk-lim}
\end{align}
For the case $0 < p \le p \vee 2 < p'$, the integrand still converges in law to the desired limit by Lemma \ref{le:continuity}, so we must only check uniform integrability to deduce again the $J_{n_k}$ converges to \eqref{pf:LD-Jnk-lim}. To do so, simply note that the growth assumption (\ref{assumption:A}.4) along with the $p'$-moment bound of Lemma \ref{le:inequalities} and Jensen's inequality yield
\begin{align*}
\sup_n \E\left[\left(\int_{\C^d \times \V}\left(\int_0^T\int_A f(t,x_t,\mu^n_t,a)q_t(da)dt + g(x_T,\mu^n_T)\right)\bm{\overline\mu}^n(dx,dq)\right)^{p'/p}\right] < \infty.
\end{align*}

To complete the proof, we must upper bound the limiting quantity \eqref{pf:LD-Jnk-lim} by the right-hand side of \eqref{pf:cost_function_limit}.
Note next that the property $\mu_T=\L(X^*_T\,|\,\F^{\mu,B}_T)$ implies
\begin{align*}
\E\left[ \int_{\C^d \times \V} g(x_T,\mu_T) \, \bm{\overline\mu}(dx,dq)\right] = \E\left[\int_{\R^d} g(x,\mu_T)\,\mu_T(dx) \right] = \E[g(X^*_T,\mu_T)].
\end{align*}
Finally, combining \eqref{pf:limitdyn-Lamba1}, \eqref{pf:limitdyn-Lamba2}, and \eqref{pf:limitdyn-projf1}, we find
\begin{align*}
\E &\left[\int_{\C^d \times \V}\left(\int_0^T\int_A f(t,x_t,\mu_t,a)q_t(da)dt\right)\bm{\overline\mu}(dx,dq)\right] \\
	& \quad = \E\left[\int_0^T \int_{\R^d} \int_A f(t,x,\mu_t,a) \, \Lambda^*(t,x,\mu,B)(da) \, \mu_t(dx) \, dt  \right] \\
	& \quad = \E\left[\int_0^T \int_{\R^d} c_2(t,x,\mu,B) \, \mu_t(dx) \, dt  \right] \\
	& \quad \le \E\left[\int_0^T \int_{\R^d} f(t,x,\mu_t,\alpha^*(t,x,\mu,B)) \, \mu_t(dx) \, dt  \right] \\
	& \quad = \E\left[\int_0^T  f(t,X^*_t,\mu_t,\alpha^*(t,X^*_t,\mu,B)) \, dt  \right],
\end{align*}
with the last identity using Fubini's theorem and the property $\mu_t=\L(X^*_t\,|\,\F^{\mu,B}_t)$.
\end{proof}

\subsection{Optimality} \label{se:1playerdeviations}

With Theorem \ref{th:limit-dynamics} now established, we have no further need for the extended empirical measure $\overline{\bm{\mu}}^n$, and we will work henceforth with its time-$t$ marginal flow. Throughout this section, we fix a weak limit $(\mu,B)$ of $(\mu^n,B)$, whose existence is guaranteed by Lemma \ref{le:tightness-main}.
We may associate this limit with a semi-Markov function $\alpha^*$ and a filtered probability space $(\Omega,\F,\FF,\PP)$ satisfying the conclusions of Theorem \ref{th:limit-dynamics}.
We abuse notation by relabeling the subsequence, so that all limits $n\to\infty$ in this section are understood to refer to this fixed convergent subsequence.
The final step of the proof of Theorem \ref{th:mainlimit}, carried out in this section, is to show that $(\mu, B)$ satisfies the optimality property (6) of Definition \ref{def:weakMFE}.

Let us write $\A_{\mathrm{semi}}$ for the set of semi-Markov functions from $[0,T] \times \R^d \times C([0, T]; \P^p(\R^d) \times \R^d)$ to $A$. 
For any $\alpha \in \A_{\mathrm{semi}}$, define $X=X[\alpha]$ to be the solution of the SDE 
\begin{align}
d X_t = b(t, X_t,\mu_t,\alpha(t, X_t,\mu, B))dt + \sigma dW_t + \gamma dB_t, \quad X_0 = X_0^*. 
\label{def:SDE-semimarkov-strong-optsec}
\end{align}
Define
\begin{align*}
J(\alpha) &:= \E\left[\int_0^T f(t, X_t[\alpha], \mu_t, \alpha(t, X_t[\alpha], \mu, B))dt + g(X_T[\alpha], \mu_T)\right].
\end{align*}
To complete the proof of Theorem \ref{th:mainlimit}, we need to show that 
\begin{align}
\sup_{\alpha \in \A_{\mathrm{semi}}}J(\alpha) = J(\alpha^*). \label{pf:optimality-goal1}
\end{align}
It will help, however, to restrict the supremum to a dense subset of better-behaved controls: Define $\A_{\mathrm{semi}}^{\mathrm{cont}}$ to be  the set of $\alpha \in \A_{\mathrm{semi}}$ which are continuous functions of all their variables.
Following \cite{Lacker_closedloop}, we prove \eqref{pf:optimality-goal1} in two steps. The following lemma is reminiscent of the $\Gamma$-limsup in $\Gamma$-convergence arguments.

\begin{lemma} \label{le:approx-optimizer-nice}
For any $\beta \in \A_{\mathrm{semi}}$, there exists a sequence $\beta^n \in \A_{\mathrm{semi}}^{\mathrm{cont}}$ such that $(\mu,B, X[\beta^n])$ converges in law to $(\mu,B, X[\beta])$ and $J(\beta^n) \rightarrow J(\beta)$.
\end{lemma}

The next proposition is the heart of the argument, showing that for any $\beta \in \A_{\mathrm{semi}}^{\mathrm{cont}}$ we may construct a sequence of deviating strategies for each player in the $n$-player game, with the property that the average value converges to $J(\beta)$. Note that $B$ is adapted to the signal process filtration $\FF^S$ by assumption so we may find a measurable function $\phi : C([0,T];\mathcal{S})  \to \C^d$ such that $B=\phi(S)$ a.s. We make some use of this function in the following.

\begin{proposition} \label{pr:1playerdeviation}
Let $\beta \in \A_{\mathrm{semi}}^{\mathrm{cont}}$. For each $n$ and $k=1,\ldots,n$, define $\beta^{n,k} \in \A_n(S)$ by
\begin{align}
\beta^{n,k}(t,\bm{x}, s) = \beta\Big(t,x^k_t,\Big(\frac{1}{n}\sum_{j=1}^n\delta_{x^j_u}\Big)_{u \in [0,T]}, \phi(s)\Big),  \label{def:pr:1playerdev-specialformbeta}
\end{align}
for $t \in [0,T]$, $\bm{x}=(x^1,\ldots,x^n) \in (\C^d)^n$, $s \in C([0,T];\mathcal{S})$.
Then (along the fixed sequence discussed above)
\begin{align}
\lim_n \frac{1}{n}\sum_{k=1}^n J^{n}_k(\alpha^{n,1},\ldots,\alpha^{n,k-1},\beta^{n,k},\alpha^{n,k+1},\ldots,\alpha^{n,n}) = J(\beta). \label{goal:1playerdeviation}
\end{align}
\end{proposition}

With these two results, the proof of Theorem \ref{th:mainlimit} is completed immediately:
 
\noindent\textbf{Proof of Theorem \ref{th:mainlimit}.}
By Lemma \ref{le:approx-optimizer-nice}, to prove \eqref{pf:optimality-goal1} it suffices to show that $J(\beta) \le J(\alpha^*)$ for any $\beta \in \A_{\mathrm{semi}}^{\mathrm{cont}}$. Considering $\beta^{n,k}$ as in Proposition \ref{pr:1playerdeviation}, we take limits (along the convergent subsquence described at the beginning of the section):
\begin{align*}
J(\beta) &= \lim_n \frac{1}{n}\sum_{k=1}^n J^{n}_k(\alpha^{n,1},\ldots,\alpha^{n,k-1},\beta^{n,k},\alpha^{n,k+1},\ldots,\alpha^{n,n}) \\
	&\le \lim_n\frac{1}{n}\sum_{k=1}^{n}J^{n}_k(\alpha^{n,1},\ldots,\alpha^{n,n}) + \epsilon_n \\
	& \leq J(\alpha^*).
\end{align*}
where we used the result \eqref{goal:1playerdeviation} of Proposition \ref{pr:1playerdeviation} in the first line, the Nash equilibrium property of $(\alpha^{n,1},\ldots,\alpha^{n,n})$ in the second line, and the inequality \eqref{pf:cost_function_limit} of Theorem \ref{th:limit-dynamics} in the last line (along with $\epsilon_n\rightarrow 0$).
\hfill\qedsymbol 

{\ }

We proceed now to proving the two main subresults Lemma \ref{le:approx-optimizer-nice} and Proposition \ref{pr:1playerdeviation}.

{\ } 

\noindent\textbf{Proof of Lemma \ref{le:approx-optimizer-nice}.} 
Let $M=\L(\mu,B) \in \P(C([0, T];\P^p(\R^d)) \times \C^d)$. Arguing exactly as in \cite[Step 2, proof of Lemma 5.5]{Lacker_closedloop}, we may find a sequence of continuous functions $\beta^n \in \A_{\mathrm{semi}}^{\mathrm{cont}}$ such that
$dt\delta_{\beta^n(t,x,m,w)}(da) \to dt\delta_{\beta(t,x,m,w)}(da)$ in $\V$ for Lebesgue-a.e.\ $x \in \R^d$ and $M$-a.e.\ $(m,w)$. With $d_A$ denoting a metric for the compact space $A$, it follows that
\[
\int_0^T \!\!d_A(\beta^n(t,x,m,w),\beta(t,x,m,w))dt = \!\int_0^T\!\!\int_A d_A(a,\beta(t,x,m,w))\,\delta_{\beta^n(t,x,m,w)}(da)dt \to 0
\]
for Lebesgue-a.e.\ $x \in \R^d$ and $M$-a.e.\ $(m,w)$. By passing to a further subsequence, we may finally assume that
$\beta^n(t,x,m,w) \to \beta(t,x,m,w)$ a.e., where ``a.e." here and throughout this proof will mean ``for Lebesgue-a.e.\ $(t,x)$ and $M$-a.e.\ $(m,w)$." A stability argument for SDEs with random coefficients, worked out in Lemma \ref{ap:le:stability_lemma}, shows that
\begin{align}
\lim_n \E [h_n(\mu,B,X[\beta^n])] = \E [h(\mu,B,X[\beta])] \label{pf:opt:strongconv1}
\end{align}
for any bounded measurable functions $h,h_n : C([0, T]; \P^p(\R^d)) \times \C^d \times \C^d \mapsto \R$ sharing a uniform bound  such  that $h_n(\mu,B,X[\beta]) \to h(\mu,B,X[\beta])$ in probability. In particular, $(\mu, B, X[\beta^n])$ converges in law to $(\mu, B, X[\beta])$. Define $F_n,F : C([0, T]; \P^p(\R^d)) \times \C^d \times \C^d \mapsto \R$ by
\begin{align*}
F_n\big(m,w,(x_t)_{t \in [0,T]}\big) &:= \int_0^T f(t, x_t, m_t, \beta^n(t, x_t , m, w))dt + g(x_T, m_T), \\
F\big(m,w,(x_t)_{t \in [0,T]}\big) &:= \int_0^T f(t, x_t, m_t, \beta(t, x_t , m, w))dt + g(x_T, m_T).
\end{align*}
Since $\beta^n \to \beta$ a.e., the continuity and the growth assumption (\ref{assumption:A}.4) of $f$ and $g$ ensure that $F_n \to F$ a.e.
Indeed, this follows from dominated convergence because, for each fixed $(m,w,(x_t)_{t \in [0,T]})$, the functions $t\mapsto f(t, x_t, m_t, \beta^n(t, x_t , m, w))$ and $t\mapsto f(t, x_t, m_t, \beta(t, x_t , m, w))$ are uniformly bounded.
Now, using \eqref{pf:opt:strongconv1}, we deduce that $F_n(\mu,B,X[\beta^n])$ converges in law to $F(\mu,B,X[\beta])$. 
The uniform moment bound \eqref{moment:Xmfe} and the growth assumption (\ref{assumption:A}.4) ensure that $F_n(\mu,B,X[\beta^n])$ are uniformly integrable, and we deduce that 
\begin{align*}
J(\beta^n) = \E[F_n(\mu,B,X[\beta^n])] \to \E[F(\mu,B,X[\beta])] = J(\beta). 
\end{align*}

{\ } 

We conclude the section with the lengthy proof of Proposition \ref{pr:1playerdeviation}.

\noindent\textbf{Proof of Proposition \ref{pr:1playerdeviation}.}
Recall that $\bm{X}^n=(X^{n,1},\ldots,X^{n,n})$ solves the SDE \eqref{def:stateprocesses-limitsection} and is defined on a probability space $(\Omega^n,\F^n,\FF^n,\PP^n)$. 
For each $k=1,\ldots,n$ we define the vector $\bm{Y}^{n,k} = (Y^{n,k,1},\ldots,Y^{n,k,n})$ of state processes arising when player $k$ switches controls from $\alpha^{n,k}$ to $\beta^{n,k}$. That is, $\bm{Y}^{n,k}$ follows the dynamics
\begin{align*}
dY^{n,k,k}_t &=  b (t,Y^{n,k,k}_t,\mu^{n,k}_t, \beta(t,Y^{n,k,k}_t,\mu^{n,k}, B))dt + \sigma dW^{k}_t + \gamma dB_t, \\
dY^{n,k,i}_t &= b(t,Y^{n,k,i}_t,\mu^{n,k}_t,\alpha^{n,i}(t,\bm{Y}^{n,k}, S))dt + \sigma dW^{i}_t + \gamma dB_t, \quad i \neq k, \\
\mu^{n,k}_t &= \frac{1}{n}\sum_{j=1}^n\delta_{Y_t^{n,k,j}}, \qquad \boldsymbol{Y}^{n, k}_0 = \boldsymbol{X}^n_0. 
\end{align*}
We may assume that the process $\bm{Y}^{n,k}$ lives on the same probability space, but this is merely for notational convenience. 
We then have
\begin{align}
\begin{split}
& \frac{1}{n}\sum_{k=1}^n  J^{n}_k(\alpha^{n,1},\ldots,\alpha^{n,k-1},\beta^{n,k},\alpha^{n,k+1},\ldots,\alpha^{n,n})  \\
 &= \frac{1}{n}\sum_{k=1}^n  \E\left[\int_0^T  f(t,Y^{n,k,k}_t,\mu^{n,k}_t, \beta(t,Y^{n,k,k}_t,\mu^{n,k}, B))dt + g(Y^{n,k,k}_T,\mu^{n,k}_T) \right].
\end{split} \label{pf:opt:Jidentity1}
\end{align}
To show \eqref{goal:1playerdeviation}, we claim that it suffices to show that, for every bounded continuous semi-Markov function $\varphi : [0,T] \times \R^d \times C([0,T];\P^p(\R^d)) \times \C^d \to \R$ and every $t \in [0,T]$, we have
\begin{align}
\lim_n \frac{1}{n}\sum_{k=1}^n  \E[ \varphi(t,Y^{n,k,k}_t,\mu^{n,k},B)] = \E[ \varphi(t,X_t[\beta],\mu,B)], \label{pf:opt:mainlaw}
\end{align}
where $X[\beta]$ solves the SDE \eqref{def:SDE-semimarkov-strong-optsec}, and the limit is as usual along the same subsequence for which $(\mu^n,B)$ converges in law to $(\mu,B)$.
Indeed, the functional
\begin{align*}
\C^d \times C([0,T];\P^p(\R^d)) \times \C^d \ni (y,m,w) \mapsto \int_0^T f(t,y_t,m_t,\beta(t,y_t,m,w))\,dt + g(y_T,m_T)
\end{align*}
is continuous since $\beta$ is continuous (cf.\ Lemma \ref{le:continuity}). For the case $p=p'=0$, the limit \eqref{pf:opt:mainlaw} allows us to conclude the proof. For the case $p>0$, using the growth assumption (\ref{assumption:A}.4) of $f$ and $g$ along with the $p'$-moment bound from Lemma \ref{le:inequalities} and Jensen's inequality, we have
\begin{align*}
\sup_n  \frac{1}{n}\sum_{k=1}^n  \E\left[\left(\int_0^T  f(t,Y^{n,k,k}_t,\mu^{n,k}_t, \beta(t,Y^{n,k,k}_t,\mu^{n,k}, B))\,dt + g(Y^{n,k,k}_T,\mu^{n,k}_T) \right)^{p'/p}\right] < \infty.
\end{align*}
Since $p'/p > 1$, this provides ample uniform integrability to deduce the claimed \eqref{goal:1playerdeviation} from \eqref{pf:opt:mainlaw} and \eqref{pf:opt:Jidentity1}.
Hence, it remains to prove \eqref{pf:opt:mainlaw}.

{\ } 

\noindent\textbf{Step 1.}
We first construct a change of measure and collect its  basic properties, following for now the same strategy as in \cite[Proof of Proposition 5.6]{Lacker_closedloop}. For each $k$, define the processes
\begin{align}
\Xi^{n,k}_t &:= \sigma^{-1}\left[b(t,X^{n,k}_t,\mu^n_t,\beta(t,X^{n,k}_t,\mu^n, B)) - b(t,X^{n,k}_t,\mu^n_t,\alpha^{n,k}(t,\bm{X}^n, S)) \right], \label{pf:def:Xink} \\ 
\zeta^{n,k}_t &:= \exp\left( \int_0^t \Xi^{n,k}_s  \cdot dW^k_s - \frac12 \int_0^t | \Xi^{n,k}_s  |^2\,ds  \right). \label{pf:def:Zetank}
\end{align}
It follows from \cite[Theorem 7.7]{liptser2001statistics} that $(\zeta^{n,k}_t)_{t \in [0,T]}$ is a martingale, not merely a local martingale; the integrability assumptions therein are easily checked using Lemma \ref{le:inequalities} and the linear growth assumption (\ref{assumption:A}.3). Hence, since $\E[\zeta^{n,k}_T]=\zeta^{n,k}_0=1$, we may define the change of measure
\begin{align*}
\frac{d\QQ^{n,k}}{d\PP^n} := \zeta^{n,k}_T.
\end{align*}
By Girsanov's theorem, the processes $(W^i)_{i \neq k}$ and
\begin{align*}
\widetilde{W}^{k}_t := W^k_t - \int_0^t \Xi^{n,k}_s\,ds
\end{align*}
are independent Brownian motions under $\QQ^{n,k}$. We thus see that $\bm{X}^n$ solves under $\QQ^{n,k}$ the same SDE that $\bm{Y}^{n,k}$ solves under $\PP^n$, and we would like to deduce from uniqueness of the SDE that $\QQ^{n,k} \circ (\bm{X}^n,S)^{-1} = \PP^n\circ(\bm{Y}^{n,k},S)^{-1}$. But, recalling Lemma  \ref{le:nplayerSDE-lemma} to do this carefully in the presence of the random factor $S$, we must check that $\bm{W}^{n,k}:=(\widetilde{W}^k,(W^i)_{i \neq k})$ is an an $\FF^{\bm{X}^n,S}$-Brownian motion under the conditional measure $\QQ^{n,k}(\cdot\,|\,S)$, a.s. 
To see this, note by Definition \ref{def:nplayerSDE} that $\bm{W}^n=(W^1,\ldots,W^n)$ is a.s.\ an $\FF^{\bm{X}^n,S}$-Brownian motion under the conditional measure $\PP^n(\cdot\,|\,S)$.
From this (and the $S$-adaptedness of $B$) we deduce easily that $\E[\zeta^{n,k}_T\,|\,S]=1$ a.s., which implies that in fact
\begin{align*}
\frac{d\QQ^{n,k}(\cdot\,|\,S)}{d\PP^n(\cdot\,|\,S)} = \zeta^{n,k}_T.
\end{align*}
Hence, we may apply Girsanov's theorem to the conditional measures to deduce that $\bm{W}^{n,k}$ is a Brownian motion under $\QQ^{n,k}(\cdot\,|\,S)$, a.s., in the filtration $\FF^{\bm{X}^n,S}$.

Now that we know that $\QQ^{n,k} \circ (\bm{X}^n,S)^{-1} = \PP^n \circ (\bm{Y}^{n,k},S)^{-1}$, we deduce (recalling that $B$ is $S$-measurable) for bounded semi-Markov functions $\varphi : [0,T] \times \R^d \times C([0,T];\P^p(\R^d) \times \R^d \times \S) \to \R$ and $t \in [0,T]$ that
\begin{align}
\frac{1}{n}\sum_{k=1}^n  \E[ \varphi(t,Y^{n,k,k}_t,\mu^{n,k},B,S)] = \frac{1}{n}\sum_{k=1}^n  \E[ \zeta^{n,k}_t\varphi(t,X^{n,k}_t,\mu^n,B,S)], \label{pf:opt:measurechange1}
\end{align}
where $\E$ denotes expectation under $\PP^n$ throughout this section.
Hence, we can study the limit \eqref{pf:opt:mainlaw} using the augmented particle system $(X^{n,k},\zeta^{n,k})_{k=1}^n$.
We note here for later use that
\begin{align}
\sup_{t\in [0,T]}\sup_{n \in \N}\frac{1}{n} \sum_{k=1}^n \E [\zeta^{n,k}_t \log \zeta^{n,k}_t] < \infty. \label{pf:uniform_bound_zeta}
\end{align}
Indeed, this follows from the identity
\begin{align*}
\E[\zeta^{n,k}_t \log \zeta^{n,k}_t] &= \frac12 \E\left[\zeta^{n,k}_t \int_0^t |\Xi^{n,k}_s|^2\,ds \right] \\
	&= \frac12 \E \int_0^t \Big| \sigma^{-1}\Big(b(s,Y^{n,k,k}_s,\mu^{n,k}_s,\beta(s,Y^{n,k,k}_s,\mu^{n,k}, B)) \\
	&\qquad\qquad\qquad\qquad - b(s,Y^{n,k,k}_s,\mu^{n,k}_s,\alpha^{n,k}(s,\bm{Y}^{n,k}, S)) \Big) \Big|^2\,ds,
\end{align*}
with the last step using \eqref{pf:opt:measurechange1}, along with the linear growth assumption (\ref{assumption:A}.3) and the $p'>2$ moment bound of Lemma \ref{le:inequalities} (the case $p'=0$ being trivial here).

{\ }

\noindent\textbf{Step 2.}\label{def:section_2}
We next define an augmented empirical measure sequence and prove its tightness.
We use relaxed controls once again, as defined in Section \ref{se:relaxedcontrols}. Recall that we view $\Lambda^{n,k}=dt\Lambda^{n,k}_t(da)=dt\delta_{\alpha^{n,k}(t,\bm{X}^n, S)}(da)$ as a $\V$-valued random variable.
Consider the extended empirical measure
\begin{align*}
\bm{\widehat{\mu}}^n & := \frac{1}{n}\sum_{k=1}^n \delta_{(X^{n,k},\zeta^{n,k},\Lambda^{n,k})},
\end{align*}
viewed as a random element of $\P(\widehat\Omega)$, where $\widehat\Omega := \C^d \times \C^1_+ \times \V$. Here $\C^1_+ := C([0,T];\R_+)$ is the space of strictly positive one-dimensional continuous paths. 

We show in this step that this sequence of $\P(\widehat\Omega)$-valued random variables is tight.  As in the proof of Lemma \ref{le:tightness-main}, it suffices to show the mean measure sequence $\E\bm{\widehat{\mu}}^n$ is tight. By definition,
\begin{align*}
\E \bm{\widehat{\mu}}^n  = \frac{1}{n}\sum_{k=1}^n\L(X^{n,k},\zeta^{n,k},\Lambda^{n,k}).
\end{align*}
It suffices to prove the tightness of each of the three marginal sequences,
\begin{align*}
\frac{1}{n}\sum_{k=1}^n\L(X^{n,k}), \quad \frac{1}{n}\sum_{k=1}^n\L(\zeta^{n,k}) , \quad \frac{1}{n}\sum_{k=1}^n\L(\Lambda^{n,k}) .
\end{align*}
The first was shown in the proof of Lemma \ref{le:tightness-main}. The third is automatic from compactness of $\V$. The second requires more care. By the continuous mapping theorem (since $\zeta^{n,k} > 0$), it suffices to prove tightness of 
\begin{align*}
\frac{1}{n}\sum_{k=1}^n\L( \log \zeta^{n,k}) \in \P(\C^1).
\end{align*}
Since $\log \zeta^{n,k}_0=0$,  using the criterion of Aldous \cite[Theorem 16.11, Lemma 16.12]{kallenberg-foundations}, it suffices to check that
\begin{align}
\lim_{\delta \to 0} \sup_{n \in \N} \sup_{\tau,\tau'} \frac{1}{n} \sum_{k=1}^n \E\left[ |\log \zeta_{\tau}^{n,k} - \log \zeta_{\tau'}^{n,k}| \right] = 0, \label{pf:opt:Aldous2}
\end{align}
where the inner supremum is over all stopping times $(\tau,\tau')$ in $[0,T]$ such that $\tau \le \tau' \le \tau+\delta$.
Note that $\log \zeta^{n,k} = M^{n,k} - \tfrac12 [M^{n,k}]$, where we define the martingales $M_t^{n,k} := \int_0^t \Xi_{s}^{n,k}\cdot dW_s^{k}$.
By It\^o isometry, for any such $(\tau,\tau')$, we have
\begin{align*}
\E  |M_{\tau}^{n,k} - M_{\tau'}^{n,k}|^2  &= \E \int_{\tau}^{\tau'} |\Xi_t^{n,k}|^2\,dt.
\end{align*}
Then \eqref{pf:opt:Aldous2} follows easily in the case $p=p'=0$. In the case $0 < p \le p \vee 2 < p'$,
the moment bounds of Lemma \ref{le:inequalities} and the growth assumption (\ref{assumption:A}.3) of the drift $b$ lead to
\begin{align}
\sup_n \frac{1}{n}\sum_{k=1}^n \E\|\Xi^{n,k}\|_T^{p'} < \infty, \label{pf:opt:xibound1}
\end{align}
and this provides enough uniform integrability to deduce \eqref{pf:opt:Aldous2}.

{\ }

\noindent\textbf{Step 3.} Having shown $(\widehat{\bm{\mu}}^n)$ to be tight in Step 2, and thus also $(\widehat{\bm{\mu}}^n,B)$, this step identifies the dynamics of an arbitrary limit point $(\widehat{\bm{\mu}},B)$ of the latter sequence in terms of a Fokker-Planck SPDE.
Recall that limits occur along the same subsequence for which $(\mu^n,B)$ converges in law to $(\mu,B)$ in $C([0,T];\P^p(\R^d)) \times \C^d$. By applying the continuous map $\pi_t : \widehat\Omega \to \R^d$ given by $\pi_t(x,y,q):=x_t$, we have necessarily that  $\widehat{\bm{\mu}} \circ \pi_t^{-1} = \mu_t$ for each $t$.
Consider also the projection $\widetilde\pi_t : \widehat\Omega \to \R^d \times \R_+$ given by $\widetilde{\pi}_t(x,y,q):=(x_t,y_t)$, and define
\begin{align}
\widetilde\mu^n_t := \widehat{\bm{\mu}}^n \circ \widetilde\pi_t^{-1}, \qquad \widetilde\mu_t := \widehat{\bm{\mu}} \circ \widetilde\pi_t^{-1}, \qquad t \in [0,T]. \label{pf:def:tildemu}
\end{align}
The goal of this step is to show that
\begin{align}
\begin{split}
\langle \widetilde{\mu}_t, \varphi \rangle &= \langle \widetilde{\mu}_0, \varphi \rangle +  \int_0^t\langle \widetilde{\mu}_s, \nabla_x \varphi  \rangle^\top \gamma dB_s \\
&\quad + \int_{\widehat{\Omega} } \int_0^t \int_A  \mathscr{L}_{s, \mu, B}\varphi(x_s, y_s, a)\,q_s(da) \, ds\, \widehat{\boldsymbol{\mu}}(dx,dy,dq),    
\end{split} \label{pf:SDE_mutilde}
\end{align}
a.s., for each $t \in [0,T]$ and $\varphi \in C^\infty_c(\R^d \times \R)$, where the generator $\mathscr{L}$ is defined as follows. For each $(t,m,w) \in [0,T] \times C([0,T];\P^p(\R^d)) \times \C^d$, the function $\mathscr{L}_{t, m, w}\varphi : \R^d \times \R_+ \times A \to \R$ is given by
\begin{align}
\begin{split}
\mathscr{L}_{t, m, w}\varphi(x, y, a) & := \nabla_x \varphi(x, y)\cdot b(t, x, m_t,a) + \frac{1}{2}\mathrm{tr}[(\sigma\sigma^\top+\gamma\gamma^\top)\nabla_x^2\varphi(x,y)] \\
    &\qquad +  \frac12y^2  \partial_{yy}\varphi(x, y) \Big|\sigma^{-1}\big[b(t, x, m_t, \beta(t, x, m, w)) - b(t, x, m_t,a) \big] \Big|^2  \\
    &\qquad + y \partial_y \nabla_x \varphi(x, y) \cdot \Big( b(t, x, m_t, \beta(t, x, m, w)) - b(t, x, m_t,a) \Big).
\end{split}    \label{pf:general_diffusion_n_plus_1_dim}
\end{align}

To justify \eqref{pf:SDE_mutilde}, we argue analogously to the proof of Theorem \ref{th:limit-dynamics}.
First, note that the processes $(X^{n,k},\zeta^{n,k})_{k=1}^n$ satisfy the SDE system
\begin{align*}
dX^{n,k}_t &= b(t,X^{n,k}_t,\mu^n_t,\alpha^{n,k}(t,\bm{X}^n,S))dt + \sigma dW^{k}_t + \gamma dB_t, \quad\quad \mu^n_t = \frac{1}{n}\sum_{k=1}^n\delta_{X^{n,k}_t}, \\
d\zeta^{n,k}_t &= \zeta^{n,k}_t\sigma^{-1}\left( b(t,X^{n,k}_t,\mu^n_t,\beta(t,X^{n,k}_t,\mu^n, B)) - b(t,X^{n,k}_t,\mu^n_t,\alpha^{n,k}(t,\bm{X}^n, S)) \right) \cdot dW^{k}_t.
\end{align*}
Apply It\^o's formula to a test function $\varphi \in C^\infty_c(\R^d \times \R)$ to get
\begin{align*}
    d \langle \widetilde{\mu}_t^n, \varphi \rangle & = \frac{1}{n} \sum_{k=1}^n d \varphi(X_t^{n,k}, \zeta_t^{n,k}) \\
    & = \frac{1}{n} \sum_{k=1}^n \mathscr{L}_{t, \mu^n, B}\varphi(X_t^{n,k}, \zeta_t^{n,k}, \alpha^{n,k}(t, \boldsymbol{X}^n, S))dt + \langle \widetilde{\mu}_t^n, \nabla_x \varphi \rangle^\top \gamma  dB_t +  d \widetilde{M}_t^{n,\varphi},
\end{align*}
where $\widetilde{M}^{n,\varphi}$ is the martingale satisfying $\widetilde{M}^{n,\varphi}_0=0$ and
\begin{align*}
    d \widetilde{M}_t^{n,\varphi} &= \frac{1}{n} \sum_{k=1}^n \Big[\nabla_x \varphi(X_t^{n,k}, \zeta_t^{n,k})^\top \sigma + \zeta_t^{n,k}\partial_y \varphi(X_t^{n,k}, \zeta_t^{n,k}) (\Xi_t^{n,k})^\top\Big] dW_t^{k}.
\end{align*}
We may write this equation as
\begin{align}
    \widehat{F}_t^\varphi(\widehat{\boldsymbol{\mu}}^n, B) = \int_0^t\langle \widetilde{\mu}_s^n, \nabla_x \varphi \rangle^\top \gamma  dB_s + \widetilde{M}_t^{n,\varphi}, \quad t\in[0, T], \label{pf:opt:limiteq1}
\end{align}
where we define $\widehat{F}_t^\varphi : \P^p(\widehat\Omega) \times \C^d \to \R$ by
\begin{align*}
\widehat{F}_t^\varphi(\widehat{\bm{m}}, w) = \int_{\mathcal{C}^d \times \C_+^1 \times \mathcal{V}} \Big[\varphi(x_t, y_t) - \varphi(x_0, y_0) - \int_0^t \int_A \mathscr{L}_{s, m, w}\varphi(x_s, y_s,a) q_s(da)ds \Big] \widehat{\bm{m}}(dx, dy, dq),
\end{align*}
with $m=(\widehat{\bm{m}} \circ \pi_t^{-1})_{t \in [0,T]} \in C([0,T];\P^p(\R^d))$.

We will complete the proof of \eqref{pf:SDE_mutilde} by passing to the limit on both sides in \eqref{pf:opt:limiteq1}. To justify this, we first note that Fatou's lemma and Lemma \ref{le:inequalities} yield the moment bounds
\begin{align}
\E \int_{\widehat{\Omega}} \|x\|_T^{p'} \,\widehat{\boldsymbol{\mu}}(dx,dy,dq) &\le \liminf_{n\to\infty} \E \int_{\widehat{\Omega}} \|x\|_T^{p'} \,\widehat{\boldsymbol{\mu}}^n(dx,dy,dq) \le \sup_n \frac{1}{n}\sum_{k=1}^n \E\|X^{n,k}\|_T^{p'} < \infty. \label{pf:opt:hatmu-moments}
\end{align}
The results of \cite{Kurtz-Protter} ensure that the stochastic integral process
 $\int_0^\cdot \langle \widetilde{\mu}_s^n, \nabla_x \varphi \rangle^\top \gamma  dB_s$ converges in law in $\C^1$ to $\int_0^\cdot \langle \widetilde{\mu}_s, \nabla_x \varphi \rangle^\top \gamma  dB_s$, 
since $\widetilde{\mu}^n$ converges in law to $\widetilde{\mu}$ in $C([0,T];\P(\R^d \times \R))$.
In fact, this weak convergence occurs jointly along with $(\widehat{\bm{\mu}}^n,B)$ converging to $(\widehat{\bm{\mu}},B)$.
(See the bullet points before \eqref{def:equation_for_mu_barre} for additional details.)
Next, notice that since $\varphi$ has compact support, the functions $y \partial_y \varphi(x,y)$ and $y^2 \partial_{yy}\varphi(x,y)$ are uniformly bounded. The function $\mathscr{L}_{s, m, w}\varphi$ is continuous, since $\beta$ is, and is bounded in all variables except for $m$. In the case $p=0$, the functional is actually bounded and, since $\mu^n$ converges in law to $\mu$ in $C([0,T];\P(\R^d))$, we deduce that $\widehat{F}_t^\varphi(\widehat{\boldsymbol{\mu}}^n, B)$ converges in law to $\widehat{F}_t^\varphi(\widehat{\boldsymbol{\mu}}, B)$; see Lemma \ref{le:continuity} with $p=0$. Similarly, for the case $0 < p \le p \vee 2 < p'$, thanks to the growth assumption (\ref{assumption:A}.3) and the fact that $\mu^n$ converges in law to $\mu$ in $C([0,T];\P^{p \vee 2}(\R^d))$ by Lemma \ref{le:tightness-main}, we apply Lemma \ref{le:continuity} with $p \vee 2$.

Finally, we argue that $\widetilde{M}_t^{n,\varphi} \to 0$ in probability. This is because $\widetilde{M}^{n,\varphi}$, is an average of orthogonal martingales:
\begin{align*}
\E|\widetilde{M}_t^{n,\varphi}|^2 &= \frac{1}{n^2} \sum_{k=1}^n \E\int_0^t \Big|\nabla_x \varphi(X_s^{n,k}, \zeta_s^{n,k})^\top \sigma + \zeta_s^{n,k}\partial_y \varphi(X_s^{n,k}, \zeta_s^{n,k})( \Xi_s^{n,k})^\top\Big|^2 ds \\
    &\leq \frac{C}{n} \Big(1 + \sup_{n\in\N} \frac{1}{n} \sum_{k=1}^n \E\|\Xi^{n,k}\|_T^2\Big),
\end{align*}
where the constant $C$ depends only on $\varphi$, $T$, and $\sigma$. We deduce from \eqref{pf:opt:xibound1} that  $\E|\widetilde{M}_t^{n,\varphi}|^2 \to 0$.

Passing now to the limit in \eqref{pf:opt:limiteq1}, we deduce that $(\widehat{\bm{\mu}},B)$ satisfies
\begin{align*}
    \widehat{F}_t^\varphi(\widehat{\boldsymbol{\mu}}, B) = \int_0^t\langle \widetilde{\mu}_s, \nabla_x \varphi \rangle^\top \gamma  dB_s,
\end{align*}
a.s., for each $t \in [0,T]$  and $\varphi \in C^\infty_c(\R^d \times \R)$.
Expanding the definition of $\widehat{F}_t^\varphi$ gives \eqref{pf:SDE_mutilde}.

{\ }

\noindent\textbf{Step 4.} In this step, we explain how to ``undo" the change of measure after passing to the limit. Write $(\widehat{\bm{\mu}},B)$ for a subsequential limit point of $(\widehat{\bm{\mu}}^n,B)$, and again denote $\mu_t := \widehat{\bm{\mu}} \circ \pi_t^{-1}$ and $\widetilde{\mu}_t := \widehat{\bm{\mu}} \circ \widetilde{\pi}_t^{-1}$ as in the previous step. 
Recalling \eqref{pf:opt:measurechange1}, we may now deduce for bounded continuous semi-Markov functions $\varphi$ that
\begin{align}
\lim_n \frac{1}{n}\sum_{k=1}^n  \E [ \varphi(t,Y^{n,k,k}_t,\mu^{n,k},B)] &= \lim_n \frac{1}{n}\sum_{k=1}^n  \E[ \zeta^{n,k}_t\varphi(t,X^{n,k}_t,\mu^n,B)] \nonumber \\
	&=  \lim_n \E \int_{\R^d \times \R_+}y\varphi(t,x,\mu^n,B)\,\widetilde{\mu}^n_t(dx,dy) \nonumber \\
	&= \E \int_{\R^d \times \R_+}y\varphi(t,x,\mu,B)\,\widetilde{\mu}_t(dx,dy). \label{pf:opt:measurechange2}
\end{align}
Indeed, to justify the limit \eqref{pf:opt:measurechange2}, simply note that \eqref{pf:uniform_bound_zeta} implies the uniform integrability
\begin{align}
\lim_{r\to\infty}\sup_n \E \int_{\R^d \times \R_+}y 1_{\{y \ge r\}} \,\widetilde{\mu}^n_t(dx,dy) &= \lim_{r\to\infty}\sup_n \frac{1}{n} \sum_{i=1}^n \E[ \zeta^{n.k}_t 1_{\{\zeta^{n.k}_t \ge r\}}] = 0. \label{pf:opt:zeta-unifint}
\end{align}
Define one final stochastic measure flow $\nu=(\nu_t)_{t \in [0,T]}$ by setting, for bounded continuous functions $\varphi : \R^d \to \R$,
\begin{align}
\int_{\R^d} \varphi \,d\nu_t &:= \int_{\R^d \times \R_+} y\varphi(x) \, \widetilde{\mu}_t(dx,dy) = \int_{\widehat\Omega} y_t\varphi(x_t)\,\widehat{\bm \mu}(dx,dy,dq). \label{pf:opt:def:nut}
\end{align}
We will later show that $\nu_t$ is a probability measure, but for now just note that it is indeed a well-defined positive finite measure, since applying \eqref{pf:opt:measurechange2} with $\varphi \equiv 1$ yields
\begin{align}
\E[\nu_t(\R^d)] &= \E\int_{\R^d \times \R_+} y\,\widetilde{\mu}_t(dx,dy) = 1. \label{pf:opt:y-nu-integrable}
\end{align}
Note that $\widetilde\mu_0 = \lambda \times \delta_0$ a.s., since $\zeta^{n,k}_0=1$ for all $n$ and $k$, and thus $\nu_0=\lambda$.

To complete the proof of \eqref{pf:opt:mainlaw}, and thus the whole proposition, we claim that it now suffices to show that we can enlarge the probability space to support an independent Brownian motion $W$ and an independent $\R^d$-valued random variable $\xi \sim \lambda$, such that
\begin{align}
\nu_t = \L(X_t[\beta]\,|\,\F^{\nu,\mu,B}_t), \ \ a.s., \text{ for each } t \in [0,T], \label{pf:opt:nu-identity}
\end{align}
where $X[\beta]$ is the unique strong (i.e., $\FF^{\xi,\nu,\mu,B,W}$-adapted) solution of the SDE
\[
d X_t = b(t, X_t,\mu_t,\beta(t, X_t,\mu, B))dt + \sigma dW_t + \gamma dB_t, \quad X_0 = \xi. 
\]
Indeed, once this is justified, we deduce \eqref{pf:opt:mainlaw} by applying \eqref{pf:opt:measurechange2}, the definition of $\nu_t$, and then \eqref{pf:opt:nu-identity} to get
\begin{align*}
\lim_n \frac{1}{n}\sum_{k=1}^n  \E[ \varphi(t,Y^{n,k,k}_t,\mu^{n,k},B)] &= \E \int_{\R^d \times \R_+}y\varphi(t,x,\mu,B)\,\widetilde{\mu}_t(dx,dy) \\
	&= \E \int_{\R^d } \varphi(t,x,\mu,B)\,\nu_t(dx) \\
	&= \E[\varphi(t,X_t[\beta],\mu,B)].
\end{align*}

To prove \eqref{pf:opt:nu-identity}, we will ultimately  apply the superposition principle, Theorem \ref{th:superposition}. Recalling the notation for the infinitesimal generator $L$ introduced in \eqref{def:generator}, we must show that $\nu$ satisfies the SPDE
\begin{align}
    \langle \nu_t, \varphi \rangle = \langle \lambda, \varphi \rangle +  \int_0^t \int_{\R^d} L_{s, \mu_s} \varphi(x,\beta(s,x,\mu,B)) \,\nu_s(dx) \, ds + \int_0^t \langle \nu_s, \nabla \varphi \rangle^ \top \gamma \,dB_s,
    \label{pf:SPDE_nu}
\end{align} 
for $t \in [0,T]$ and $\varphi\in C_c^\infty(\R^d)$,
along with
\begin{align}
\E\int_0^T\int_{\R^d} | b(t,x,\mu_t,\beta(t,x,\mu,B))|^{p'}  \nu_t(dx)\,dt < \infty. \label{pf:opt:nu-moment}
\end{align}
Indeed, once \eqref{pf:SPDE_nu} and \eqref{pf:opt:nu-moment} are established,  and once we showed that the process $\nu$ takes values in $\P(\R^d)$,  we may deduce \eqref{pf:opt:nu-identity} from Theorem \ref{th:superposition} (for the case $p=p'=0$, the drift $b$ is bounded and we can apply directly Theorem \ref{th:superposition} from \eqref{pf:SPDE_nu}).

{\ }

\noindent\textbf{Step 5.} In this step we complete the proof of the proposition by showing that the measure flow $\nu$ defined in \eqref{pf:opt:def:nut} takes values in $\P(\R^d)$ and satisfies the SPDE \eqref{pf:SPDE_nu} along with \eqref{pf:opt:nu-moment}.
Formally, this will follow by applying the SPDE \eqref{pf:SDE_mutilde} derived for $\widetilde{\mu}$ in Step 3, with test functions of the form $(x,y)\mapsto y\varphi(x)$. To do this rigorously, we must approximate $y$ by a smooth function of compact support.
Consider $u(y) := \exp(1-(1-y^2)^{-1})1_{(-1,1)}(y)$, and let $h_n(y)=yu(y/n)$. The functions  $h_n \in C^\infty_c(\R)$ satisfy the following properties:
\begin{enumerate}[(i)]
\item $h_n(y) \uparrow y$, $h_n'(y) \to 1$, $h_n''(y) \to 0$, as $n\to\infty$, for each $y \ge 0$.
\item $0 \leq h_n(y) \leq y$ for each $y \ge 0$.
\item $\sup_n\sup_{y \ge 0}(|h_n'(y)|+|y h_n''(y)|) < \infty$.
\end{enumerate}
Let $\varphi \in C^\infty_c(\R^d)$, and apply the test function $\psi_n(x,y) := h_n(y)\varphi(x)$ in the SPDE \eqref{pf:SDE_mutilde} to get
\begin{align}
\begin{split}
\langle \widetilde{\mu}_t, \psi_n \rangle = \ & \langle \widetilde{\mu}_0, \psi_n \rangle  + \int_0^t \langle \widetilde{\mu}_s, \nabla_x \psi_n\rangle^\top \gamma  dB_s \\
& + \int_{\widehat{\Omega}} \int_0^t \int_A \mathscr{L}_{s, \mu, B}\psi_n(x_s, y_s,a) q_s(da)ds \widehat{\boldsymbol{\mu}}(dx, dy, dq).
\end{split} \label{pf:opt:testfunctionSPDE}
\end{align}
For the first two terms, use $0 \le h_n(y) \le y$ and with dominated convergence, justified by  \eqref{pf:opt:y-nu-integrable}, to get
\begin{align}
\lim_{n\to\infty} \langle \widetilde{\mu}_t, \psi_n \rangle &= \int_{\R^d\times \R_+} y \varphi(x)\, \widetilde{\mu}_t(dx,dy) = \langle \nu_t,\varphi\rangle, \quad \text{for } t \in [0,T]. \label{pf:opt:testfunctionSPDE-just1}
\end{align}
To handle the stochastic integral term, we cannot simply apply It\^o isometry to show $L^2$ convergence, because we do not yet know if the limiting integrand $\langle \nu_s,\nabla\varphi\rangle$ is square-integrable.
We instead appeal to the dominated convergence theorem  for stochastic integrals (see \cite[Theorem 1.26]{Kallianpur} or \cite[Theorem IV.32]{protter2005stochastic}). The processes $(\langle \widetilde{\mu}_s, \nabla_x \psi_n \rangle)_{s\in[0, T]}$ for each $n \in \N$ and $(\langle \nu_s, \nabla \varphi \rangle)_{s\in[0, T]}$ are continuous.
Moreover, we have for each $n$,
\begin{align*}
| \nabla_x \psi_n(x,y)| &= |h_n(y)\nabla \varphi(x)| \le y\|\nabla\varphi\|_\infty, \quad y \ge 0, \ x \in \R^d.
\end{align*}
Hence, since $\int y\,\widetilde{\mu}_s(dx,dy) < \infty$ a.s.\ by \eqref{pf:opt:y-nu-integrable}, dominated convergence yields
\begin{align*}
\langle \widetilde{\mu}_s, \nabla_x \psi_n \rangle = \int_{\R^d\times\R_+}\!\!h_n(y)\nabla\varphi(x)\,\widetilde{\mu}_s(dx,dy) \to \int_{\R^d\times\R_+}\!\!y\nabla \varphi(x)\,\widetilde{\mu}_s(dx,dy) = \langle \nu_s, \nabla \varphi \rangle,  
\end{align*}
a.s., for each $s$. Note also that $|\langle \widetilde{\mu}_s, \nabla_x \psi_n \rangle| \le \nu_s(\R^d)\|\nabla \varphi\|_\infty$, and $\E\int_0^t\nu_s(\R^d)\,ds=t<\infty$ by  \eqref{pf:opt:y-nu-integrable}. Hence, the dominated convergence theorem for stochastic integrals yields
\begin{align}
\lim_n \int_0^t \langle \widetilde{\mu}_s, \nabla_x \psi_n \rangle^\top \gamma  dB_s =  \int_0^t \langle \nu_s, \nabla \varphi \rangle^\top \gamma  dB_s, \label{pf:opt:testfunctionSPDE-just2}
\end{align}
in probability.
It remains to deal with the final term in \eqref{pf:opt:testfunctionSPDE}. First, plug $\psi_n$ into the generator \eqref{pf:general_diffusion_n_plus_1_dim} to obtain 
\begin{align}
    \mathscr{L}_{t, m, w}\psi_n(x, y, a) & = y  h_n'(y) \nabla \varphi(x) \cdot b(t, x, m_t, \beta(t, x, m, w)) + \frac12 h_n(y) \mathrm{tr}[(\sigma\sigma^\top+\gamma\gamma^\top)\nabla^2 \varphi(x)] \nonumber \\ 
    & + \Big(h_n(y) - y  h_n'(y)\Big) \nabla \varphi(x) \cdot b(t, x, m_t,a) \label{pf:gen11} \\
     & + \frac{1}{2}  y^2  h_n''(y)   \varphi(x) \  \Big|\sigma^{-1}\Big(b(t, x, m_t, \beta(t, x, m, w)) - b(t, x, m_t,a) \Big) \Big|^2. \nonumber 
\end{align}
Using the compact support of $\varphi$, the properties (ii) and (iii) of $(h_n)$ above, and the linear growth assumption (\ref{assumption:A}.3), we get
\begin{align}
\Big|\mathscr{L}_{t, m, w}\psi_n(x, y, a) \Big| \leq C  y \Big(1 +1_{\{p>0\}} \int_{\R^d} |z|^{p\vee 2} \, m_t(dz)\Big). \label{pf:Lbound11}
\end{align}
for a constant $C$ depending only on $\varphi$ and the constant $c_1$ from (\ref{assumption:A}.3). As $n\to\infty$, we have $(h_n(y),h_n'(y),h_n''(y)) \to (y,1,0)$ and thus the last two lines of \eqref{pf:gen11} vanish, yielding
\begin{align*}
\mathscr{L}_{t, m, w}\psi_n(x, y, a) \to yL_{t, m_t} \varphi(x,\beta(t,x,m,w)).
\end{align*}
where the generator $L$ was defined in \eqref{def:generator}. Note that the right-hand side does not depend on $a$.
Since $p' > p \vee 2$, the moment bound \eqref{pf:opt:hatmu-moments} along with \eqref{pf:Lbound11} provide the uniform integrability needed to complete the justification of taking limits in the last term in \eqref{pf:opt:testfunctionSPDE} (and the bounded case $p'=0$ is evident).  Recalling \eqref{pf:opt:testfunctionSPDE-just1} and \eqref{pf:opt:testfunctionSPDE-just2}, taking limits in \eqref{pf:opt:testfunctionSPDE} gives
\begin{align*}
\int_{\R^d \times \R_+} y\varphi(x) \, \widetilde{\mu}_t(dx,dy) = \ & \int_{\R^d \times \R_+} y\varphi(x) \, \widetilde{\mu}_0(dx,dy) + \int_0^t \left(\int_{\R^d \times \R_+} y\nabla\varphi(x)^\top \, \widetilde{\mu}_s(dx,dy) \right) \gamma  dB_s \\
& + \int_{\widehat{\Omega}} \int_0^t \int_A y_sL_{s, \mu_s}\varphi(x_s,\beta(s,x_s,\mu,B))\, q_s(da)\,ds\, \widehat{\boldsymbol{\mu}}(dx, dy, dq).
\end{align*}
Note that there is no dependence on $a$ or $q$ in the innermost integrand in the final term. Using Fubini's theorem, and recalling that $\widetilde\mu_s$ denotes the $(x_s,y_s)$-marginal under $\widehat{\bm\mu}(dx,dy,dq)$ as defined in \eqref{pf:def:tildemu}, the last term can be rewritten as
\begin{align*}
\int_0^t \int_{\R^d \times \R_+}yL_{s, \mu_s}\varphi(x,\beta(s,x,\mu,B))\, \widetilde{\mu}_s(dx,dy)\,ds.
\end{align*}
Recalling the definition of $\nu$, we arrive at the SPDE \eqref{pf:SPDE_nu}.
The integrability requirement \eqref{pf:opt:nu-moment} follows from the growth assumption (\ref{assumption:A}.3) and the fact that, by Fatou's lemma,
\begin{align*}
\E\int_0^T\int_{\R^d} |x|^{p'}\nu_t(dx)dt &= \E\int_{\widehat{\Omega}}\int_0^T y_t|x_t|^{p'}\,dt\, \widehat{\bm{\mu}}(dx,dy,dq) \\
	&\le \liminf_{n\to\infty} \frac{1}{n}\sum_{k=1}^n \E\int_0^T\zeta^{n,k}_t |X^{n,k}_t|^{p'}\,dt \\
	&= \liminf_{n\to\infty} \frac{1}{n}\sum_{k=1}^n \E\int_0^T |Y^{n,k,k}_t|^{p'}\,dt,
\end{align*}
for each $t \in [0,T]$.
The last equality above follows from the change of measure argument of Step 1, and the right-hand side is finite by Lemma \ref{le:inequalities}.  The final step in the proof is to show that $\nu_t(\R^d) = 1$ a.s.\ for all $t\in [0, T]$. Formally, since $\nu_0 = \lambda$, plugging $\varphi\equiv 1$ in \eqref{pf:SPDE_nu} would imply that $d\nu_t(\R^d)=0$ and thus $\nu_t(\R^d) = \lambda(\R^d) = 1$. Because we may only use \eqref{pf:SPDE_nu} with smooth functions of compact support, we justify this using the test functions $\varphi_n(x) = u(|x|/n)$,  sending $n\to \infty$ with similar arguments to those above.
\hfill\qedsymbol

\section{Constructing n-player equilibria from mean field equilibria}\label{se:constructing_nash_eq}

In this section, we give the construction of approximate Nash equilibria for the $n$-player game from a weak MFE, proving Theorem \ref{th:converselimit-strongMFE}.
 The proof follows the same general strategy as that of \cite[Theorem 3.10]{Lacker_closedloop}, but there are additional technical details due to the common noise.

\subsection*{Proof of Theorem \ref{th:converselimit-strongMFE}}

Let $(\Omega,\F,\FF,\PP ,W , B, \alpha^*,\mu, X^*)$ be a weak MFE, in the sense of Definition \ref{def:weakMFE}. We consider the signal process $S = (\mu, B)$, which takes values in $\S = \P(\R^d) \times \R^d$. We will write $s = (s^1, s^2)$ for a generic element $s \in C([0,T];\mathcal{S})$. We will also denote by $s^1_t$ and $s^2_t$ the respective values at time $t$. Recall that $X^*$ satisfies
\begin{align}
dX^*_t = b(t,X^*_t,\mu_t,\alpha^*(t,X^*_t, \mu, B))dt + \sigma dW_t +\gamma dB_t, \quad X^*_0 \sim \lambda.  \label{pf:strongMFElimit0}
\end{align}
Define $\alpha^{n,i} \in \A_n(S)$ for each $n \ge i \ge 1$ by setting:
\begin{align*}
\alpha^{n,i}(t,\bm{x}, s) = \alpha^*(t,x^i_t, s)
\end{align*}
for $t \in [0,T]$, $\bm{x}=(x^1,\ldots,x^n) \in (\C^d)^n$ and $s \in C([0,T];\mathcal{S})$. Define
\begin{align*}
\epsilon_n := \sup_{\beta \in \A_n(S)}J^n_1(\beta,\alpha^{n,2},\ldots,\alpha^{n,n}) - J^n_1(\alpha^{n,1},\ldots,\alpha^{n,n}).
\end{align*}
Clearly $\epsilon_n \ge 0$. By symmetry, $\bm{\alpha}^n=(\alpha^{n,1},\ldots,\alpha^{n,n})$ is an $\epsilon_n$-Nash equilibrium.

Enlarging the probability space $\Omega$, we may construct independent Brownian motions $(W^i)_{i \in \N}$ and $\R^d$-valued random variables $\xi^i$ with law $\lambda$.
We may then let $\bm{X}^n=(X^{n,1},\ldots,X^{n,n})$ denote the associated state process, given as the unique solution of
\begin{align}
\begin{split}
dX^{n,i}_t &= b(t,X^{n,i}_t,\mu^n_t,\alpha^*(t,X^{n,i}_t, \mu, B))dt + \sigma dW^i_t + \gamma dB_t, \quad X^{n,i}_0=\xi^i, \\
\mu^n_t &:= \frac{1}{n}\sum_{k=1}^n\delta_{X^{n,k}_t}.
\end{split}
\label{pf:strongMFE_nplayer}
\end{align}
Because this SDE is Markovian in $\bm{X}^n$, the solution is strong by Lemma \ref{ap:le:SDE-uniq-strong}. It remains to show that $\epsilon_n \to 0$ and that $\L(\mu^n, B)\to \L(\mu, B)$ in $\P(C([0, T]; \P(\R^d) \times \R^d))$. 

{\ }

\noindent\textbf{Step 1.} We first claim that $\mu^n$ converges in probability to  $\mu$ in $\CP$, and also
\begin{align}
\langle \mu^n_t , \varphi(S,\cdot)\rangle \to \langle \mu_t, \varphi(S,\cdot)\rangle, \quad \text{in probability} \label{pf:converse:strongprop}
\end{align}
for each $t \in [0,T]$ and each bounded measurable (not necessarily continuous) function $\varphi : C([0,T];\S) \times \R^d \to \R$.
We wish to apply Theorem \ref{ap:th:strong-propagation-of-chaos}, which gives propagation of chaos for McKean-Vlasov equations under minimal continuity assumptions in $x$, which is convenient here because we do not have continuity of $\alpha^*$. We cannot immediately apply Theorem \ref{ap:th:strong-propagation-of-chaos} due to the presence of common noise, but a well known change of variables offers a workaround: Define $\widetilde{X}_t^*:=X_t^*-\gamma B_t$ and $\widetilde{X}^{n,i}_t := X^{n,i}_t-\gamma B_t$. 
For $m \in \P(\R^d)$ and $w \in \R^d$, write $m(\cdot - w)$ for the image of $m$ under the map $x \mapsto x + w$.
Let $\widetilde\mu_t := \mu_t(\cdot + \gamma B_t)$.
For $s \in C([0,T];\S)$ define $\widetilde{b}_s : [0,T] \times \R^d \times \P(\R^d) \to \R$ by
\begin{align*}
\widetilde{b}_s(t,x,m) := b(t,x + \gamma s^2_t,m(\cdot-\gamma s^2_t),\alpha^*(t,x_t+\gamma s^2_t,s)).
\end{align*}
We then find that $\widetilde{X}^*$ and $\widetilde{X}^{n,i}$  satisfy
\begin{align*}
d\widetilde{X}^*_t = \widetilde{b}_S(t,\widetilde{X}^*_t ,\widetilde\mu_t))dt + \sigma dW_t, \quad \widetilde{X}^*_0 = X^*_0,
\end{align*}
as well as
\begin{align*}
d\widetilde{X}^{n,i}_t &= \widetilde{b}_{S}(t,\widetilde{X}^{n,i}_t,\widetilde\mu^n_t)dt + \sigma dW^i_t, \quad \ \  \widetilde\mu^n_t = \frac{1}{n}\sum_{k=1}^n\delta_{\widetilde{X}^{n,k}_t}.
\end{align*}
Note that we have
\begin{align*}
\E\left[\varphi(\widetilde{X}^*_t) \,|\, \F^S_t\right] &= \E\left[\varphi(X^*_t - \gamma B_t) \,|\, \F^{\mu, B}_t\right] = \langle \mu_t, \varphi(\cdot - \gamma B_t)\rangle = \langle \widetilde\mu_t, \varphi\rangle,
\end{align*}
a.s., for each $t \in [0,T]$ and bounded continuous $\varphi : \R^d \to \R$. That is, $\widetilde\mu_t = \L(\widetilde{X}^*_t\,|\,\F^S_t)$.
Note that $\widetilde\mu$ is $S$-measurable, so we may write $\widetilde\mu=\widehat\mu(S)$ a.s.\ for some measurable function $\widehat\mu : C([0,T];\S) \to \CP$.
Since $X^*_0$, $S$, and $W$ are independent, we find that conditionally on $S$ it holds a.s.\ that the then-non-random measure flow $\widetilde\mu_t$ satisfies an ordinary McKean-Vlasov equation, without common noise. We may thus safely condition on the signal $S=(\mu,B)$. With the signal frozen, we may then apply Theorem \ref{ap:th:strong-propagation-of-chaos} to get a strong form propagation of chaos. Namely, we get the weak convergence $\L(\widetilde\mu^n\,|\,S=s) \to \delta_{\widehat\mu(s)}$ in $\P(\CP)$, for (a.e.) $s \in C([0,T];\S)$ by Theorem \ref{ap:th:strong-propagation-of-chaos}. This easily implies $\widetilde\mu^n \to \widehat\mu(S) = \widetilde\mu$ in probability. Changing variables, this implies $\mu^n \to \mu$ in probability. Moreover, the last claim of Theorem \ref{ap:th:strong-propagation-of-chaos} yields the conditional convergence in probability
\begin{align*}
\PP\left(|\langle \widetilde\mu^n_t - \widehat\mu_t(s), \varphi(s,\cdot)\rangle| \ge \epsilon \,|\, S = s\right) \to 0,
\end{align*}
for each $\epsilon > 0$, $t \in [0,T]$, and bounded measurable function $\varphi : C([0,T];\S) \times \R^d \to \R$.
Taking expectations and recalling $\widetilde\mu=\widehat\mu(S)$ yields
\begin{align}
\PP\left(|\langle \widetilde\mu^n_t - \widetilde\mu_t, \varphi(S,\cdot)\rangle| \ge \epsilon \right) \to 0,
\end{align}
again for all $\epsilon$ and $t$ and all such $\varphi$.
Reversing the change of variables leads to \eqref{pf:converse:strongprop}.

{\ }

\noindent\textbf{Step 2.} We next claim that
\begin{align}
\lim_{n\to\infty} \E \int_{\R^d} f(t,x,\mu^n_t,\alpha^*(t,x, \mu, B)) \, \mu^n_t(dx)  = \E \int_{\R^d} f(t,x,\mu_t,\alpha^*(t,x, \mu, B)) \, \mu_t(dx), \label{pf:converse:flimit1}
\end{align}
for each $t \in [0,T]$. Note that this is not immediate from the weak convergence $\mu^n \to \mu$ because  $\alpha^*$ may be discontinuous, nor is it immediate from Step 1 because $f$ depends on $\mu^n_t$.
First, let $r > 0$, and let $B_r$ denote the centered open ball in $\R^d$ of radius $r$. Consider the function
\begin{align*}
F_r(m,m') := \sup_{a \in A, \, x \in B_r}\left|f(t,x,m,a )- f(t,x,m',a)\right|.
\end{align*}
By joint continuity of $f(t,\cdot)$, using the compactness of $A$ and the closure of $B_r$, it holds that $F(m,m') \to 0$ as $m' \to m$ weakly, for each $m$.
Thus
\begin{align*}
&\left|\E \int_{\R^d} \big(f(t,x,\mu^n_t,\alpha^*(t,x, \mu, B)) - f(t,x,\mu_t,\alpha^*(t,x, \mu, B))\big) \, \mu^n_t(dx) \right| \\
&\quad \le \E F_r(\mu_t,\mu^n_t) + 2\|f\|_\infty\E\mu^n_t(B_r^c).
\end{align*}
This tends to zero by sending $n\to\infty$ and then $r\to\infty$. Indeed, since $\mu^n_t\to\mu_t$ in probability, the Portmanteau theorem yields $\limsup_{n\to\infty}\E\mu^n_t(B_r^c) \le \E\mu_t(B_r^c)$. To prove \eqref{pf:converse:flimit1}, it now suffices to show that
\begin{align*}
\lim_{n\to\infty} \E \int_{\R^d} f(t,x,\mu_t,\alpha^*(t,x, \mu, B)) \, \mu^n_t(dx)  = \E \int_{\R^d} f(t,x,\mu_t,\alpha^*(t,x, \mu, B)) \, \mu_t(dx).
\end{align*}
which follows directly from \eqref{pf:converse:strongprop}.

{\ }

\noindent\textbf{Step 3.} It remains to show that $\epsilon_n \to 0$. We do this in two steps. First, we claim that 
\begin{align}
\lim_{n\to\infty}J^n_1(\bm\alpha^n) = \E\left[\int_0^Tf(t,X^*_t,\mu_t,\alpha^*(t,X^*_t,  \mu, B))dt + g(X^*_T,\mu_T)\right]. \label{pf:converse11}
\end{align}
To see this, use symmetry to express $J^n_1(\bm{\alpha}^n)$ in terms of the empirical measure $\mu^n$, and then use the result \eqref{pf:converse:flimit1} of Step 2 to take limits:
\begin{align*}
\lim_{n\rightarrow\infty}J^n_1(\bm{\alpha}^n) &= \lim_{n\rightarrow\infty}\E\left[\int_0^T f(t,X^{n,1}_t,\mu^n_t,\alpha^*(t,X^{n,1}_t, \mu, B))dt + g(X^{n,1}_T,\mu^n_T)\right] \\
	&= \lim_{n\rightarrow\infty}\E \Bigg[\int_0^T\int_{\R^d} f(t,x,\mu^n_t,\alpha^*(t,x, \mu, B)) \, \mu^n_t(dx) \, dt + \int_{\R^d}g(x,\mu^n_T) \, \mu^n_T(dx) \Bigg] \\
	&= \E \Bigg[\int_0^T\int_{\R^d} f(t,x,\mu_t,\alpha^*(t,x, \mu, B)) \, \mu_t(dx) \, dt + \int_{\R^d}g(x,\mu_T) \, \mu_T(dx) \Bigg].
\end{align*}
This equals the right-hand side of \eqref{pf:converse11}, by the consistency condition $\mu_t = \L(X^*_t\,|\,\mu, B)$ and Fubini's theorem.

{\ }

\noindent\textbf{Step 4.} To complete the proof that  $\epsilon_n \to 0$, the remaining step is to choose for each $n$ an arbitrary $\beta^{n} \in \A_n(S)$ such that 
\begin{align*}
J^n_1(\beta^n, \alpha^{n,2},\ldots, \alpha^{n,n}) \ge \sup_{\beta \in \A_n(S)}J^n_1(\beta,\alpha^{n,2},\ldots, \alpha^{n,n}) -  \frac{1}{n},
\end{align*}
and then argue that (recalling $S=(\mu,B)$, and favoring this shorter notation henceforth)
\begin{align}
\limsup_{n\rightarrow\infty} J^n_1(\beta^n,\alpha^{n,2},\ldots,\alpha^{n,n}) \le \E\left[\int_0^T f(t,X^*_t,\mu_t,\alpha^*(t,X^*_t,S))dt + g(X^*_T,\mu_T)\right]. \label{pf:converse:step3}
\end{align}
This will be accomplished by showing essentially that, along any convergent subsequence, the limit of the left-hand side equals the mean field value achieved by some deviating control, so that \eqref{pf:converse:step3} will follow from the MFE optimality condition (6) of Definition \ref{def:weakMFE}.

Introduce $\bm{Y}^n = (Y^{n,1}, \ldots,  Y^{n,n})$, the state processes associated with the controls $(\beta, \alpha^{n,2}, \ldots, \alpha^{n,n})$, governed by the SDEs
\begin{align}
\begin{split}
d Y_t^{n,1} & = b(t, Y_t^{n,1}, \nu_t^n, \beta^n(t, \bm{Y}^n, S))\,dt + \sigma d W_t^1 + \gamma d B_t,  \\
d Y_t^{n,k} & = b(t, Y_t^{n,k}, \nu_t^n, \alpha^*(t, Y^{n,k}_t, S))\,dt + \sigma d W_t^k + \gamma d B_t, \quad k \geq 2,
\end{split} \label{eq:deviating_player_cv_result}
\end{align}
with $\bm{Y}^n_0 = \bm{X}^{n}_0$ and $\nu_t^n = \frac{1}{n} \sum_{k=1}^n \delta_{Y^{n,k}_t}$. 
Because these SDEs are not Markovian, for each $n$ we must enlarge the probability space $(\Omega,\F,\FF,\PP)$ in order to construct them (as in Definition \ref{def:nplayerSDE}), but we abuse notation by keeping the same notation for the probability space.

{\ }

\noindent\textbf{Step 4a.}
We first show that $\nu^n \to \mu$ in probability. To do so, we relate $\bm{X}^n$ and $\bm{Y}^n$ through a change of measure similar to Step 1 of the proof of Proposition \ref{pr:1playerdeviation}.  Define $\QQ^n \ll \PP$ by
\begin{align*}
\frac{d\QQ^n}{d\PP} & = \exp\left( \int_0^T \Xi^{n}_s  \cdot dW^1_s - \frac12 \int_0^T | \Xi^{n}_s  |^2\,ds  \right), \\
\Xi^n_t &:= \sigma^{-1}\left[b(t,X^{n,1}_t,\mu^n_t,\beta^n(t,\bm{X}^{n}, S)) - b(t,X^{n,1}_t,\mu^n_t,\alpha^*(t,X^{n,1}_t, S)) \right].
\end{align*}
Since $(\Xi_t^n)_{t\in [0, T]}$ is bounded, we can apply Girsanov's theorem and conclude that the processes $\bm{\widetilde W}^n = (\widetilde{W}^{1}, W^2, \ldots , W^n)$ are independent Brownian motions under $\QQ^n$ where
\begin{align*}
\widetilde{W}^{1}_t := W^1_t - \int_0^t \Xi^{n}_s\,ds .
\end{align*}
Note that $\bm{X}^n$ satisfies the same SDE system that $\bm{Y}^n$ solves under $\QQ^n$.
Moreover, we claim that $\bm{\widetilde W}^n$ is in fact a $\FF^{\bm{X}^n,S}$-Brownian motion under the conditional measure $\QQ^{n}(\cdot\,|\,S)$, a.s. Indeed, by Definition \ref{def:nplayerSDE}, $\bm{W}^n=(W^1,\ldots,W^n)$ is  an $\FF^{\bm{X}^n,S}$-Brownian motion under the conditional measure $\PP(\cdot\,|\,S)$, a.s. This implies that $\E[ d\QQ^n/d\PP\,|\,S]=1$ a.s., and thus 
\begin{align*}
\frac{d\QQ^n(\cdot \,|\, S) }{d\PP(\cdot \,|\, S)} = \frac{d\QQ^n}{d\PP}, \ \ a.s.
\end{align*}
Applying Girsanov's theorem to the conditional measure yields the claim that $\bm{\widetilde W}^n$ is $\FF^{\bm{X}^n,S}$-Brownian motion under  $\QQ^{n}(\cdot\,|\,S)$, a.s. The uniqueness in Lemma \ref{le:nplayerSDE-lemma} then implies that $\QQ^{n} \circ (\bm{X}^n,S)^{-1} = \PP \circ (\bm{Y}^{n},S)^{-1}$.  

Finally, recall that $\mu^n \to \mu$ in probability under $\PP$ by Step 1, and note that boundedness of $b$ easily implies $\sup_n\E[|d\QQ^n/d\PP|^q] < \infty$ for any $q > 1$.
Thus, for any $\epsilon > 0$ and any compatible metric $d$ on $\CP$, 
\begin{align*}
\PP\big(d(\nu^n, \mu) > \epsilon \big) &= \QQ^n\big( d(\mu^n, \mu)> \epsilon\big) = \E\left[\frac{d\QQ^n}{d\PP} 1_{\{d(\mu^n, \mu)> \epsilon\}}\right]  \to 0,
\end{align*}
and we deduce that $\nu^n\to\mu$ in probability under $\PP$.

{ \ }

\noindent\textbf{Step 4b.}
We next study the sequence $\Upsilon^n := (Y^{n,1}, \beta^n(t, \bm{Y}^n, S), \nu^n, W^1, S=(\mu,B))_{n \in \N}$ of random variables in
\begin{align*}
\widehat\Omega := \C^d \times \V \times \CP \times \C^d \times C([0,T];\S), \quad \S := \P(\R^d) \times \R^d,
\end{align*}
where we identify $\beta^n(t, \bm{Y}^n, S)$ with the relaxed control $dt\delta_{\beta^n(t, \bm{Y}^n, S)}(da)$ (see Section \ref{se:relaxedcontrols}).
The sequence $(\Upsilon^n)_{n\in\N}$ is easily seen to be tight. Indeed, $\V$ is compact, and the marginal of $(W^1,S)$ does not depend on $n$. Moreover, the convergence in probability  $\nu^n\to\mu$ established in Step 4a implies the tightness of $(\nu^n)_{n \in \N}$. Finally, tightness of $(Y^{n,1})_{n \in \N}$ is a straightforward consequence of the boundedness of $b$ and the fact that $Y^{n,1}_0 \sim \lambda$ for each $n$, e.g., by applying Aldous's criterion \cite[Theorem 16.11, Lemma 16.12]{kallenberg-foundations}.
  
Let $\Upsilon = (Y, \Lambda,  \widetilde{\nu}, \widetilde{W}, \widetilde{S}=(\widetilde{\mu},\widetilde{B}))$ any limit point of this tight sequence, and relabel the convergent subsequence. 
It follows from the convergence in probability $\nu^n\to\mu$ in Step 4a that necessarily $\widetilde{\nu}=\widetilde{\mu}$ a.s. Since also the marginal $(W^1,S)$ in $\Upsilon^n$ is the same for each $n$, it follows that $\L(\widetilde{W},\widetilde{S} ) = \L(W,S=(\mu,B))$. We may thus assume without loss of generality that $\Upsilon$ is constructed on the same probability space on which $(\mu,B)$ is defined, and we remove the tildes from the notation $\Upsilon = (Y, \Lambda,  \mu, W, S=(\mu,B))$.
Along this subsequence, we deduce that
\begin{align}
\begin{split}
\lim_n J^n_1(\beta^n,\alpha^{n,2},\ldots,\alpha^{n,n}) & = \lim_n \E\left[\int_0^T f(t, Y_t^{n,1}, \nu_t^n, \beta^{n}(t, \bm{Y}^n, S))\, dt + g (Y_T^{n,1}, \nu_T^n) \right] \\
& = \E\left[\int_0^T\int_A  f(t, Y_t, \widetilde\mu_t, a)\,\Lambda_t(da)\, dt + g (Y_T, \widetilde\mu_T) \right],
\end{split} \label{pf:converse-lim1}
\end{align}
by continuity and boundedness of $f$ and $g$ (and Lemma \ref{le:continuity}).
Moreover, because $W^1$ and $B$ are Brownian motions with respect to the filtration $\FF^{\Upsilon^n}$, it follows easily that the limiting $\widetilde{W}$ and $\widetilde{B}$ are Brownian motions in the filtration $\FF^{\Upsilon}$.
In addition, by using the continuity and boundedness of $b$ and sending $n\to\infty$ in \eqref{eq:deviating_player_cv_result}, we obtain the limiting SDE
\begin{align}
d Y_t = \int_A b(t, Y_t, \widetilde{\mu}_t, a) \,\Lambda_t(da)\,dt + \sigma d\widetilde{W}_t + \gamma d \widetilde{B}_t. \label{pf:converse-SDE-Y}
\end{align}
We finally claim that, by a projection argument, we may construct a process $Z$ and a semi-Markov function $\alpha : [0,T] \times \R^d \times C([0,T];\S) \to A$ such that $Z$ satisfies
\begin{align}
dZ_t = b(t,Z_t,\widetilde{\mu}_t,\alpha(t,Z_t,\widetilde{S}))\,dt + \sigma d\widetilde{W}_t + \gamma d\widetilde{B}_t, \quad Z_0 \sim \lambda, \label{pf:converse-SDE}
\end{align}
as well as
\begin{align}
\begin{split}
\E&\left[\int_0^T\int_A  f(t, Y_t, \widetilde{\mu}_t, a)\,\Lambda_t(da)\, dt + g (Y_T, \widetilde{\mu}_T) \right] \\
&\le \E\left[\int_0^T  f(t, Z_t, \widetilde{\mu}_t, \alpha(t,Z_t,\widetilde{S}))\, dt + g (Z_T, \widetilde{\mu}_T) \right].
\end{split} \label{pf:converse-value}
\end{align}
Once this is justified, we complete the proof as follows: By the mean field optimality condition (6) of Definition \ref{def:weakMFE}, the right-hand side of \eqref{pf:converse-value} is less than or equal to the right-hand side of \eqref{pf:converse:step3}. 
The claimed inequality \eqref{pf:converse:step3} then follows from \eqref{pf:converse-lim1}, along the given convergent subsequence. This argument applies to arbitrary convergent subsequences, and thus \eqref{pf:converse:step3} follows.

{\ }

\noindent\textbf{Step 4c.} It remains to construct  $(Z,\alpha)$ satisfying \eqref{pf:converse-SDE} and \eqref{pf:converse-value} as in the previous paragraph.
The idea, similar to that of Theorem \ref{th:limit-dynamics}, is to ``project away" the extra randomness from $\Lambda_t$ in the SDE \eqref{pf:converse-SDE-Y}, but we will see that the precise implementation is rather involved.
First, apply Ito's formula for $\varphi \in C_c^\infty(\R^d)$ to get
\begin{align}
 \varphi(Y_t) & =  \varphi(Y_0) + \int_0^t\int_A L_{s, \mu_s}\varphi(Y_s, a) \Lambda_s(da) ds + \int_0^t \nabla \varphi(Y_s)^\top \left(\sigma d W_s + \gamma d B_s \right),\label{pf:convSDE1}
\end{align}
where we recall that the generator $L$ is defined in \eqref{def:generator}.
Define $\eta_t := \L(Y_t\,|\,\F^{\mu,B}_T)$.
We would like to take conditional expectations given $\F^{\mu,B}_T=\F^S_T$ on both sides, to obtain an SPDE for $\eta$. To do so, as in \cite[Appendix B]{superposition_theorem} we need to first prove the following compatibility condition:
\begin{align}
\F^{Y, \Lambda}_t \indep \F^{Y_0, W, S}_T \,|\, \F^{Y_0, W, S}_t, \quad \forall t \in [0,T]. \label{eq:compatibility_condition}
\end{align}
We write $\G_1 \indep \G_2 \,|\, \G_3$ to mean that  $\G_1$ and $\G_2$ are conditionally independent given $\G_3$.
To prove \eqref{eq:compatibility_condition}, which is really a property of the law of $\Upsilon=(Y,\Lambda,W,\mu,W,S)$, we will first show the analogous condition for $\Upsilon^n=(Y^{n,1},\beta^n,\nu^n,W^1,S)$ for each $n$, and then we will argue that this property is preserved by weak limits.
Indeed, abbreviating $\beta^n_t = \beta^n(t,\bm{Y}^n,S)$, we first argue that 
\begin{align}
\F^{Y^{n,1}, \beta^n}_t \indep \F^{Y_0^{n,1}, W^1, S}_T \,|\, \F^{Y_0^{n,1}, W^1, S}_t, \quad \forall t \in [0,T], \ n \in \N. \label{eq:compatibility_condition-n}
\end{align}
By the equivalence (1)$\Leftrightarrow$(2) of Lemma \ref{ap:le:technical-compatibility-lemma},  \eqref{eq:compatibility_condition-n} is equivalent to the following two properties holding:
\begin{enumerate}
\item[(a)] $W^1$ is a Brownian motion with respect to the filtration $(\F^S_T \vee \F^{Y^{n,1}, \beta^n,W^1}_t)_{t \in [0,T]}$.
\item[(b)] $\F^{Y^{n,1}, \beta^n,W^1}_t \indep \F^S_T \,|\, \F^S_t$, $\forall t \in [0,T]$.
\end{enumerate}
Since $\sigma$ is non-degenerate and $B$ is $\FF^S$-adapted, we may rearrange the SDE \eqref{eq:deviating_player_cv_result} to find that $W^1$ is adapted to $\FF^{\bm{Y}^n,S}$. Trivially, $\beta^n$ is also adapted to $\FF^{\bm{Y}^n,S}$. 
Hence, (a) and (b) will follow from the stronger properties
\begin{enumerate}
\item[(a')] $W^1$ is a Brownian motion with respect to the filtration $(\F^S_T \vee \F^{\bm{Y}^n,S}_t)_{t \in [0,T]}$.
\item[(b')] $\F^{\bm{Y}^n}_t \indep \F^S_T \,|\, \F^S_t$, $\forall t \in [0,T]$.
\end{enumerate}
Property (a') follows easily from the fact that $W^1$ is a $\FF^{\bm{Y}^n,S}$-Brownian motion under the conditional measure $\PP(\cdot\,|\,S)$, a.s., by Lemma \ref{le:nplayerSDE-lemma}.
By Lemma \ref{ap:le:SDE-uniq}, the unique weak solution of the SDE system \eqref{eq:deviating_player_cv_result} automatically satisfies the property (b').
We have thus shown that \eqref{eq:compatibility_condition-n} holds, and we next send $n\to\infty$ to deduce \eqref{eq:compatibility_condition}. To do so, let $t \in [0,T]$, and let $h_t : \C^d\times \V \to \R$ be bounded, continuous, and time-$t$-measurable in the natural filtration on $\C^d \times \V$. Let $\phi,\psi_t : \R^d \times \C^d \times C([0,T];\mathcal{S}) \to \R$ be bounded and measurable, with $\psi_t$ assumed to be time-$t$-measurable in the natural filtration of $\R^d \times \C^d \times C([0,T];\mathcal{S})$. Note that the law of $(Y^{n,1}_0,W^1,S)$ does not depend on $n$, and thus there exists a bounded measurable function $\widehat{\phi}_t$ such that
\begin{align*}
\widehat{\phi}_t(Y^{n,1}_0,W^1,S) &= \E[\phi(Y^{n,1}_0,W^1,S)\,|\,\F^{Y^{n,1}_0,W^1,S}_t], \ \ a.s., \ \ \forall n \in \N.
\end{align*}
Thus, along the subsequence for which $\Upsilon^n$ converges in law to $\Upsilon$, we have
\begin{align*}
\E[h_t(Y,\Lambda)\phi(Y_0,W,S)\psi_t(Y_0,W,S)] &= \lim_n \E[h_t(Y^{n,1},\beta^n)\phi(Y_0^{n,1},W^1,S)\psi_t(Y_0^{n,1},W^1,S)] \\
	&= \lim_n \E[h_t(Y^{n,1},\beta^n)\widehat{\phi}_t(Y_0^{n,1},W^1,S)\psi_t(Y_0^{n,1},W^1,S)] \\
	&= \E[h_t(Y,\Lambda)\widehat{\phi}_t(Y_0,W,S)\psi_t(Y_0,W,S)] \\
	&= \E[\E[h_t(Y,\Lambda)\,|\,\F^{Y_0,W,S}_t]\widehat{\phi}_t(Y_0,W,S)\psi_t(Y_0,W,S)]
\end{align*}
Indeed, the second identity follows from \eqref{eq:compatibility_condition-n}, and both limits hold despite $(\phi,\widehat{\phi}_t,\psi_t)$ being potentially discontinuous because the marginal law of $(Y_0^{n,1},W^1,S)$ does not depend on $n$; see \cite[Lemma 2.1]{beiglbock2018denseness}. Applying the above with $h_t\equiv 1$, we find
\begin{align*}
\widehat{\phi}_t(Y_0,W,S) &= \E\big[\phi(Y_0,W,S)\,|\,\F^{Y_0,W,S}_t\big], \ \ a.s.
\end{align*}
For general $h_t$, we thus deduce
\begin{align*}
\E[h_t(Y,\Lambda)\phi(Y_0,W,S)\,|\,\F^{Y_0,W,S}_t] =  \E[h_t(Y,\Lambda)\,|\,\F^{Y_0,W,S}_t]\E\big[\phi(Y_0,W,S)\,|\,\F^{Y_0,W,S}_t\big], \ \ a.s.
\end{align*}
This proves \eqref{eq:compatibility_condition}.

Now that we have proven the compatibility condition \eqref{eq:compatibility_condition}, we make use of it as follows.
Recall that $\eta_t := \L(Y_t\,|\,\F^S_T)$ for each $t$. It follows from \eqref{eq:compatibility_condition} and (1)$\Rightarrow$(3) of Lemma \ref{ap:le:technical-compatibility-lemma} that $\F^Y_t \indep \F^S_T\,|\,\F^S_t$ , which implies that $\eta_t = \L(Y_t\,|\,\F^S_t)$.
Moreover, \eqref{eq:compatibility_condition} implies  the identity
\begin{align*}
\E\left[\int_0^t\int_A L_{s, \mu_s}\varphi(Y_s, a) \Lambda_s(da) ds \,\Big|\, \F_T^S \right] = \int_0^t \E\left[ \int_A L_{s, \mu_s}\varphi(Y_s, a) \Lambda_s(da) \,\Big|\, \F_s^S \right] ds,
\end{align*} 
for $0 \le s \le t \le T$,
as well as (since $S=(\mu,B)$ and $W$ are independent)
\begin{align*}
\E\left[\int_0^t \nabla \varphi(Y_s)^\top \left(\sigma d W_s + \gamma d B_s \right)\,\Big|\,\F^S_T\right] = \int_0^t \E[\nabla \varphi(Y_s)^\top \,|\,\F^S_s] \gamma \,dB_s = \int_0^t \langle \eta_s,\,\nabla \varphi \rangle^\top \gamma \,dB_s, 
\end{align*}
via a stochastic Fubini theorem shown in \cite[Lemma B.1]{superposition_theorem}.
For $t \in [0,T]$, we denote by $\widehat{\Lambda}(t,S)$ the conditional expectation  given $\F^S_t$ of the random probability measure $\delta_{Y_t} \times \Lambda_t$ on $\R^d \times A$. That is, $\widehat{\Lambda} : [0,T] \times C([0,T];\S) \to \P(\R^d \times A)$ is a function, which can be taken to be progressively measurable (see \cite[Proposition 5.1]{brunick2013mimicking} or \cite[Lemma C.3]{Lacker_closedloop}), satisfying
\begin{align}
\int_{\R^d \times A} \psi(t, x, \mu_t, a)\widehat{\Lambda}(t,S)(dx,da) = \E\left[ \int_A \psi(t, Y_t, \mu_t, a)\Lambda_t(da) \Big| \F_t^S\right] \label{eq:projecting_away_randomess_converse}
\end{align}
a.s., for each bounded measurable $\psi : [0,T] \times \R^d \times \P(\R^d) \times A\to \R$. Since the $\R^d$-marginal of $\widehat{\Lambda}(t,S)$ equals $\L(Y_t\,|\,\F^S_t)=\eta_t$, we can disintegrate it by
\begin{align}
\widehat{\Lambda}(t,S)(dx,da) = \eta_t(dx)\widehat{\Lambda}'(t,x,S)(da),  \label{eq:desintegration_proj_converse}
\end{align}
for some semi-Markov function $\widehat{\Lambda}' : [0,T] \times \R^d \times C([0,T];\S) \to \P(A)$.
In particular, we find
\begin{align*}
\int_{\R^d}\int_A L_{t,\mu_t}\varphi(x,a)\widehat{\Lambda}'(t, x,S)(da)\eta_t(dx) = \E\left[ \int_{\C^d \times \V} \int_A L_{t, \mu_t}\varphi(Y_t, a)\Lambda_t(da) \,\Big|\, \F_t^S\right].
\end{align*}
Combining the last five equations, we find that taking conditional expectations with respect to $\F^S_T$ in \eqref{pf:convSDE1} yields the following SPDE:
\begin{align}
\langle \eta_t, \varphi \rangle = \langle \lambda, \varphi \rangle + \int_0^t \int_A L_{s, \mu_s} \varphi(x,a) \widehat{\Lambda}'(s, x,S)(da) \,\eta_s(dx)\,ds + \int_0^t \langle \eta_s, \nabla \varphi\rangle ^\top \gamma d B_s, \label{pf:conv-SPDE1}
\end{align}
where we used also the independence of $Y_0$ and $S$ to get $\langle \eta_0,\varphi\rangle = \E[\varphi(Y_0)]=\langle \lambda,\varphi\rangle$ a.s.

We lastly pass from the relaxed control to a strict control. Define semi-Markov functions $(c_1,c_2) : [0,T] \times \R^d \times C([0,T]; \S) \to \R^d \times \R$ by
\begin{align*}
\big(c_1(t,x,S),c_2(t,x,S)\big) &:= \int_A \big(b(t,x,\mu_t,a),f(t,x,\mu_t,a)\big) \widehat{\Lambda}'(t, x,S)(da),
\end{align*}
a.s., for $(t,x) \in [0,T] \times \R^d$, and note that this belongs to the set $K(t,x,\mu_t)$ from Assumption (\ref{assumption:A}.5).
Hence, using a measurable selection argument  \cite[Lemma 3.1]{dufourstockbridge-existence}, we may find a semi-Markov function $\alpha : [0,T] \times \R^d \times C([0,T]; \S) \to A$ such that
\begin{align}
c_1(t,x,S) &= b(t,x,\mu_t,\alpha(t,x,S)), \label{pf:limitdyn-projb1-conv} \\
c_2(t,x,S) &\le f(t,x,\mu_t,\alpha(t,x,S)).  \label{pf:limitdyn-projf1-conv}
\end{align}
Applying \eqref{pf:limitdyn-projb1-conv} in \eqref{pf:conv-SPDE1}, we rewrite \eqref{pf:conv-SPDE1} as
\begin{align}
\langle \eta_t, \varphi \rangle &= \langle \lambda, \varphi \rangle + \int_0^t\int_{\R^d}L_{s,\mu_s}\varphi(x,\alpha(s, x,S))\,\eta_s(dx)\, ds  + \int_0^t \langle \eta_s,\nabla\varphi\rangle^\top \gamma \, dB_s.
\end{align}
We may now apply Theorem \ref{th:superposition} (noting that $\eta$ is $\FF^S$-adapted and that $b$ is bounded) to deduce that the unique (see Lemma \ref{ap:le:SDE-uniq-strong}) strong solution $Z$ of the SDE 
\begin{align*}
d Z_t = b(t, Z_t, \mu_t, \alpha(t, Z_t, S))dt + \sigma d W_t + \gamma d B_t, \quad Z_0 = Y_0
\end{align*}
satisfies $\eta_t = \L(Z_t \,|\, \F_t^S)$ a.s.\ for each $t \in [0,T]$.

Putting it all together, we find
\begin{align*}
\E&\left[\int_0^T\int_A  f(t, Y_t, \mu_t, a)\,\Lambda_t(da)\, dt + g (Y_T, \mu_T) \right] \\
	&= \E\left[\int_0^T\int_{\R^d}\int_A  f(t, Y_t, \mu_t, a)\,\widehat{\Lambda}'(t,x,S)(da)\,\eta_t(dx)\, dt + g (Y_T, \mu_T) \right] \\
	&= \E\left[\int_0^T\int_{\R^d}c_2(t,x,S)\,\eta_t(dx)\, dt + g (Y_T, \mu_T) \right] \\
	&\le \E\left[\int_0^T\int_{\R^d}f(t,x,\mu_t,\alpha(t,x,S))\,\eta_t(dx)\, dt + \langle \eta_T, g (\cdot, \mu_T)\rangle \right] \\
	&= \E\left[\int_0^T f(t,Z_t,\mu_t,\alpha(t,Z_t,S))\, dt + g (Z_T, \mu_T) \right]
\end{align*}
where we used, in order, the identities \eqref{eq:projecting_away_randomess_converse} and \eqref{eq:desintegration_proj_converse} combined with Fubini's theorem and \eqref{eq:desintegration_proj_converse}, the definition of $c_2$, the identity \eqref{pf:limitdyn-projf1-conv} along with the definition of $\eta_T$, and finally  Fubini's theorem along with the identity $\eta_t = \L(Z_t \,|\, \F_t^S)$ a.s.\ for each $t$.  
This finally justifies \eqref{pf:converse-value}, and the proof of the theorem is complete. \hfill \qedsymbol

\section{Connections with other mean field equilibrium concepts} \label{se:connectionsWeakMFE}

In this section we connect our notion of weak semi-Markov MFE with the notion of \emph{weak MFG solution} introduced in \cite{carmona-delarue-lacker}. We paraphrase the definition of the latter here. First, recall the space $\V$ of relaxed controls defined in Section \ref{se:relaxedcontrols}. Let $\X=\C^d \times \V \times \C^d$, equipped with the filtration $\FF^\X=(\F^\X_t)_{t \in [0,T]}$ defined by letting $\F^\X_t$ be generated by the maps $(w,q,x) \mapsto (w_s,q(C),x_s)$ where $s \le t$ and $C$ is a Borel set of $[0,t] \times A$. For $\bm{\widetilde m} \in \P^p(\X)$, define $\widetilde m^x \in C([0,T];\P^p(\R^d))$ by setting $\widetilde m_t^x = \bm{\widetilde m}\circ[(w,q,x)\mapsto x_t]^{-1}$ for $t \in [0,T]$.

\begin{definition} \label{def:weakMFGsolution}
A weak MFG solution is a tuple $(\Omega, \FF, \PP, B, W, \widetilde{\bm\mu}, \Lambda, X)$ where $(\Omega, \FF, \PP)$ is a filtered probability space supporting $(B, W, \bm{\widetilde\mu}, \Lambda, X)$ satisfying: 
\begin{enumerate}[(1)]
    \item The processes $B=(B_t)_{t \in [0,T]}$ and $W=(W_t)_{t \in [0,T]}$ are independent $\FF$-Brownian motions of dimension $d$, and the process $X=(X_t)_{t \in [0,T]}$ is $\FF$-adapted with values in $\R^d$ and with $\PP \circ X_0^{-1} = \lambda$. Moreover, $\widetilde{\bm\mu}$ is a random element of $\P^p(\X)$ such that $\widetilde{\bm\mu}(C)$ is $\F_t$-measurable for each $C \in  \F_t^\X$ and $t \in [0,T]$.
    \item $X_0$, $W$, and $(B, \bm{\widetilde\mu})$ are independent.
    \item The process $\Lambda=(\Lambda_t)_{t\in[0, T]}$ is $\FF$-progressively measurable with values in $\P(A)$. Moreover, $\F^\Lambda_t$ is conditionally independent of $\F_T^{X_0, B, W, \bm{\widetilde\mu}}$ given $\F_t^{X_0, B, W, \bm{\widetilde\mu}}$ for each $t \in [0,T]$. 
    \item The state equation holds:
    \begin{align}
        dX_t = \int_A b(t, X_t, \widetilde\mu_t^x, a)\Lambda_t(da)dt + \sigma dW_t + \gamma dB_t.
        \label{def:state_equation_open_loop}
    \end{align}
    \item  If $(\Omega', \FF', \PP')$ is another filtered probability space supporting $(B', W', \widetilde{\boldsymbol{\mu}}',\Lambda', X')$ satisfying (1--4) and $\PP \circ (X_0, B, W, \widetilde{\boldsymbol{\mu}})^{-1} = \PP' \circ (X_0', B', W', \widetilde{\boldsymbol{\mu}}')^{-1}$, then
    \begin{align*}
        &\E\left[\int_0^T \int_A f(t, X_t, \widetilde{\mu}_t^x, a)\Lambda_t(da)dt + g(X_T, \widetilde{\mu}_T^x) \right] \\
        &\qquad \geq \E\left[\int_0^T  \int_A f(t, X_t', \widetilde{\mu}'_t { }^{x}, a)\Lambda_t'(da)dt + g(X_T', \widetilde{\mu}'_T \!\!{ }^{x}) \right]
    \end{align*}
    \item  $\widetilde{\boldsymbol{\mu}} = \L((W,\Lambda, X)\,|\, B, \widetilde{\bm\mu})$ a.s. 
\end{enumerate}
\end{definition}

 The following theorem shows that our weak semi-Markov MFE in Definition  \ref{def:weakMFE} is equivalent to the weak MFG solutions in a distributional sense.
 
\begin{theorem} \label{th:weaksemiMarkov-to-weakMFGsolution}
Suppose Assumption \ref{assumption:A} holds.
Suppose $(\Omega,\F,\FF,\PP ,W , B, \alpha^*,\mu,X)$ is a weak semi-Markov MFE. Define the $\P(A)$-valued process $\Lambda_t = \delta_{\alpha^*(t,X_t,\mu, B)}$ and $\bm{\widetilde\mu} = \L((W,\Lambda,X) \, | \, \mu, B )$. Then $(\Omega,\FF,\PP,B, W,\bm{\widetilde\mu},\Lambda,X)$ is a weak MFG solution. Conversely, if $(\widetilde \Omega, \widetilde \FF, \widetilde \PP, \widetilde B, \widetilde W, \widetilde{\bm\mu}, \widetilde \Lambda, \widetilde X)$ is a weak MFG solution, then there exists a weak semi-Markov MFE $(\Omega,\F,\FF,\PP , W , B, \alpha^*,\mu, X)$ such that $\L(\mu,B) = \L(\widetilde{\mu}^x,\widetilde B)$. 
\end{theorem}

The proof of this theorem is no different from the case without common noise treated in \cite[Theorem 6.2 \& 6.3]{Lacker_closedloop}. The only changes are that the filtration $\FF^{B,\bm{\widetilde \mu}}$ should replace the one denoted $\FF^{\bm{\widetilde\mu}}$ therein, one should condition also on $B$ whenever one conditions on $\bm{\widetilde \mu}$ or $\widetilde{\mu}^x$, and one should change variables to $Y := \sigma^{-1}\left(X - \gamma B\right)$ to check well-posedness of the relevant SDEs.

Using the correspondence of \ref{th:weaksemiMarkov-to-weakMFGsolution}, we immediately deduce the existence part of Theorem \ref{th:existence} from the corresponding existence theorem for weak MFG solutions \cite[Theorem 3.2]{carmona-delarue-lacker}, and we similarly deduce the uniqueness claim from \cite[Theorem 6.2 and Proposition 4.4]{carmona-delarue-lacker}.
Alternatively, one could prove such a uniqueness result directly, by adapting \cite[Proof of Theorem 2.9]{Lacker_closedloop}.

\begin{appendix}

\section{SDEs with random coefficients} \label{ap:SDE}

This section extends some of the results of \cite[Appendix A]{Lacker_closedloop}, regarding SDEs with random coefficients, to the unbounded and path-dependent case. This  justifies some well-posedness claims for the SDEs arising in the paper.
The arguments given here are somewhat simpler and more self-contained than those of \cite{Lacker_closedloop}.

Throughout the section, let $(E,d_E)$ be a complete separable metric space, equipped with a distinguished point $\circ \in E$. We are mainly interested in the cases $E=\mathcal{S}$, for the $n$-player game, and $E=\R^d \times \P^p(\R^d)$ with $\circ = (0,\delta_0)$, for the MFG. 
Equip the space $\CE$ with the metric
\[
(e,e') \mapsto \sup_{t \in [0,T]}d_E(e_t,e'_t).
\]
Fix $d \in \N$, an initial distribution $\lambda \in \P(\R^d)$ and a progressively measurable function $b : [0,T] \times \C^d \times \CE \rightarrow \R^d$. Assume there exists $C < \infty$ such that, for all $(t,x,e)$,
\begin{align}
|b(t, x, e)| \leq C\bigg(1 + \|x\|_t + \sup_{s \in [0,t]}d_E(e_s,\circ)\bigg), \label{ap:growthassump-path}
\end{align}
where we recall that $\|x\|_t=\sup_{s \in [0,t]}|x_s|$.
The following lemmas spell out the precise nature of existence, uniqueness, and stability for SDEs with random coefficients of the form
\begin{align}
dX_t = b(t,X,\eta)\,dt + dW_t, \qquad X_0 \sim \lambda, \label{ap:eq:SDEform}
\end{align}
where $\eta$, $W$, and $X_0$ are independent.
The results of this section apply to the SDEs encountered in the paper, after a simple change of variables, illustrated in the following remarks.

\begin{remark} \label{re:ap:SDEs-MFG}
Suppose $b$ is as in Assumption \ref{assumption:A}, and consider the SDE
\begin{align*}
dX_t = b(t,X_t,\mu_t,\alpha(t,X_t,\mu,B))dt + \sigma dW_t + \gamma dB_t,
\end{align*}
for some given semi-Markov function $\alpha : [0,T] \times \R^d \times C([0,T];\P^p(\R^d) \times \R^d) \to A$, as encountered in Definition \ref{def:weakMFE}.
Here $\mu$ is a $\P^p(\R^d)$-valued process with $X_0$, $W$, and $(\mu,B)$ independent. Define $Y_t := \sigma^{-1}(X_t-\gamma B_t)$, and for $(t,y,(m,w)) \in [0,T] \times \R^d \times C([0,T];\P^p(\R^d) \times \R^d)$ define
\begin{align*}
\overline{b}(t,y,(m,w)) := \sigma^{-1} b(t,\sigma y + \gamma w_t,m_t,\alpha(t,\sigma y + \gamma w_t,m,w)).
\end{align*}
Then, with $\eta=(\mu,B)$, we have
\begin{align*}
dY_t = \overline{b}(t,Y_t,\eta)dt + dW_t,
\end{align*}
which fits the form \eqref{ap:eq:SDEform} of SDE covered in this section. Because this SDE is Markovian in $X$, we will see in Lemma \ref{ap:le:SDE-uniq-strong} that it has a unique strong solution.
\end{remark}

\begin{remark} \label{re:ap:SDEs-nplayer}
Suppose $b$ is as in Assumption \ref{assumption:A}, and consider the SDE \eqref{def:nplayer-eq}. A similar transformation puts this SDE in the form \eqref{ap:eq:SDEform} covered in this section. Indeed, let $Y^{n,i}_t := \sigma^{-1}(X^{n,i}_t-\gamma B_t)$, and note that $B$ is adapted to $S$, so there exists a Borel function $\phi : C([0,T];\mathcal{S}) \to \C^d$ such that $B=\phi(S)$ a.s. Rewriting the SDE \eqref{def:nplayer-eq} in terms of $\bm Y^n=(Y^{n,1},\ldots,Y^{n,n})$, we see that it takes the form \eqref{ap:eq:SDEform} with $\eta=S$. In general, because the controls in \eqref{def:nplayer-eq} may be path dependent, we will only have weak solutions in this case.
\end{remark}

In some of the following proofs, we will make use of the measure $\W_\lambda \in \P(\C^d)$ defined as the law of a Brownian motion started from initial law $\lambda$.

\begin{lemma}[Deterministic case] \label{ap:le:SDE-determ}
For $e \in \CE$, the SDE
\begin{align}
dX^e_t = b(t,X^e,e)dt + dW_t, \quad\quad X^e_0\sim \lambda  \label{ap:eq:SDE-determ}
\end{align}
has a unique solution in law, and the law is denoted $P^e \in \P(\C^d)$. Moreover, $\CE \ni e \mapsto P^e \in \P(\C^d)$ is Borel measurable, and it is non-anticipative in the following sense: Letting $R_t(x) := x|_{[0,t]}$ denote the restriction map, it holds that $P^e \circ R_t^{-1} = P^{e'} \circ R_t^{-1}$ for every $t \in [0,T]$ and $e,e' \in \CE$ satisfying $e_s=e'_s$ for all $s \in[0,t]$.
\end{lemma}
\begin{proof}
Using \eqref{ap:growthassump-path}, the existence and uniqueness in law is a standard consequence of Girsanov's theorem. See \cite[Theorem 3.5.16]{Karatzas} for existence, and uniqueness is a quick consequence of \cite[Theorem 7.7]{liptser2001statistics}. 
The claimed non-anticipativity follows from the uniqueness and from the assumed progressive measurability of $b$.
The Radon-Nikodym derivative is easily identified as
\begin{align*}
\frac{dP^e}{d\W_\lambda}(w) &= \exp\left(\int_0^T b(t,w,e) \cdot dw_t - \frac12\int_0^T|b(t,w,e)|^2\,dt\right).
\end{align*}
Note that $\int_0^T|b(t,w,e)|^2\,dt$ is Borel measurable in $(w,e)$. By a straightforward monotone class argument, the stochastic integral $\int_0^T b(t,w,e) \cdot dw_t$ can be shown to admit a jointly measurable version, in the sense that there exists a Borel function of $(w,e)$ which agrees $\W_\lambda$-a.s.\ with $\int_0^T b(t,w,e) \cdot dw_t$ for each fixed $e \in \CE$; cf.\ \cite[Theorem 63]{protter2005stochastic}. Hence, we may find a Borel function $F : \C^d \times \CE \to \R$ such that $dP^e/d\W_\lambda =F(\cdot,e)$, $\W_\lambda$-a.s.\ for each $e$. Thus $\langle P^e,h\rangle = \langle \W_\lambda, F(\cdot,e)h\rangle$ is a measurable function of $e$ for each bounded continuous $h$, which implies the claim.
\end{proof}

\begin{lemma}[Existence] \label{ap:le:SDE-exist}
Let $M \in \P(\CE)$. Let $(\eta,X)$ be a $\CE \times \C^d$-valued random variable defined on some probability space $(\Omega,\F,\PP)$, with joint law $M(de)P^e(dx)$ where $P^e$ is the law of the solution of the deterministic SDE \eqref{ap:eq:SDE-determ}.
Define 
\begin{align*}
W_t := X_t - X_0 - \int_0^t b(s,X,\eta)\,ds.
\end{align*}
Then $W$ is an $\FF^{\eta,X}$-Brownian motion. Moreover, $W$ is an $\FF^{\eta,X}$-Brownian motion under the conditional measure $\PP(\cdot\,|\,\eta=e)$, for $M$-a.e.\ $e \in \CE$.
\end{lemma}
\begin{proof}
We begin with the second claim.
Note that $\FF^{W,X} \subset \FF^{\eta,X}$, since $W$ is clearly $\FF^{\eta,X}$-adapted. Note that $\PP(X \in \cdot\,|\,\eta=e)=P^e(\cdot)$ by definition, and thus 
\begin{align*}
W_t=X_t - X_0  - \int_0^t b(s,X,e)\,ds
\end{align*}
is a Brownian motion in the completion of  $\FF^{\eta,X}$ under $\PP( \cdot\,|\,\eta=e)$, for $M$-a.e.\ $e$. 

To prove the first claim, let $0 \le s < t \le T$, and let $g(\eta)$ and $h(X)$ be bounded random variables, measurable with respect to $\F^\eta_s$ and $\F^X_s$, respectively. Then, for bounded measurable $\varphi : \R^d \to \R$,
\begin{align*}
\E[g(\eta)h(X)\varphi(W_t-W_s)] &= \E\big[g(\eta)\E\big[h(X)\varphi(W_t-W_s)\,|\,\eta\big] \big] \\
	&= \E\big[g(\eta)\E\big[h(X)\,|\,\eta\big] \big] \,\E[\varphi(W_t-W_s)] \\
	&= \E[g(\eta)h(X)]\E[\varphi(W_t-W_s)],
\end{align*}
with the second step using the result of the first paragraph of the proof. This shows that $W_t-W_s$ is independent of $\F^{\eta,X}_s$, completing the proof.
\end{proof}

\begin{lemma}[Uniqueness in law] \label{ap:le:SDE-uniq}
Suppose $(\Omega,\F,\FF,\PP)$ is a filtered probability space supporting continuous adapted processes $\eta$, $X$, and $W$, with values in $E$, $\R^d$, and $\R^d$, respectively. Let $M=\L(\eta)$. Assume $W$ is an $\FF$-Brownian motion, and that $\eta$ and $(X_0,W)$ are independent. Suppose also that
\begin{align*}
dX_t = b(t,X,\eta)\, dt + dW_t.
\end{align*}
Suppose $W$ is an $\FF^{\eta,X}$-Brownian motion under $\PP(\cdot\,|\,\eta=e)$, for $M$-a.e.\ $e \in \CE$. Then $\L(\eta,X)=M(de)P^e(dx)$. Moreover, $\F^X_t$ is conditionally independent of $\F^\eta_T$ given $\F^\eta_t$, for each $t \in [0,T]$.
\end{lemma}
\begin{proof}
Under $\PP(\cdot\,|\,\eta=e)$ for $M$-a.e.\ $e \in \CE$, $X$ solves the SDE
\begin{align*}
dX_t = b(t,X,e)\,dt + dW_t, \qquad X_0\sim\lambda,
\end{align*}
with $W$ an $\FF^{\eta,X}$-Brownian motion. By Lemma \ref{ap:le:SDE-determ}, this implies $\PP(X \in \cdot\,|\,\eta=e)=P^e(\cdot)$, proving the first claim. To prove the second, let $t \in [0,T]$, and let $h(X)$ be a bounded $\F^X_t$-measurable random variable. The non-anticipativity of Lemma \ref{ap:le:SDE-determ} implies that $\E[h(X)\,|\,\F^\eta_T]=\E[h(X)\,|\,\eta]=\langle P^\eta,h\rangle$ is $\F^\eta_t$-measurable. Hence, $\E[h(X)\,|\,\F^\eta_T]=\E[h(X)\,|\,\F^\eta_t]$ a.s. This holds for every such $h$, which implies the claimed conditional independence.
\end{proof}

\begin{lemma}[Strong uniqueness]  \label{ap:le:SDE-uniq-strong}
Under the assumptions of Lemma \ref{ap:le:SDE-uniq}, suppose also that
\begin{align}
b(t,x,e)=\widetilde{b}(t,x_t,e), \qquad \forall (t,x,e) \in [0,T] \times \C^d \times \CE, \label{ap:def:Markov}
\end{align}
for some semi-Markov function $\widetilde{b} : [0,T] \times \R^d \times \CE \to \R^d$. Then the process $X$  in Lemma \ref{ap:le:SDE-uniq} is necessarily adapted to the completion of the filtration $\FF^{X_0,W,\eta}=(\sigma(X_0,W_s,\eta_s:s \le t))_{t \in [0,T]}$.
\end{lemma}
\begin{proof}
It follows from a result of Veretennikov \cite{veretennikov1981strong} in the case of bounded $b$ that the SDE \eqref{ap:eq:SDE-determ} is pathwise unique, for each (non-random) $e \in \CE$. For general $b$ satisfying \eqref{ap:growthassump-path} and \eqref{ap:def:Markov}, the same is true by a straightforward localization argument. 

Now, let $P(dx;e,w,x_0)$ denote a version of the conditional law of $X$ given $(\eta,W,X_0)=(e,w,x_0)$.
Let $(\widetilde\eta,\widetilde X)$ be the process constructed in Lemma \ref{ap:le:SDE-exist}, and let $\widetilde W$ be the corresponding Brownian motion therein. Define $\widetilde{P}(d\widetilde{x};e,w,x_0)$ to be a version of the conditional law of $\widetilde{X}$ given $(\eta,W,\widetilde{X}_0)=(e,w,x_0)$.
Consider the space $\overline\Omega= \CE \times \C^d \times \R^d \times \C^d\times \C^d$, equipped with the probability measure
\begin{align*}
\PP(de,dw,dx_0,dx,d\widetilde{x}) := M(de)\W_{\delta_0}(dw)\lambda(dx_0)P(dx;e,w,x_0)\widetilde{P}(d\widetilde{x};e,w,x_0),
\end{align*}
recalling that $\W_{\delta_0}$ is the Wiener measure.
Now, on  some new probability space, construct $(\eta,W,X_0,X,\widetilde{X})$ to have joint law $\PP$, so that $X$ and $\widetilde{X}$ are conditionally independent given $(\eta,W,X_0)$.
Note also that $\L(\eta,\widetilde{X})=M(de)P^e(dx)$ by construction, and 
\begin{align*}
\widetilde{X}_t = X_0 + \int_0^t b(s,\widetilde{X},\eta)\,ds + W_t, \qquad \forall t\in [0,T], \ a.s.
\end{align*}
It is straightforward to check that, under the conditional measure $\PP(\cdot\,|\,\eta=e)$, for $M$-a.e.\ $e$, it holds that $W$ is a Brownian motion in the filtration $\FF^{W,X,\widetilde{X}}$, that $X_0$ has law $\lambda$, and that both $X$ and $\widetilde{X}$ satisfy
\begin{align*}
\widetilde{X}_t &= X_0 + \int_0^t b(s,\widetilde{X},e)\,ds + W_t, \\
X_t &= X_0 + \int_0^t b(s,X,e)\,ds + W_t,\qquad \forall t\in [0,T].
\end{align*}
The pathwise uniqueness argued above implies that $\PP(X=\widetilde{X}\,|\,\eta)=1$, and thus also $\PP(X=\widetilde{X}\,|\,\eta,W,X_0)=1$ a.s. Since $X$ and $\widetilde{X}$ are conditionally independent given $(\eta,W,X_0)$, it follows that there must exist a Borel map $F : \CE \times \C^d \times \R^d \to \C^d$ such that $X=\widetilde{X}=F(\eta,W,X_0)$ a.s. The same is valid for any time $t < T$, and the claim follows.
\end{proof}

\begin{lemma}\label{ap:le:stability_lemma}
Let $\eta$ be a random variable with law $M \in \P(\CE)$.
Consider a sequence of semi-Markov functions $b_n : [0,T] \times \R^d \times \CE \rightarrow \R^d$ such that $\lim_{n\to\infty}b_n(t,x,e)= b(t,x,e)$ for Lebesgue a.e.\ $(t,x) \in [0,T] \times \R^d$ and $M$-a.e.\ $e \in \CE$.
Suppose one of the following holds, for some $C < \infty$:
\begin{itemize}
\item \textit{Bounded case}: $|b| \le C$ and $|b_n| \le C$ for all $n \in \N$.
\item \textit{Unbounded case}: For all $(t,x,e) \in [0,T] \times \R^d \times \CE$ and $n \in \N$,
\[
|b(t,x, e)| + | b_n(t, x, e)| \le C\bigg(1 + |x| + \sup_{s \in [0,t]}d_E(e_s,\circ)\bigg),
\]
and also $\lambda \in \P^2(\R^d)$.
\end{itemize}
 Consider the SDEs
\begin{align}
dX_t & = b(t,X_t,\eta)dt + dW_t, \quad\quad X_0\sim \lambda, \label{def:ap:SDE-stochastic-markovian} \\
dX^n_t & = b_n(t,X_t^n,\eta)dt + dW_t, \quad\quad X^n_0\sim \lambda, \label{def:ap:SDE-stochastic-n}
\end{align}
which are well-posed in the sense of Lemma \ref{ap:le:SDE-exist} when $(X_0,W)$ and $(X_0^n,W)$ are each independent of $\eta$.
Then, for any measurable functions $h_n,h : \CE \times \C^d \to \R$ sharing a common uniform bound, such that $h_n(\eta,X) \to h(\eta,X)$ in probability, we have
\begin{align*}
\lim_{n\rightarrow\infty}\E\,h_n(\eta,X^n) = \E\, h(\eta,X).
\end{align*}
\end{lemma}
\begin{proof}
The bounded case was proven in \cite[Appendix A]{Lacker_closedloop}, so we focus on the unbounded case.
From Lemma \ref{ap:le:SDE-exist}, we know that $\L(\eta, X^n) = M(de)P^{n,e}(dx)$ where $P^{n,e}$ is defined as the law of the unique solution of the SDE
\[
dX^{n,e}_t = b(t, X^{n,e}_t, e)dt + dW_t, \quad X^{n,e}_0 \sim \lambda.
\]
We work mostly on the space $\CE \times \C^d$, with a generic element denoted $(e,w)$, and with the reference probability measure $\PP:=M \times \W_\lambda$, where $\W_\lambda$ again denotes the law of a $d$-dimensional Brownian motion started from initial law $\lambda$.
For each $n \in \N$, the stochastic exponential
\[
Z^n_t(e,w)= \exp\left(\int_0^t b_n(s, w_s, e) \cdot dw_s - \frac{1}{2}\int_0^t |b_n(s, w_s, e)|^2 ds\right)
\]
is well defined $M \times \W_\lambda$-a.e. By Girsanov's theorem (justified in the linear growth case as in \cite[Theorem 3.5.16]{Karatzas} or \cite[Theorem 7.7]{liptser2001statistics}), we have
\begin{align}
\frac{dP^{n,e}}{d\W_\lambda}(w) = Z^n_T(e,w), \label{pf:ap:dpdw1}
\end{align}
and the process $(Z^n_t(e,\cdot))_{t \in [0,T]}$  is a strictly positive martingale under $\W_\lambda$ for every $e$. Define $P^e$ and $Z_T(e,w)$ in exactly the same manner but using $b$ instead of $b_n$, so that $\L(\eta, X) = M(de)P^e(dx)$ and $\frac{dP^{e}}{d\W_\lambda}(w) = Z_T(e,w)$. Abbreviate $Z:=Z_T$ and $Z^n:=Z^n_T$.
The formulas \eqref{pf:ap:dpdw1} and $\L(\eta, X^n) = M(de)P^{n,e}(dx)$, as well as the analogous formulas without the $n$, imply
\begin{align}
Z^n = \frac{d\L(\eta,X^n)}{d\PP}, \qquad Z = \frac{d\L(\eta,X)}{d\PP}. \label{pf:ap-RNform}
\end{align}

\vskip.2cm

\noindent\textbf{Step 1.} 
We first show $Z^n\to Z$ in probabililty under $\PP$. To do so, it suffices to show that $\log Z^n(e,\cdot) \to \log Z(e,\cdot)$ in $L^1(P^e)$ for $M$-a.e.\ $e \in \CE$. 
Recall that $b_n\to b$ for Lebesgue a.e.\ $(t, x) \in [0, T] \times \R^d$ and $M$-a.e.\ $e$. Using the uniform linear growth assumption and the square-integrability of $W$, we deduce from It\^o's isometry and dominated convergence that, for $M$-a.e.\ $e$,
\[
\E \left[\Big|\int_0^T \left(b_n(t, W_t, e) - b(t, W_t, e)\right) \cdot dW_t\Big|^2\right] = \E  \int_0^T \left|b_n(t, W_t, e) - b(t, W_t, e)\right|^2dt  \rightarrow 0,
\]
where $W \sim \W_\lambda$ is a Brownian motion started from law $\lambda$.
The claim easily follows.

\vskip.2cm

\noindent\textbf{Step 2.} 
Next, we show that $\left\{1_{K \times \C^d}  Z^n \ : n\in \N \right\}$ is uniformly integrable under $\PP$, for any bounded set $K \subset \CE$. A well known entropy calculation shows for each $e \in \CE$ that
\begin{align*}
\langle \W_\lambda,\,Z^n(e,\cdot)\log Z^n(e,\cdot)\rangle	&=  \int_{\C^d} \log Z^n(e,w) \,P^{n,e}(dw)  \\
    &= \frac12 \int_{\C^d} \int_0^T |b_n(t, w_t, e)|^2 \,dt \,P^{n,e}(dw)  \\
    &= \frac12\E\left[\int_0^T |b_n(s, X^n_s, \eta)|^2 dt \,\Big|\, \eta=e\right].
\end{align*}
Using the uniform growth assumption and a standard argument involving Gronwall's inequality, we deduce from boundedness of $K$ that
\[
\sup_n\sup_{e \in K}\int_{\C^d}\|w\|_T^2\,P^{n,e}(dw) < \infty.
\]
Hence, 
\begin{align*}
\sup_n \langle \PP, \, 1_{K \times \C^d} Z^n\log Z^n\rangle = \sup_n \frac12\E\left[1_K(\eta)\int_0^T |b_n(s, X^n_s, \eta)|^2 dt \right] < \infty.
\end{align*}
The criterion of de la Vall\'ee Poussin implies the claimed uniform integrability. 

\vskip.2cm

\noindent\textbf{Step 3.} 
To complete the proof, consider functions $h_n$ and $h$ as in the statement of the Lemma. 
Let $K \subset \CE$ be bounded.
Using \eqref{pf:ap-RNform} and $\eta \sim M$, we have
\begin{align*}
\big| \E[ h_n( \eta, X^n)] - \E[h(\eta, X)] \big| & \le 2M(K^c) + \big| \E[ h_n( \eta, X)] - \E[ h(\eta, X)] \big| \\
	&\qquad + \big| \E[1_K(\eta)h_n( \eta, X^n)] - \E[1_K(\eta)h_n(\eta, X)] \big| \\
    & = 2M(K^c) + \big| \E[ h_n(\eta, X)] - \E [h(\eta, X)] \big| \\
    	&\qquad + \big| \langle \PP, 1_{K \times \C^d} (Z^n-Z)h_n\rangle \big|.
\end{align*}
As $n\to\infty$, the second to last term vanishes by assumption, and the last term vanishes because of the results of  Steps 1 and 2. Finally, let $K$ increase to $\CE$ so that $M(K^c)\to 0$.
\end{proof}

\section{Relations between Nash equilibria}
\label{ap:relations_Nash_eq}

This section proves Proposition \ref{pr:np-eq-inclusion0}, relating the different Nash  equilibrium concepts for $n$-player games defined in Section \ref{se:nplayergame}. We first recall a well known projection lemma.

\begin{lemma} \cite[Corollary 3.11]{brunick2013mimicking} \label{le:markovprojection}
Let $(\Omega,\F,\FF,\PP)$ be a filtered probability space supporting an $\FF$-adapted continuous process $X$ and two independent $\FF$-Brownian motion $W$ and $B$, as well as an $\FF$-progressively measurable process $(b_t)_{t \in [0,T]}$ satisfying $\E\int_0^T|b_t|dt < \infty$, such that
\[
X_t = X_0 + \int_0^t b_s\,ds + \sigma W_t + \gamma  B_t, \quad t \in [0,T].
\]
Let $\widetilde{b} : [0,T] \times \C^d \rightarrow \R^d$ be any progressively measurable function satisfying
\begin{align*}
\widetilde{b}(t,X) = \E[b_t\, | \, \F_t^X], \ \ a.s., \ a.e.\ t \in [0,T].
\end{align*}
Then, there exists a filtered probability space $(\widetilde{\Omega},\widetilde{\F},\widetilde{\FF},\widetilde{\PP})$ supporting an $\widetilde{\FF}$-adapted continuous process $\widetilde{X}$ and independent $\widetilde{\FF}$-Brownian motions $\widetilde{W}$ and $\widetilde{B}$,  such that $\widetilde\PP \circ \widetilde{X}^{-1} =  \PP \circ X^{-1}$ and
\[
d\widetilde{X}_t = \widetilde{b}(t,\widetilde{X})dt + \sigma \widetilde{W}_t + \gamma \widetilde{B}_t, \ \ t \in [0,T], \ \  \widetilde{X}_0=X_0.
\]
\end{lemma}

\noindent\textbf{Proof of Proposition \ref{pr:np-eq-inclusion0}.}
Let $\epsilon \ge 0$, and denote $\widetilde{W}^i = (\sigma W^i, \gamma B)$. We omit the proof of (a), as it is identical to the proof of \cite[Proposition 2.2]{Lacker_closedloop}, with every $W^i$ therein replaced by $\widetilde{W}^i$.
The proof of (b) is quite similar as well, but we give the details:
Fix a closed-loop $\epsilon$-Nash equilibrium $\bm{\alpha}=(\alpha^1,\ldots,\alpha^n) \in \A_n^n$ and a signal process $S$. We want to show that $\bm{\alpha}$ is also a S-closed $\epsilon$-Nash equilibrium. Consider the state processes $\bm{X}=(X^1,\ldots,X^n)$ solving the SDE system
\begin{align*}
dX^i_t &= b(t,X^i_t,\mu^n_t,\alpha^i(t,\bm{X}))dt + d\widetilde{W}^i_t, \quad i = 1, \ldots, n, \quad \mu^n_t = \frac{1}{n}\sum_{k=1}^n\delta_{X^k_t}.
\end{align*}
Let $\beta \in \A_n(S)$.
By symmetry, it suffices to just focus on player $1$, showing that 
\begin{align}
J^n_1(\alpha^1,\ldots,\alpha^n) \ge J^n_1(\beta,\alpha^2,\ldots,\alpha^n) - \epsilon. \label{pf:np-inclusion2-3}
\end{align}
Let $\bm{Y}=(Y^1,\ldots,Y^n)$ be the solution of
\begin{align*}
dY^1_t &= b(t,Y^1_t,\nu^n_t,\beta(t,\bm{Y}, S))dt +  d\widetilde{W}^1_t, \\
dY^i_t &= b(t,Y^i_t,\nu^n_t,\alpha^i(t,\bm{Y}))dt +  d\widetilde{W}^i_t, \quad \ i \neq 1, \qquad \nu^n_t = \frac{1}{n}\sum_{k=1}^n\delta_{Y^k_t}.
\end{align*}
There exists progressively measurable functions $c_1: [0, T] \times (\C^d)^n \mapsto \R^d$ and $c_2: [0, T] \times (\C^d)^n \mapsto \R$ such that
\begin{align*}
(c_1(t, \boldsymbol{Y}), c_2(t, \boldsymbol{Y})) = \E\left[\Big(b(t, Y_t^1, \nu_t^n, \beta(t, \boldsymbol{Y}, S)),f(t, Y_t^1, \nu_t^n, \beta(t, \boldsymbol{Y}, S)) \Big)  \Big| \F_t^{\boldsymbol{Y}}\right], \ \ a.s.,
\end{align*}
for each $t \in [0,T]$. For details about constructing jointly measurable versions $(c_1,c_2)$, see, e.g., \cite[Proposition 5.1]{brunick2013mimicking} or \cite[Lemma C.3]{Lacker_closedloop}.
Recalling the definition of the convex set $K(t,x,m)$ from Assumption (\ref{assumption:A}.5), notice that
\begin{align*}
(c_1(t, \boldsymbol{Y}), c_2(t, \boldsymbol{Y})) \in K(t, Y_t^1, \nu_t^n)
\end{align*}
Using a measurable selection argument (see \cite[Theorem A.9]{haussmannlepeltier-existence} or \cite[Lemma 3.1]{dufourstockbridge-existence}), we may find a progressively measurable function $\widehat\alpha : [0, T] \times (\C^d)^n \mapsto A$ such that  
\begin{align*}
\E\left[b(t, Y_t^1, \nu_t^n, \beta(t, \boldsymbol{Y}, S)) \Big| \F_t^{\boldsymbol{Y}} \right] &= c_1(t,\bm Y) = b(t, Y_t^1, \nu_t^n, \widehat\alpha(t,\boldsymbol{Y})), \quad \text{and} \\
\E\left[f(t, Y_t^1, \nu_t^n, \beta(t, \boldsymbol{Y}, S)) \Big| \F_t^{\boldsymbol{Y}} \right] &= c_2(t,\bm Y) \leq f(t, Y_t^1, \nu_t^n, \widehat\alpha(t,\boldsymbol{Y})), \quad a.s., \ \ t\in[0, T].
\end{align*}
Thus, by Lemma \ref{le:markovprojection}, the unique in law weak solution $\bm{\widetilde{Y}}=(\widetilde{Y}^1,\ldots,\widetilde{Y}^n)$ of the SDE system
\begin{align*}
d\widetilde{Y}^1_t &= b(t, \widetilde{Y}^i_t,\widetilde{\nu}^n_t, \widehat\alpha(t,\bm{\widetilde{Y}}))dt + d\widetilde{W}^1_t,  \\
d\widetilde{Y}^i_t &= b(t,\widetilde{Y}^i_t,\widetilde{\nu}^n_t,\alpha^i(t,\bm{\widetilde{Y}}))dt + d\widetilde{W}^i_t, \quad \ i \neq 1, \quad \widetilde{\nu}^n_t = \frac{1}{n}\sum_{k=1}^n\delta_{\widetilde{Y}^k_t},
\end{align*}
satisfies $\L(\bm{\widetilde{Y}}) = \L(\bm{Y})$ and in particular $\widetilde{\boldsymbol{Y}}$ is the state process under the controls $(\widehat\alpha, \alpha^{2}, \cdots, \alpha^n)$. Using Fubini's theorem and the tower property,
\begin{align*}
J^n_1(\beta,\alpha^2,\ldots,\alpha^n) &= \E\left[\int_0^T f(t,Y^1_t,\nu^n_t,\beta(t,\bm{Y}, S))dt + g(Y^1_T,\nu^n_T)\right] \\
	& = \E\left[\int_0^T\E\left[f(t, Y_t^1, \nu_t^n, \beta(t, \boldsymbol{Y}, S)) \Big| \F_t^{\boldsymbol{Y}} \right]dt + g(Y^1_T,\nu^n_T)\right] \\
	& \leq \E\left[\int_0^T f(t,Y^1_t,\nu^n_t,\widehat\alpha(t, \boldsymbol{Y}))dt + g(Y^1_T,\nu^n_T)\right] \\
	& = \E\left[\int_0^T f(t,\widetilde{Y}^1_t,\widetilde{\nu}^n_t,\widehat\alpha(t,\bm{\widetilde{Y}}))dt + g(\widetilde{Y}^1_T,\widetilde{\nu}^n_T)\right] \\
	&= J^n_1(\widehat\alpha,\alpha^2,\ldots,\alpha^n) \\
	&\le J^n_1(\alpha^1,\ldots,\alpha^n)  + \epsilon.
\end{align*}
The last inequality follows from the facts that $\widehat\alpha \in \A_n$ is a closed-loop control and $(\alpha^1,\ldots,\alpha^n)$ is a closed-loop $\epsilon$-Nash equilibrium. \hfill\qedsymbol

\section{A result on strong propagation of chaos}
In this short section, we recall a propagation of chaos result of \cite{lacker2018strong} that will be useful in the proof of Theorem \ref{th:converselimit-strongMFE}. Let $\sigma\in\R^{d\times d}$ be a non-degenerate matrix and $b:[0, T] \times \R^d \times \P(\R^d)\mapsto \R^d$ be a measurable function. Recall that unless stated otherwise the space $\P(\R^d)$ is equipped with the usual weak topology and the corresponding Borel $\sigma$-field. We consider a weak solution $\bm{X}^n = (X^{n,1}, \ldots, X^{n,n})$ to the SDE
\begin{align}
d X_t^{n,i} = b(t, X_t^{n,i}, \mu_t^n)dt + \sigma dW_t^i, \quad \mu^n=\bigg(\mu^n_t = \frac{1}{n} \sum_{k=1}^n \delta_{X^{n,k}_t}\bigg)_{t \in [0,T]},
\label{eq:particle-system-strong-prop-chaos-result}
\end{align}
starting with i.i.d positions $X_0^{n,1}, \ldots, X^{n,n}$, with independent Brownian Motions $W^1, \ldots, W^n$. 
Let $\mu$ be a solution of the associated McKean-Vlasov equation:
\begin{align}
d X_t = b(t,X_t,\mu_t)dt + \sigma d W_t, \quad \mu=(\mu_t = \L(X_t))_{t \in [0,T]}.
\label{eq:McKean-Vlasov-strong-prop-chaos-result}
\end{align}
We have the following theorem, taken from \cite{lacker2018strong}. 
\begin{theorem}\label{ap:th:strong-propagation-of-chaos}
Suppose that $b$ is jointly measurable and bounded. Suppose moreover that there exists $c > 0$ such that, for each $(t,x) \in [0,T] \times \R^d$ and $m,m' \in \P(\R^d)$, we have
\begin{align}
|b(t,x,m) - b(t,x,m')| \le c \|m-m'\|_{\mathrm{TV}}.
\label{eq:Lipschitz-theorem-strong-prop-chaos-result}
\end{align}
Then, there exists a unique in law weak solution $\bm{X}^n$ to the SDE \eqref{eq:particle-system-strong-prop-chaos-result} and a unique in law weak solution $\mu \in C([0,T];\P(\R^d))$ to the McKean-Vlasov equation \eqref{eq:McKean-Vlasov-strong-prop-chaos-result}. Moreover, $\mu^n$ converges in probability to $\mu$, and also $\langle\mu_t^n, \varphi\rangle \to \langle\mu_t, \varphi\rangle$ in probability for each $t \in [0,T]$ and each bounded measurable function $\varphi : \R^d\to\R$.
\end{theorem} 
The claimed existence and uniqueness for \eqref{eq:particle-system-strong-prop-chaos-result} follow from Appendix \ref{ap:SDE}. The existence and uniqueness for \eqref{eq:McKean-Vlasov-strong-prop-chaos-result} is shown in \cite[Theorem 2.4]{lacker2018strong}. The final conclusions come from \cite[Theorem 2.6, Remark 2.8]{lacker2018strong}.

\section{Proof of Superposition Theorem}\label{ap:se:superposition} 
In this section, we provide a proof for Theorem \ref{th:superposition}, which is a simplification of the results of \cite[Theorem 1.3]{superposition_theorem}. We also report a useful Lemma that can be found in that same paper, and that we use several times. 
We use the notation $\G_1 \indep \G_2 \,|\, \G_3$ to mean that  $\G_1$ and $\G_2$ are conditionally independent given $\G_3$. 
\begin{lemma}\cite[Lemma 2.1]{superposition_theorem}\label{ap:le:technical-compatibility-lemma}
Suppose $(\Omega, \FF = (\F_t)_{t\in[0,T]}, \PP)$ is a filtered probability space supporting a $d$-dimensional $\FF$-Brownian motion $W$ as well as two subfiltrations $\GG = (\G_t)_{t\in[0,T]}$ and $\HH = (\H_t)_{t\in[0,T]}$. Assume $W$ is independent of $\G_T$. Then the following two statements are
equivalent:
\begin{enumerate}
\item[(1)] $\H_t \indep \F_T^W \vee \G_T \,|\, \F_t^W \vee \G_t$, for each $t\in[0, T]$.
\item[(2)] The following two conditions hold:
\begin{enumerate}
\item[(2a)] $W$ is a Brownian motion with respect to the filtration $(\G_T \vee \F^W_t \vee \H_t)_{t\in[0,T]}$.
\item[(2b)] $\H_t \vee \F^W_t \indep \G_T \,|\, \G_t$, for all $t \in [0, T]$.
\end{enumerate}
\end{enumerate}
If either one of the conditions (1) or (2) hold, then we have 
\begin{enumerate}
\item[(3)] $\H_t \indep \G_T \,|\, \G_t$, for each $t \in [0, T]$.
\end{enumerate}
\end{lemma}
\noindent\textbf{Proof of Theorem \ref{th:superposition}}
Note that the $\sigma$-field $\F^{\nu,\mu,B}_T$ is countably generated because $(\nu,\mu,B)$ takes values in a Polish space. We may thus apply \cite[Theorem 1.3]{superposition_theorem} to get an extension $(\widetilde\Omega,\widetilde\F,\widetilde\FF,\widetilde\PP)$ of the original filtered probability space $(\Omega,\F,\FF^{\nu,\mu,B},\PP)$ supporting a $d$-dimensional $\widetilde\FF$-Brownian motion $W$ and a continuous $\R^d$-valued $\widetilde\FF$-adapted process $X$, satisfying the following:
\begin{itemize}
\item $B$ is an $\widetilde\FF$-Brownian motion.
\item $X_0$, $W$, and $(B,\mu)$ are independent.
\item We have $X_0 \sim \lambda$, and the SDE holds
\begin{align}
dX_t = b(t,X_t,\mu_t,\alpha(t,X_t,\mu,B))dt + \sigma dW_t + \gamma dB_t, \label{eq:superpositionSDE}
\end{align}
with $\nu_t=\L(X_t\,|\,\F^{\nu,\mu,B}_t)=\L(X_t\,|\,\F^{\nu,\mu,B}_T)$ a.s., for each $t \in [0,T]$.
\item $\widetilde\F_t \indep \F^{\nu,\mu,B,W}_T \,|\, \F^{\nu,\mu,B,W}_t$ for each $t \in [0,T]$.
\end{itemize}
The final condition implies that $\F^X_t \indep \F^{\nu,\mu,B,W}_T \,|\, \F^{\nu,\mu,B,W}_t$ for each $t \in [0,T]$. Applying (1$\Rightarrow$2a) of Lemma \ref{ap:le:technical-compatibility-lemma} with $\HH$ therein taken to be $\FF^X$ and $\GG$ taken to be $\FF^{\nu,\mu,B}$, we find that $W$ is a Brownian motion with respect to the filtration $(\F^{\nu,\mu,B}_T \vee \F^{W,X}_t)_{t \in [0,T]}$. This easily implies that $W$ is a $\FF^{W,X}$-Brownian motion under the conditional measure $\widetilde\PP(\cdot\,|\,\nu,\mu,B)$, a.s. We may then apply Lemmas \ref{ap:le:SDE-uniq} and \ref{ap:le:SDE-uniq-strong} with  $\eta=(\nu,\mu,B)$ (similarly to Remark \ref{re:ap:SDEs-MFG}) to deduce that $X$ is necessarily the unique strong solution of the SDE \eqref{eq:superpositionSDE}, and  in particular $X$ is a.s.\ $\FF^{X_0,\nu,\mu,B,W}$-adapted. We may thus reduce from $\widetilde\FF$ to $\FF^{X_0,\nu,\mu,B,W}$, and the claim follows. \hfill \qed
\end{appendix}

\bibliographystyle{amsplain}
\bibliography{common-noise-bib}

\end{document}